\documentclass[a4paper,12pt]{amsart}

\usepackage[utf8]{inputenc}
\usepackage[T1]{fontenc}
\usepackage{lmodern}
\usepackage{amssymb}
\usepackage[pagebackref]{hyperref}
\usepackage{bbm}
\usepackage{mathtools}
\usepackage{enumerate}
\usepackage{ulem}
\usepackage[left=2cm,right=2cm]{geometry}
\usepackage{newunicodechar}
\allowdisplaybreaks
\newtheorem{theorem}{Theorem}[section]
\newtheorem{lemma}[theorem]{Lemma}
\newtheorem{cor}[theorem]{Corollary} % added
\newtheorem{prop}[theorem]{Proposition} % added
\theoremstyle{definition}
\newtheorem{defi}[theorem]{Definition}

\theoremstyle{remark}
\newtheorem{remark}[theorem]{Remark}

\numberwithin{equation}{section}

\def\ii{{\bf i}}

\newcommand{\NN}{\mathbf{N}}
\newcommand{\eeta}{\widehat{\eta}}
\newcommand{\ind}[1]{\mathbf{1}_{\{#1\}}}
\newcommand{\Ind}[1]{\mathbf{1}_{#1}}

\newcommand{\E}{\mathbb{E}} % expected value
\newcommand{\p}{\mathbb{P}} % probability
\newcommand{\R}{\mathbb{R}} % real numbers

\newcommand{\N}{\mathbb{N}}
\newcommand{\XX}{\mathbb{X}}
\newcommand{\YY}{\mathbb{Y}}
\newcommand{\Leb}{\mathcal{L}}
\newcommand{\PP}{\mathcal{P}}
\newcommand{\loglog}{\log\log}
\DeclareMathOperator{\Var}{Var}	% variance
\DeclareMathOperator{\Cov}{Cov}
\DeclareMathOperator*{\esssup}{esssup}
\DeclareMathOperator{\Vol}{Vol}
\title[Exponential inequalities and LIL for multiple Poisson integrals]{Exponential inequalities and laws of the iterated logarithm for multiple Poisson--Wiener integrals and Poisson $U$-statistics}

\author{Rados{\l}aw Adamczak}
\address{Institute of Mathematics, University of Warsaw, ul. Banacha 2, 02-097 Warszawa, Poland}
\email{r.adamczak@mimuw.edu.pl}
\author{Dominik Kutek}
\address{Institute of Mathematics, University of Warsaw, ul. Banacha 2, 02-097 Warszawa, Poland}
\email{d.kutek@mimuw.edu.pl}

\date{}

\begin{document}
\maketitle
\begin{abstract}
We prove tail and moment inequalities for multiple stochastic integrals on the Poisson space and for Poisson $U$-statistics.  We use them to demonstrate the Law of the Iterated Logarithm for these processes when the intensity of the Poisson process tends to infinity, with normalization depending on the degree of the multiple stochastic integral or degeneracy of the kernel defining the $U$-statistic.

We apply our results to several classical functionals of Poisson point processes, obtaining improvements or complements of known concentration of measure results as well as new laws of the iterated logarithm. Examples include subgraph counts and power length functionals of geometric random graphs, intersections of Poisson $k$-flats, quadratic functionals of the Ornstein--Uhlenbeck L\'evy process and $U$-statistics of marked processes.
\end{abstract}

\section{Introduction}
\subsection*{Motivation and informal description of main results}

Chaos expansions, going back to the work of Wiener \cite{MR1507356} and It\^{o} \cite{MR0077017}, constitute an important ingredient of infinite-dimensional stochastic analysis. They play a role in the definition of Skorokhod integrals and in Malliavin calculus \cite{MR2200233,MR1474726,MR2791919}. They also underlie the recent theory of Malliavin--Stein method for proving weak convergence \cite{MR2962301,MR2641769,MR3336636,MR3379337,MR3585401,MR3687774}, which grew out of the famous fourth moment theorem by Nualart and Peccati \cite{MR2118863}. In this context they have been investigated not only in the original Gaussian setting, but also in other situations, most importantly on the discrete cube and in the Poisson space. In the former case they are known as Walsh expansions and have been studied for many years, mostly from the point of view of functional analysis and discrete harmonic analysis (see, e.g., \cite{MR3752640,MR1036275}). Chaos expansion for Poisson processes on the real line was first obtained by It\^{o} \cite{MR0077017}, however its systematic study for abstract Poisson spaces is more recent, starting with the work \cite{MR2824870} by Last and Penrose, and has been motivated to a large extent by developments related to stochastic geometry and statistics (see, e.g., \cite{MR3791470,MR3444831,MR3161465,MR3585404,MR3813981,MR3896829}).

The goal of this article is to provide sharp tail and moment estimates for the building blocks of chaos expansions on the Poisson space, i.e., for multiple Wiener--It\^o integrals, and to apply them to the study of variables with finite chaos decomposition, in particular to Poisson $U$-statistics. This class of random variables allows to encode many important functionals of Poisson point processes, for instance the number of subgraphs of the Gilbert graph isomorphic to a fixed graph, length power functionals, intrinsic volumes of intersections of Poisson hyperplane processes with a fixed window, and estimators of Sylvester type probabilities \cite{MR3161465,MR3485348,MR3849811,MR3585402}. Poisson $U$-statistics are also important in an applied context, see, e.g., \cite{MR3447000,MR3585404} for statistical applications motivated by analysis of astrophysical data. In recent years Poisson $U$-statistics have been an object of intense studies, with particular focus on limit theorems with quantitative bounds and concentration of measure results (see, e.g., \cite{MR3161465,MR3585402,MR3485348,MR3849811,MR3096352,MR3161465,bonnet2024concentration}).
Our objective is twofold: first, to derive inequalities for multiple Wiener--It\^o integrals on the Poisson space and Poisson $U$-statistics, which provide multi-level concentration, similarly as in the case of known inequalities for multiple Gaussian stochastic integrals \cite{MR2294983}, canonical $U$-statistics in independent random variables or multiple stochastic integrals with respect to processes with independent increments on the line \cite{MR1857312,MR2073426,MR2294982}; second, to use these inequalities to complement known results on convergence in distribution of Poisson $U$-statistics and multiple stochastic integrals with laws of the iterated logarithm.

Since a rigorous statement of our results requires quite involved notation, we postpone it to Section \ref{sec:main-results}, and here we describe them in a slightly informal setting and only in special cases.

Let thus $\eta$ be a Poisson point process on a space $\XX$, with a $\sigma$-finite intensity $\lambda$. Consider also a Poisson proces $\eeta$ on $\XX\times [0,\infty)$ with intensity $\lambda \otimes \Leb$, where $\Leb$ is the Lebesgue measure. We interpret the process $\eeta$ in the usual way, the first coordinate of a point from $\eeta$ describes the \emph{spatial} position of the \emph{particle}, while the second one its \emph{time of arrival}.  By $\eta_t$ we denote the projection on $\XX$ of the process $\eeta$ restricted to $\XX\times [0,t]$, i.e., $\eta_t$ consists of locations of particles which have arrived up to time $t$. Such a spatiotemporal setting is used, e.g., in statistics, when modelling consecutive observations arriving in time. Note that $\eta_1$ has the same distribution as $\eta$.

For a symmetric function $g\colon \XX^d \to \R$ in $L_2(\XX^d,\lambda^{\otimes d})$ and $t \ge 0$ we consider the multiple stochastic Wiener--It\^{o} integral of $g$ with respect to $\eta_t$ and denote it by $I^{(d)}_t(g)$.

Our first main result, given in Theorem \ref{thm:tails-and-moments-main}, is a concentration estimate of the form
\begin{align}\label{eq:main-ineq-intro}
  \p(\sup_{t\le T} |I^{(d)}_t(g)| \ge u) \le 2\exp\Big(-c_d \min_{1\le k\le 2d} \Big(\frac{u}{A_k(T)}\Big)^{2/k}\Big),
\end{align}
where $c_d$ is a constant depending only on $d$ and $A_k$ are certain norms of the function $g$. We defer the exact definition of the parameters $A_k$ to Section \ref{sec:main-results}, here we only mention that they are related to certain injective tensor product norms of the kernel $g$, interpreted as a multilinear functional. Importantly, $A_1$ is comparable to the variance of $I^{(d)}_T(g)$ and $A_{2d} = \|g\|_\infty$. The advantage of the estimate \eqref{eq:main-ineq-intro} is related to the form of parameters $A_k$ and comes from two important properties. Under appropriate scaling in the limit as $T \to \infty$ it recovers the estimates for Gaussian stochastic integrals due to Latała \cite{MR2294983}, which are known to be two-sided (up to constants depending only on $d$). Moreover, it is strong enough to imply (after some preliminary truncation arguments) the Law of the Iterated Logarithm (LIL) for the process $(I^{d}_t(g))_{t\ge 0}$, which is our second main result, Theorem \ref{thm:LIL-integrals}. It asserts that for a square integrable $g$, with probability one,
\begin{displaymath}
  \limsup_{t\to \infty}\frac{|I^{(d)}_t(g)|}{(t\log\log t)^{d/2}} < \infty
\end{displaymath}
and the cluster set of the function $t\to \frac{I^{(d)}_t(g)}{(2t\loglog t)^{d/2}}$ as $t\to \infty$ is almost surely equal to the set
\begin{align}\label{eq:cluster-set-intro}
  \mathcal{K}_g = \Big\{\int_{\XX^d} g(x_1,\ldots,x_d)\varphi(x_1)\cdots \varphi(x_d)d\lambda^{\otimes d}(x_1,\ldots,x_d)\colon\; \varphi \colon \XX \to \R, \int_\XX \varphi^2 d\lambda \le 1\Big\}.
\end{align}

This result is an analogue of the compact LIL for $U$-statistics in independent random variables due to Arcones and Gin\'e \cite{MR1348376}. It complements in a natural way earlier results on the weak convergence of $t^{-d/2} I^{(d)}_t(g)$ due to Lachi\`eze-Rey and Peccati \cite{MR3096352}.

Results for multiple stochastic integrals are interesting in their own right, however a part of our motivation comes from the analysis of Poisson $U$-statistics, which as we have already mentioned, represent many functionals of interests in particular in stochastic geometry and statistics. Recall first that classical $U$-statistics, based on an i.i.d. sample $Y_1,Y_2,\ldots$ and a symmetric function $g$ of $d$ arguments, are random variables of the form
\begin{align}\label{eq:classical-Ustat}
  \sum_{1\le i_1 \neq \ldots \neq i_d \le n} h(Y_{i_1},\ldots,Y_{i_d}).
\end{align}
They were introduced by Halmos \cite{MR0015746} and Hoeffding \cite{MR0026294} in the 1940s in the context of unbiased estimation. They play an important role in asymptotic statistics not only as estimators, but also as higher order terms in Taylor expansions of smooth functionals of the data. They also appear as tools in combinatorics, stochastic geometry and statistical physics, and from a purely probabilistic point  of view they serve as a natural model of sums of dependent random variables, since despite their simple definition they exhibit quite complex behaviour. We recommend the monographs \cite{MR1666908,MR1075417} for a detailed account of their theory.

Poisson $U$-statistics are natural counterparts of \eqref{eq:classical-Ustat}. For a $\lambda^{\otimes d}$-integrable symmetric function $g\colon \XX^d\to \R$, the Poisson $U$-statistic based on $g$ is defined as
\begin{displaymath}
  U(g,\eta) = \sum_{i_1,\ldots, i_d}^{\neq} g(X_{i_1},\ldots,X_{i_d}),
\end{displaymath}
where $X_i$'s are the points of the Poisson process $\eta$ and the summation is over all $d$-tuples $(i_1,\ldots,i_d)$ with pairwise distinct coordinates.

A square integrable $U$-statistic has a finite chaos expansion, i.e., it can be represented as a finite sum of multiple stochastic integrals,
\begin{align}\label{eq:Ustat-chaos-expansion-intro}
U(g,\eta) - \E U(g,\eta)= \sum_{n=1}^d \binom{d}{n}I^{(n)}_t(g_n)
\end{align}
for certain (explicit) functions $g_n$ (see \cite{MR3161465} and formula \eqref{eq:U-stat-chaos-kernels} below).

Inequalities and laws of the iterated logarithm for stochastic integrals can be therefore easily transferred to the process $(U(g,\eta_t))_{t\ge 0}$. We defer the precise formulations to Corollaries \ref{cor:U-stat-tail} and \ref{cor:Ustat-LIL}. Here let us only explain that in the Law of the Iterated Logarithm we will consider the setting of fixed $g$, sometimes called the \emph{geometric case} in the literature (as the kernel does not change with growing intensity, i.e., it depends only on $\lambda$ and not on $t$), while our inequalities will be also applied to analyze the \emph{local case}, when $\XX$ is endowed with a metric $\rho$ and one considers functions of the form (for simplicity of the exposition we restrict here to $d=2$)
\begin{displaymath}
g_t(x,y) = g(x,y)\ind{\rho(x,y)\le r_t}
\end{displaymath}
with $r_t \to 0$ as $t\to \infty$. In the simplest case of $g(x,y)\equiv 1$, the $U$-statistic based on this particular function $g$ equals to a double number of edges in the geometric random Gilbert graph. Depending on the exact behaviour of $r_t$ as $t\to \infty$, it may happen that various summands in the chaos decomposition \eqref{eq:Ustat-chaos-expansion-intro} dominate. Our inequalities allow to asymptotically capture this behaviour at the level of tail inequalities in a certain growing window of parameters. We provide the examples in Section \ref{sec:examples}. They complement known results \cite{MR3485348,MR3849811,schulte2023moderate,bonnet2024concentration} which are focused mostly on the situation in which the first term in the chaotic expansion dominates (which corresponds to what is known as dense and thermodynamic regimes), by capturing Gaussian tail estimates in all regimes.

\subsection*{Connections with previous results}

As mentioned above, the chaos decomposition property for Poisson processes goes back to the work by It\^{o} \cite{MR0077017}. For abstract Poisson spaces it was proved by Last and Penrose in \cite{MR2824870}. In the following years a lot of attention has been devoted to Malliavin calculus and the Malliavin--Stein method in the Poissonian context, see the book \cite{MR3444831} for a nice exposition. Starting from the early 2000s many authors have also investigated the concentration of measure phenomenon on the Poisson space, by means of various tools, e.g., isoperimetric type estimates, martingale methods, functional inequalities, and transportation of measure techniques. The (probably non-exhaustive) list of references is \cite{MR1962538,MR2073426,MR2317609,MR2458025,MR3151752,MR3485348,MR3473096,MR3849811,MR4372142,MR4282691,MR4283247}.

In particular, inequalities for Poisson $U$-statistics were considered by Bachmann and Peccati \cite{MR3485348} and Bachmann and Reitzner \cite{MR3849811}. The first article provides certain abstract inequalities for functionals with self-bounding properties, which can be verified under appropriate assumptions on the kernel of the $U$-statistic. The second one deals with a special class of kernels  corresponding to subgraph counts in the Gilbert graph. Further inequalities are proved in recent articles by Schulte and Th\"ale \cite{schulte2023moderate} and Bonnet and Gusakova \cite{bonnet2024concentration}. They are based on assumptions concerning bounds on moments of all order of the kernel of the $U$-statistics, which can be verified in many case of interest, especially if the kernel is nonnegative. In particular,  \cite{bonnet2024concentration} complements earlier results with a Poisson type logarithmic correction in the estimate on $\p(|U(g)| \ge u)$ for large values of $u$.

Our inequalities for stochastic integrals and subsequent corollaries to $U$-statistics are extensions of results for double stochastic integrals with respect to Poisson processes on the line, which were proven by Houdr\'{e} and Reynaud-Bouret \cite{MR2073426}. They were counterparts of tail and moment estimates for $U$-statistics of order 2 in independent random variables, obtained by Gin\'e--Lata{\l}a--Zinn \cite{MR1857312}. These inequalities were later extended in \cite{MR2294982} to $U$-statistics and stochastic integrals of arbitrary order with respect to processes with independent increments. Our approach is inspired by this article. The common point of the inequalities we present and those from aforementioned papers is the presence of a family of norms (see Section \ref{sec:inequalities} for the definitions), indexed by partitions of finite sets. In the simplest case such norms appeared for the first time as parameters in exponential tail estimates in the Hanson--Wright inequality for quadratic forms in independent subgaussian random variables, in disguise of the operator norm of the matrix. In full generality they were introduced in the context of tail estimates by Latała, in his two-sided inequalities for homogeneous polynomials in independent Gaussian variables \cite{MR2294983} (see also Remark \ref{re:discussion-of-norms} below). The inequalities for $U$-statistics proven in \cite{MR2294982} were subsequently used in \cite{MR2408582} to derive necessary and sufficient conditions for the bounded Law of the Iterated Logarithm for $U$-statistics. In our proof of the LIL we will follow some ideas from the latter article as well as from \cite{MR1348376}.

To the best of our knowledge there are not too many results in the literature, which consider the Law of the Iterated Logarithm for Poisson functionals. The case of single stochastic integrals and $U$-statistics of order one can be inferred from classical results for L\'{e}vy processes due to Gnedenko \cite{MR0009265}. As for higher order $U$-statistics, the Law of the Iterated Logarithm with normalization $t^{d-1/2}\sqrt{\loglog t}$ (which corresponds to the situation when the function $g_1$ in \eqref{eq:Ustat-chaos-expansion-intro} is nonzero) has been obtained by Matthias Reitzner in an unpublished note. In \cite{MR4260506}, Krebs considers the homogeneous Poisson process on $\R^d$ and shows the LIL in discrete time with normalization $\sqrt{n\loglog n}$ for Poisson functionals with distribution invariant under translations, satisfying certain \emph{stabilizing property}. He considers also the LIL for fixed intensity and growing window of observation, which we do not discuss in this paper. Laws of the iterated logarithm for sums of stabilizing functionals have been also considered in \cite{MR2548003}. Another article concerning laws of the iterated logarithm related to Poisson point processes is \cite{MR2451052}, which studies the case of additive functionals of random sequential packing measures on the cube. We would like to stress that in all these results the limiting set is a symmetric interval, they do not involve more complex phenomena when the limiting set is given by the set $\mathcal{K}_g$ defined in \eqref{eq:cluster-set-intro}.

\subsection*{Organization of the article} In Section \ref{sec:notation} we introduce the setting and basic notation. Next in Section \ref{sec:inequalities} we present our main results concerning moment and tail estimates, and in Section \ref{sec:LIL} we state the laws of the iterated logarithm. Section \ref{sec:examples} is devoted to applications of our abstract results to specific models and functionals. The proofs of these results are provided in Sections \ref{sec:inequalities-proofs} (exponential inequalities), \ref{sec:lil-proofs} (LIL) and \ref{sec:proofs-applications} (applications). In the Appendix we present certain auxiliary results on decoupling inequalities and proofs of technical lemmas introduced in Section \ref{sec:proofs}.

\subsection*{Acknowledgements} We would like to thank Matthias Reitzner for inspiring conversations and for communicating to us his results on the LIL for Poisson $U$-statistics, and Christoph Th\"ale for providing us with references concerning LIL for Poisson functionals.

\section{Main results}\label{sec:main-results}

\subsection{Notation and setting} \label{sec:notation}

We will now describe our setting, in particular we will briefly recall the basic notions related to Poisson point processes, multiple stochastic integrals and Wiener--It\^{o} expansions. We refer, e.g., to the monograph \cite{MR3791470} for a detailed exposition.

\subsection*{Generalities on Poisson processes}

Let $(\XX,\mathcal X)$ be a measurable space, equipped with a $\sigma$-finite measure $\lambda$. For technical reasons we will assume that $\mathcal{X}$ is countably generated. Let also $\NN$ be the set of all measures on $\mathcal{X}$ which can be represented
as a countable sum of $\N$-valued measures. We endow $\NN$ with the smallest $\sigma$-field $\mathcal{N}$ such that for all $A \in \mathcal{X}$ the map $\NN \ni \mu \mapsto \mu(A)$ is Borel-measurable.

A Poisson point process on $\XX$ with intensity $\lambda$, is a random element $\eta$ of $\NN$, such that for any $n \ge 1$ and pairwise disjoint $A_1,\ldots,A_n \in \mathcal{X}$
the random variables $\eta(A_1),\ldots,\eta(A_n)$ are independent and $\eta(A_i)$ has the Poisson distribution with parameter $\lambda(A_i)$ (with natural interpretation as Dirac's deltas at zero or infinity if $\lambda(A_i) = 0$ or $\infty$, respectively).

Without loss of generality we will most of the time assume that $\eta$ is proper, i.e., there is an $\N\cup\{\infty\}$-valued random variable $\kappa$ and $\XX$-valued random variables $X_1,X_2,\ldots$ such that
\begin{align}\label{eq:proper-process}
  \eta = \sum_{n=1}^\kappa \delta_{X_n}.
\end{align}
Indeed, the standard construction of the Poisson process implies that in our setting there always exists a proper Poisson process with a given intensity $\lambda$ (see, e.g., \cite[Section 3.2]{MR3791470}).

For $d \in \N$ and $p\ge 1$, by $L_{p}(\XX^d,\lambda^{\otimes d})$ we will denote the space of all Borel measurable functions $g\colon \XX^d \to \R$, satisfying $\|g\|_p = (\int_{\XX^d} |g|^p d\lambda^{\otimes d})^{1/p} < \infty$. By $L_p^s(\XX^d,\lambda^{\otimes d})$ we denote the subspace of $L_p(\XX^d,\lambda^{\otimes d})$ consisting of functions symmetric with respect to the permutation of their arguments. Whenever it does not lead to a misunderstanding we will abbreviate the notation and write simply $L_p(\XX^d)$ or $L_p^s(\XX^d)$, suppressing the product measure $\lambda^{\otimes d}$. If $d=0$, we interpret $L_{p}(\XX^d)$ as $\R$, the space of constants.

If $g \in L_1(\XX^d)$, we define the Wiener--It\^{o} integral of $g$ with respect to $\eta$ as
\begin{align}\label{eq:L_1-integral-non-symmetric}
I^{(d)}(g,\eta) = I^{(d)}(g) = \sum_{J \subseteq [d]} (-1)^{d-|J|} \int_{\XX^{J^c}}\int_{\XX^{J}} g(x_1,\ldots,x_d) d\eta^{(|J|)}(x_J)d \lambda^{\otimes (d-|J|)}(x_{J^c})
\end{align}
where $[d]= \{1,\ldots,d\}$ and $\eta^{(k)}$ are the factorial measures of $\eta$, defined inductively as

\begin{align}\label{eq:factorial-measures}
  \eta^{(1)} &= \eta,\\
  \eta^{(k+1)} &= \int_{\XX^k}\Big[\int_{\XX} \ind{(x_1,\ldots,x_{k+1}) \in \cdot}d \eta(x_{k+1}) - \sum_{j=1}^k \ind{(x_1,\ldots,x_k,x_{j}) \in \cdot}\Big]d\eta^{(k)}(x_1,\ldots,x_k).\nonumber
\end{align}
For $k=0$ we interpret the above formulas as $I^{(0)}(c) = c$.

\medskip

If $g \in L_1^s(\XX^d)$, then \eqref{eq:L_1-integral-non-symmetric} simplifies to
\begin{align}\label{eq:L_1-integral}
  I^{(d)}(g,\eta) = I^{(d)}(g) = \sum_{k=0}^d (-1)^{d-k} {d \choose k}\int_{\XX^d} g(x_1,\ldots,x_d) d(\eta^{(k)} \otimes \lambda^{\otimes (d-k)})(x_1,\ldots,x_d).
\end{align}

If $\eta$ is given by \eqref{eq:proper-process}, then

\begin{align}\label{eq:factorial-proper}
  \eta^{(k)}(A) = \sum_{1\le i_1,\ldots,i_k\le \kappa}^{\neq} \ind{(X_{i_1},\ldots,X_{i_k}) \in A},
\end{align}
where the superscript $\neq$ indicates that the summation is taken over all k-tuples with pairwise distinct coordinates.

One can show that the linear operator $I^{(d)}$ extends to a bounded linear operator on $L_2(\XX^d)$. For $g_i \in L_2(\XX^{d_i})$, $i = 1,2$ we have
\begin{displaymath}
  \E I^{(d_1)}(g_1)I^{(d_2)}(g_2) = d_1! \ind{d_1=d_2} \int_{\XX^{d_1}} g_1^{sym} g_2^{sym}d\lambda^{\otimes d_1}.
\end{displaymath}
where $g^{sym} \in L_2(\XX^d)$ is the symmetrization of $g$, given by
\begin{align}\label{eq:symmetrization-definition}
  g^{sym}(x_1,\ldots,x_d) = \frac{1}{d!}\sum_{\sigma \in S_d} g(x_{\sigma(1)},\ldots,x_{\sigma(d)}),
\end{align}
where $S_d$ is the set of all permutations of $[d]$.

This implies that for $g \in L_2(\XX^d)$,
\begin{align}\label{eq:integral-non-symmetric-symmetric}
I^{(d)}(g) = I^{(d)}(g^{sym})
\end{align}
and, moreover, the restriction of $(d!)^{-1/2} I^{(d)}(g)$ to $L_2^s(\XX^d)$ is an isometric embedding of $L_2^s(\XX^d)$ into $L_2(\Omega,\sigma(\eta),\p)$, where $\sigma(\eta)$ is the $\sigma$-field generated by $\eta$.

In fact, we have an orthogonal decomposition
\begin{displaymath}
  L_2(\Omega,\sigma(\eta),\p) = \bigoplus_{d=0}^\infty H_d,
\end{displaymath}
where $H_d = I^{(d)}(L_2^s(\XX^d))$.

More precisely, for any $F \in L_2(\Omega,\sigma(\eta),\p)$, we have
\begin{displaymath}
  F = \sum_{d=0}^\infty \frac{1}{d!}I^{(d)}(T_dF),
\end{displaymath}
with $T_d F \colon \XX^d \to \R$ given by $T_d F(x_1,\ldots,x_d) = \E D_{x_1}\cdots D_{x_d} F(\eta)$, where for any $x \in \XX$, $D_x F(\eta) = F(\eta+\delta_x) - F(\eta)$ is the add one cost operator. Observe that with certain abuse of notation we sometimes write $F$, treating it as a random variable on $\Omega$ and sometimes we write $F(\eta)$, treating $F$ as a function on $\NN$. Clearly there may be many measurable functions on $\NN$ which lead to the same random variables when composed with $\eta$, but each two of them coincide outside a set, which is null with respect to the law of $\eta$. Moreover, $D_x F(\eta)$ is determined almost everywhere with respect to the product measure $\lambda\otimes \p$ on $\XX \times \Omega$. This convention should therefore not lead to misunderstanding. Note that $I^{(0)}(T_0 F) = T_0 F = \E F(\eta)$.

\subsection*{The time dependent process}

Let $\mathcal L$ be Lebesgue measure on $([0,\infty),\mathcal B([0,\infty)))$, where $B([0,\infty))$ is the Borel $\sigma$-field on $[0,\infty)$. On the product space $(\YY,\mathcal Y)$, where $\YY = \XX \times [0,\infty)$ and $\mathcal Y = \mathcal X \otimes \mathcal B([0,\infty))$, we define a Poisson point process $\eeta$ with intensity measure $\lambda \otimes \mathcal L$. We will denote a generic element of $\YY$ by $(x,t)$ and we will think of $t$ as \emph{time}. For any $t \in [0,\infty)$ we define $\eeta_t$ as the restriction of $\eeta$ to $\XX\times [0,t]$
and $\eta_t$ as the projection of $\eeta_t$ to $\XX$, i.e., $\eta_t(A) = \eeta(A \times [0,t])$. Thus $\eta_t$ may be interpreted as the collection of random points on $\XX$, which have \emph{arrived} up to time $t$.

\medskip

Note that the intensity measure of $\eta_t$, treated as a Poisson point process on the space $(\mathcal \XX,\mathcal X)$, is $t\lambda$. In particular, when discussing properties depending only on the distribution, we may identify the Poisson point process $\eta$ introduced above with $\eta_1$.

\medskip

Fix $d \in \N_+$ and let $g \in L_2^s(\XX^d)$. By $\pi:\YY^d \to \XX^d$ we denote the projection onto $\XX$ coordinates, i.e., $\pi( (x_1,t_1),(x_2,t_2),...,(x_d,t_d)) = (x_1,x_2,...,x_d)$ (we will use the same letter $\pi$ to denote all the projections for $d=1,2,\ldots$).

For $t \ge 0$ we will consider the stochastic integral
\begin{align}\label{eq:integral-equalities}
  I_t^{(d)}(g) = I^{(d)}(g\circ \pi, \eeta_t) = I^{(d)}(g,\eta_t)
\end{align}
(in Section \ref{sec:reductions} we comment on the second equality above).

For fixed $t$, $I_t^{(d)}(g)$ has the same distribution as the Wiener--It\^{o} integral of $g$ with respect to the Poisson process on $\XX$, with intensity $t\lambda$. The above construction allows to define all such integrals on a common probability space and leads to a stochastic process $(I^{(d)}_t(g))_{t\ge 0}$. It is easy to check that this process is a square integrable martingale. In what follows we will always consider its c\`adl\`ag version and for this version we will prove the Law of the Iterated Logarithm.

\subsection*{Poisson $U$-statistics and their Wiener--It\^o expansions}

Assume that $\eta$ is given by \eqref{eq:proper-process}. Following \cite{MR3161465}, for $g \in L_1^s(\XX^d)$ we define the Poisson $U$-statistics as
\begin{displaymath}
  U(g) = U(g,\eta) = \sum_{1\le i_1,\ldots,i_d\le \kappa}^{\neq} g(X_{i_1},\ldots,X_{i_d}).
\end{displaymath}

Note that one can equivalently write $U(g) = \int_{\XX^d} g d\eta^{(d)}$, where $\eta^{(d)}$ is the factorial measure of $\eta$, as defined in \eqref{eq:factorial-measures}. We will treat this equality as the definition of $U(g)$ in the general case. If $\eta$ is simple, i.e., it is almost surely a sum of Dirac deltas at distinct points and can be thus identified with a countable subset of $\XX$, we will sometimes write
\begin{displaymath}
  U(g) = \sum_{(x_1,\ldots,x_d) \in \eta^d_{\neq}} g(x_1,\ldots,x_d),
\end{displaymath}
with $\eta^d_{\neq}$ denoting the subspace of the Cartesian product $\eta^d$ consisting of points with pairwise distinct coordinates.
We remark that if $\XX$ is a Borel subset of a Polish space and $\lambda$ is a locally finite Borel measure, then $\eta$ is proper, moreover in this case $\eta$ is simple iff $\lambda$ is diffuse, i.e., $\lambda(\{x\}) =0$ for all $x \in \XX$ (see \cite[Proposition 6.9]{MR3791470}).

The definition in terms of factorial measures shows that $U(g)$ is measurable with respect to $\eta$ seen as a random measure.  Below we present basic facts concerning Poisson $U$-statistics and refer to \cite[Chapter 12.3]{MR3791470}, \cite[Chapter 12]{MR3585402}, \cite{MR3161465} for a detailed exposition.

In a dynamical context we will also consider a process $(U_t(g))_{t\ge 0}$, defined as
\begin{displaymath}
  U_t(g) = U(g,\eta_t) = U(g\circ \pi, \eeta_t).
\end{displaymath}

The  multivariate Mecke equation (see, e.g., \cite[Theorem 4.4]{MR3791470}) gives that
\begin{align}\label{eq:Mecke}
  \E U(g) = \int_{\XX^d} g(x_1,\ldots,x_d)d\lambda^{\otimes d}(x_1,\ldots,x_d),
\end{align}
whereas square integrability and the form of the variance for $U(g)$ is described in the following theorem, proved in \cite{MR3161465}.
\begin{theorem}\label{thm:Reitzner-Schulte}

(i) If $F=U(g) \in L_2(\Omega,\sigma(\eta),\p)$ then
\begin{displaymath}
  D_{x_1,\ldots,x_n}F = \frac{d!}{(d-n)!} \sum_{1\le i_1,\ldots,i_{d-n}\le \kappa}^{\neq} g(x_1,\ldots,x_n,X_{i_1},\ldots,X_{i_{d-n}})
\end{displaymath}
and
\begin{displaymath}
  T_n F(x_1,\ldots,x_n) = \ \frac{d!}{(d-n)!}\int_{\XX^{d-n}} g(x_1,\ldots,x_n,y_{n+1},\ldots,y_d)d\lambda^{\otimes (d-n)}(y_{n+1},\ldots,y_d).
\end{displaymath}

In particular, $\Var(F) = \sum_{n=1}^d n! \binom{d}{n}^2 V_n$, where

\begin{multline*}
 V_n = \int_{\XX^n}\Big(\int_{\XX^{d-n}} g(x_1,\ldots,x_n,y_{n+1},\ldots,y_d)d\lambda^{\otimes(d-n)}(y_{n+1},\ldots,y_d)\Big)^2d\lambda^{\otimes n}(x_1,\ldots,x_n).
\end{multline*}
\medskip

\noindent (ii) If $F=U(g) \in L_2(\Omega,\sigma(\eta),\p)$, then $F$ has a finite Wiener--It\^{o} expansion with kernels $T_n F \in L_1^s(\XX^n)\cap L_2^s(\XX^n)$, $n=0,\ldots,d$.

\medskip

\noindent (iii) If $F\in L_2(\Omega,\sigma(\eta),\p)$ has a finite Wiener--It\^{o} expansion with kernels $T_n F$, $n= 0,\ldots, d$, such that for all $n$, $T_n F \in L_1^s(\XX^n)\cap L_2^s(\XX^n)$, then $F$ is a finite sum of Poisson $U$-statistics.
\end{theorem}

The authors of \cite{MR3161465} discuss also many examples showing that in general one cannot drop the assumptions of the above theorem.

The formula \eqref{eq:Mecke} and the above theorem, when applied to $U_t(g)$, give that
\begin{align}\label{eq:U-stat-mean}
\E U_t(g) = t^d \int_{\XX^d} g(x_1,\ldots,x_d)d\lambda^{\otimes d}(x_1,\ldots x_d)
\end{align}
and
\begin{align}\label{eq:U-stat-chaos-expansion}
  U_t(g) - \E U_t(g)= \sum_{n=1}^d \binom{d}{n} t^{d-n} I_t^{(n)}(g_n),
\end{align}
where
\begin{align}\label{eq:U-stat-chaos-kernels}
g_n(x_1,\ldots,x_n) = \int_{\XX^{d-n}} g(x_1,\ldots,x_n,y_{n+1},\ldots,y_d)d\lambda^{\otimes (d-n)}(y_{n+1},\ldots, y_d).
\end{align}
Moreover,
\begin{align}\label{eq:U-stati-variance}
  \Var(U_t(g)) = \sum_{n=1}^d n! \binom{d}{n}^2 t^{2d-n} \|g_n\|_2^2
\end{align}
and thus, for fixed $g$, the leading term in the asymptotic behaviour of the variance as $t\to \infty$ comes from the first summand for which $g_n$ does not vanish. For positive functions $g$, appearing in stochastic geometry, it is the first term. For local $U$-statistics, in which the kernel $g$ may change with $t$ and is non-vanishing only if all the arguments are at a distance at most $r_t$ for some $r_t\to 0$, it may happen that one of the other terms on the right-hand side of \eqref{eq:U-stati-variance} dominates even if the kernels are positive.

\subsection{Moment and exponential inequalities}\label{sec:inequalities}

In the whole article $C_a, c_a, C_{a,b}, c_{a,b}, \widetilde{C}_a$, etc., will denote positive constants depending only on the parameters in the subscripts. Their values may vary between occurrences, however at some places in the proofs, for transparency, we will keep track of how the constants change between lines.

In order to express our tail and moment estimates, we need to introduce a certain family of norms of functions $g\in L_2^s(\XX^d)$, indexed by partitions of the subsets of $[d] = \{1,\ldots,d\}$. For $I \subseteq [d]$ let $\PP_I$ denote the family of all partitions of $I$ into nonempty, pairwise disjoint sets (with the convention that for $I=\emptyset$, $\PP_I$ consists of a single element -- the empty partition). For $x = (x_1,\ldots,x_d)\in \XX^d$ and $I\subseteq [d]$ we will denote $x_I = (x_i)_{i\in I}$. We will first define the norms $\|\cdot\|_\mathcal{J}$ for $\mathcal{J} \in \PP_{[d]}$. Next, we will comment on their conditional versions for $\mathcal{J} \in \PP_I$ with $I \subsetneq [d]$, which will be functions of $x_{I^c}$. Recall that for $g \colon \XX^d \to \R$, by $\|g\|_p$ we denote the $L_p(\XX^d) = L_p(\XX^d,\lambda^{\otimes d})$ norm of $g$.

\begin{defi}
For $g \in L_2(\XX^d)$ and $\mathcal{J} = \{J_1,\ldots,J_k\} \in \PP_{[d]}$ define
\begin{displaymath}
  \|g\|_\mathcal{J} = \sup\Big\{\int_{\XX^d} g(x) \prod_{j=1}^k \varphi_j(x_{J_j})d\lambda^{\otimes d}(x_1,\ldots,x_d)\colon\; \varphi_j\colon \XX^{J_j}\to \R, \|\varphi_j\|_2 \le 1 \Big\}.
\end{displaymath}
\end{defi}

Observe that by the Cauchy--Schwarz inequality $\|g\|_{\{[d]\}} = \|g\|_2 \ge \|g\|_\mathcal{J}$ for all $\mathcal{J} \in \PP_{[d]}$. In general, $\|g\|_{\mathcal{J}}$ is a norm of a multilinear functional defined by $g$ on the product $\prod_{j=1}^k L_2(\XX^{J_j})$. In what follows, to simplify the notation, we will often suppress the outer brackets and commas in the notation for $\mathcal{J}$ and write, e.g., $\|\cdot\|_{\{1,2,3\}}$, $\|\cdot\|_{\{1,2\}\{3\}}$ instead of $\|\cdot\|_{\{\{1,2,3\}\}}$, $\|\cdot\|_{\{\{1,2\},\{3\}\}}$. We will stick to this convention throughout the whole article.

For instance, for $g \colon \XX^2 \to \R$, we have
\begin{align*}
\|g\|_{\{1,2\}} & = \|g\|_2 = \Big(\int_{\XX^2} g(x,y)^2\lambda(dx)\lambda(dy)\Big)^{1/2}, \\
\|g\|_{\{1\},\{2\}} & = \sup\Big\{ \int_{\XX^2} g(x,y) \varphi_1(x)\varphi_2(y) \lambda(dx)\lambda(dy) \colon \|\varphi_i\|_2 \le 1, \; i=1,2\Big\} \\
&= \sup\Big\{\Big(\int_\XX \Big(\int_\XX g(x,y)\varphi(y)\lambda(dy)\Big)^2\lambda(dx)\Big)^{1/2}\colon \|\varphi\|_2 \le 1\Big\}
\end{align*}
and for $g \colon \XX^3 \to \R$,
\begin{align*}
  \|g\|_{\{1,2,3\}} & = \|g\|_2 = \Big(\int_{\XX^3} g(x,y,z)^2\lambda(dx)\lambda(dy)\lambda(dz)\Big)^{1/2},\\
  \|g\|_{\{1,2\}\{3\}} & =   \sup\Big\{ \int_{\XX^3} g(x,y,z) \varphi_1(x,y)\varphi_2(z) \lambda(dx)\lambda(dy)\lambda(dz) \colon \|\varphi_i\|_2 \le 1, \; i=1,2\Big\} \\
&= \sup\Big\{\Big(\int_{\XX^2} \Big(\int_\XX g(x,y,z)\varphi(z)\lambda(dz)\Big)^2\lambda(dx)\lambda(dy)\Big)^{1/2}\colon \|\varphi\|_2 \le 1\Big\}\\
& = \sup\Big\{\Big(\int_{\XX} \Big(\int_{\XX^2} g(x,y,z)\varphi(x,y)\lambda(dx)\lambda(dy)\Big)^2\lambda(dz)\Big)^{1/2}\colon \|\varphi\|_2 \le 1\Big\},\\
\|g\|_{\{1\}\{2\}\{3\}} & = \sup\Big\{ \int_{\XX^3} g(x,y,z) \varphi_1(x)\varphi_2(y)\varphi_3(z) \lambda(dx)\lambda(dy) \lambda(dz) \colon \|\varphi_i\|_2 \le 1, \; i\in [3]\Big\}.
\end{align*}

If $g \in L_2(\XX^d)$ and $I \subsetneq [d]$ then by the Fubini theorem, the function $\XX^{I} \ni x_{I} \mapsto  g(x_1,\ldots,x_d)$ is an element of $L_2(\XX^{I})$ for all $x_{I^c} \in \XX^{I^c}$ outside a set of $\lambda^{\otimes(d-|I|)}$ measure zero. For $\mathcal{J} \in \PP_I$ by $\|g\|_{\mathcal{J}}$ we will denote the function $\XX^{d-|I|} \ni x_{I^c} \to \|g(x_{I^c},\cdot)\|_{\mathcal{J}}$. What will be important in the exponential inequalities is the $L_\infty$ norm of this function, $\|\|g\|_\mathcal{J}\|_\infty$. To illustrate this notation, let us consider two examples. First, for $g \colon \XX^2 \to \R$, $I = \{2\}$ and $\mathcal{J} = \{\{2\}\} \in \PP_{\{2\}}$ we have
\begin{displaymath}
  \Big\|\|g\|_{\{2\}}\Big\|_\infty = \esssup_{x \in \XX} \; \sup\Big\{\int_\XX g(x,y)\varphi(y)\lambda(dy) \colon \|\varphi\|_2 \le 1 \Big\} = \esssup_{x \in \XX} \|g(x,\cdot)\|_2.
\end{displaymath}
Now, for $g \colon \XX^3 \to \R$, $I = \{1,2\}$ and $\mathcal{J} = \{\{1\},\{2\}\}$,
\begin{align*}
  \Big\|\|g\|_{\{1\}\{2\}}\Big\|_\infty & = \esssup_{z\in \XX} \; \sup\Big\{\int_{\XX^2} g(x,y,z)\varphi_1(x)\varphi_2(y) \lambda(dx)\lambda(dy)\colon \|\varphi_i\|_2 \le 1,\; i=1,2\Big\} \\
  & = \esssup_{z\in \XX} \|g(\cdot,\cdot,z)\|_{\{1\}\{2\}}.
\end{align*}
Note that thanks to the assumption that $\mathcal{X}$ is countably generated, the suprema above can be taken over countable dense subsets in the unit balls of the corresponding $L_2$ spaces, so the functions $\|g\|_\mathcal{J}$ are measurable.

In the special case, when $\mathcal{J} = \emptyset$ is the unique partition of $I = \emptyset$, we have
\begin{displaymath}
  \|g\|_\emptyset = g,\textrm{ and thus } \|\|g\|_\emptyset\|_\infty = \|g\|_\infty.
\end{displaymath}

To shorten the notation, in what follows we will sometimes write $\|\|g\|_{\mathcal{J}}\|_\infty = \|g\|_\mathcal{J}$ also when $\mathcal{J} \in \PP_{[d]}$. This will allow us to write all the terms in our estimates under a single sum.

Having introduced the notation, we are ready to state the first of our main results.
\begin{theorem}\label{thm:tails-and-moments-main}
Let $g \in L_2^s(\XX^d)$.

\noindent (i) There exists a constant $C_d$, depending only on $d$, such that for all $T > 0$ and $p \ge 2$,
\begin{align}\label{eq:moment-main}
\Big(\E \sup_{t\le T} |I_t^{(d)}(g)|^p\Big)^{1/p} \le C_d \sum_{k=0}^d \sum_{\mathcal{J} \in \PP_{\{k+1,\ldots,d\}}} p^{k + |\mathcal{J}|/2} T^{(d-k)/2} \Big\| \|g\|_{\mathcal{J}}\Big\|_\infty.
\end{align}
\noindent (ii) There exists a constant $c_d$, depending only on $d$, such that for all $T > 0$ and $u > 0$,
\begin{align}\label{eq:tail-main}
\p\Big(\sup_{t\le T} |I_t^{(d)}(g)| \ge u\Big) \le 2\exp\Big(-c_d\min_{0\le k \le d} \min_{\mathcal{J} \in \PP_{\{k+1,\ldots,d\}}} \Big(\frac{u}{T^{(d-k)/2}\|\|g\|_{\mathcal{J}}\|_\infty}\Big)^{\frac{2}{2k+|\mathcal{J}|}}\Big).
\end{align}
\end{theorem}

\begin{remark}
Note that the inequalities of the above proposition trivialize if for some $\mathcal{J}$, $\| \|g\|_{\mathcal{J}}\|_\infty = \infty$. To simplify the formulation of this and forthcoming results we, therefore, do not state explicitly assumptions on finiteness of these norms. The inequalities are meaningful only if all of them are finite.
\end{remark}

\begin{remark} \label{re:discussion-of-norms}
The partition norms used in the above theorem may at first look overly complicated, especially that in many situations it may be difficult to calculate them explicitly. We would like to stress that for $\mathcal{J} \in \PP_{[d]}$ their necessity in tail and moment estimates for multiple stochastic integrals comes from the fact that they are present in Latała's two-sided estimates for Gaussian multiple stochastic integrals, which appear as distributional limits of their Poisson counterparts. Let us now briefly describe this result to compare with our estimates. The original formulation in \cite{MR2294983} is given in terms of multilinear forms in independent standard Gaussian variables. It is straightforward to extend it to Gaussian stochastic integrals. We refer to \cite{MR1474726,MR2200233} for their definitions and general theory.

Let $G$ be a Gaussian random measure on $(\XX,\mathcal{X})$ with covariance $\Cov(G(A),G(B)) = \lambda(A\cap B)$ and for $g \in L_2^s(\XX^d)$ let $\mathfrak{I}^{(d)}(g)$ be the $d$-fold stochastic integral of $g$ with respect to $G$. Then, the results by Latała assert that for some constant $C_d$, depending only on $d$,
for all $p \ge 2$,
\begin{displaymath}
  \frac{1}{C_d} \sum_{\mathcal{J}\in\PP_{[d]}} p^{|\mathcal{J}|/2} \|g\|_\mathcal{J} \le \|\mathfrak{I}^{(d)}(g)\|_p \le C_d \sum_{\mathcal{J}\in\PP_{[d]}} p^{|\mathcal{J}|/2} \|g\|_\mathcal{J}
\end{displaymath}
and for $u > 0$,
\begin{multline*}
  \frac{1}{C_d} \exp\Big(-C_d\min_{\mathcal{J} \in \PP_{[d]}} \Big(\frac{u}{\|g\|_{\mathcal{J}}}\Big)^{2/|\mathcal{J}|}\Big) \le \p(|\mathfrak{I}^{(d)}(g)| \ge u) \\
  \le C_d \exp\Big(-\frac{1}{C_d}\min_{\mathcal{J} \in \PP_{[d]}} \Big(\frac{u}{\|g\|_{\mathcal{J}}}\Big)^{2/|\mathcal{J}|}\Big).
\end{multline*}
In other words, the part of our estimates \eqref{eq:moment-main} and \eqref{eq:tail-main} corresponding to $k = 0$ governs (up to constants) the moment and tail behaviour of $\mathfrak{I}^{(d)}$.

Since, as proved in \cite{MR3096352}, $\frac{1}{t^{d/2}} I_t^{(d)}(g)$ converges in distribution to $\mathfrak{I}^{(d)}(g)$, we can see that after some additional truncation steps, allowing to approximate a general function in $L_2^s(\XX^d)$ by functions $g$ for which $\|\|g\|_\mathcal{J}\|_\infty < \infty$ for all $\mathcal{J}$, we can recover from Theorem \ref{thm:tails-and-moments-main} the upper bounds from the estimates by Latała. Moreover, we can see that at least for $k=0$, we cannot in general replace the norms appearing in Theorem \ref{thm:tails-and-moments-main} by any substantially smaller quantities.\footnote{We would like to stress that our results do not provide a new proof of the inequalities by Latała, since in our argument (presented in Section \ref{sec:inequalities-proofs}) we rely on inequalities for $U$-statistics from \cite{MR2294982}, which in turn use in a crucial way the main technical estimate from \cite{MR2294983}.}

As for the conditional norms $\| \|g\|_\mathcal{J}\|_\infty$ for $\mathcal{J} \in \PP_{\{k+1,\ldots,d\}}$ and $k > 0$, their counterparts are present in inequalities for $U$-statistics in independent random variables presented in \cite{MR1857312,MR2073426,MR2294982} (generalizing the classical Bernstein inequality for sums of independent random variables, which also contains the $L_\infty$ norm of summands) and while we do not know if their presence in the estimates is necessary, by comparison of Gaussian and Poisson tails as well as by two-sided tail and moment estimates for polynomials in independent exponential variables \cite{MR3052405,MR3412778}, whose tails are intermediate between Gaussian and Poissonian, one can suspect that certain norms of this form should appear in the estimates. We stress that the question of providing two-sided estimates for moments of Poisson multiple stochastic integrals, analogous to Latała's results in the Gaussian case, remains open.
\end{remark}

\begin{remark}
In certain situations, depending on available information on the function $g$, it may be practical (and sometimes sufficient, as we will see, e.g., in the proof of the Law of the Iterated Logarithm) to replace the norms $\|g\|_\mathcal{J}$ and $\| \|g\|_\mathcal{J}\|_\infty$ by $\|g\|_2$ and $\esssup_{x_{I^c}} \|g(x_{I^c},\cdot)\|_2$). This way one loses the correspondence with two-sided inequalities for Gaussian integrals as well as the subgaussian nature of the inequality for small values of $u$, but one is able to obtain an estimate in terms of more tractable quantities. We will use this simplified form in the proof of the Law of Iterated Logarithm, so we state it here explicitly.
\end{remark}

\begin{cor}\label{cor:simplified-estimate}
Let $g \in L_2^s(\XX^d)$ and for $k = 0,\ldots,d$, let
\begin{displaymath}
B_k^2 = \| \|g\|_{\{k+1,...,d\}} \|_{\infty}^2 = \esssup_{x_1,\ldots,x_k} \int_{\XX^{d-k}} g(x_1,\ldots,x_d)^2d\lambda^{\otimes(d-k)}(x_{k+1},\ldots, x_d).
\end{displaymath}
Then, there exists a constant $C_d$, depending only on $d$, such that for all $T > 0$, $p \ge 2$,
\begin{displaymath}
  \Big(\E \sup_{t\le T} |I^{(d)}_t(g)|^p\Big)^{1/p} \le C_d \sum_{k=0}^d p^{(d+k)/2} T^{(d-k)/{2}} B_k.
\end{displaymath}
Moreover, for some constant $c_d>0$, depending only on $d$, and for all $T>0$, $u > 0$,
\begin{displaymath}
  \p(\sup_{t\le T} |I^{(d)}_t(g)| \ge u) \le 2\exp\Big(-c_d\min_{0\le k \le d} \Big(\frac{u}{T^{(d-k)/2} B_k}\Big)^{\frac{2}{d+k}}\Big).
\end{displaymath}
\end{cor}

We will now apply Theorem \ref{thm:tails-and-moments-main} to Poisson $U$-statistics. Using \eqref{eq:U-stat-chaos-expansion} and the triangle inequality one easily gets the following moment and tail estimates. In Section \ref{sec:examples} we will apply it to specific examples of Poisson $U$-statistics studied in the literature.

\begin{cor}\label{cor:U-stat-tail} Assume that $g\in L_1^s(\XX^d)\cap L_2^s(\XX^d)$ and let $g_n$ be the functions given by \eqref{eq:U-stat-chaos-kernels}.

\noindent (i) There exists a constant $C_d$, depending only on $d$, such that for all $T> 0$ and $p \ge 2$,
\begin{displaymath}
  \Big(\E \sup_{t\le T} | U_t(g) - \E U_t(g)|^p\Big)^{1/p} \le C_d\sum_{n=1}^d \sum_{k=0}^n\sum_{\mathcal{J}\in \PP_{\{k+1,\ldots,n\}}} T^{d-(n+k)/2} p^{k+|\mathcal{J}|/2}\Big\|\|g_n\|_{\mathcal{J}}\Big\|_\infty.
\end{displaymath}
\noindent (ii) There exists a constant $c_d > 0$, depending only on $d$, such that for all $T>0$ and $u>0$,
\begin{multline*}
  \p\Big(\sup_{t\le T} |U_t(g) - \E U_t(g)| \ge u\Big) \\  \le 2\exp\Big(-c_d \min_{1\le n\le d}\min_{0\le k \le n}\min_{\mathcal{J}\in \PP_{\{k+1,\ldots,n\}}} \Big(\frac{u}{T^{d-(n+k)/2}\|\|g_n\|_{\mathcal{J}}\|_\infty}\Big)^{2/(2k+|\mathcal{J}|)}\Big).
\end{multline*}
\end{cor}

Similarly as in the result for $I_t^{(d)}$, the exponents in the tail estimate vary from $2$ (for $k=0$, $|\mathcal{J}| = 1$) to $1/d$ (for $k = n=d$, $|\mathcal{J}| = 0$, i.e., for the unique partition of the empty set). Note that the part corresponding to exponent 2 can be equivalently (up to the value of the constant $c_d$) written as
\begin{displaymath}
  \frac{u^2}{\sum_{n=1}^d T^{2d-n}\|g_n\|_2^2}.
\end{displaymath}
By \eqref{eq:U-stati-variance}, the denominator is (again up to a constant depending only on $d$) equivalent to the variance of $U_T(g)$. Clearly, we can apply Corollary \ref{cor:U-stat-tail} to estimate just the tail of $U_T(g)$, both in the case of fixed $g$, independent of $T$, and for $g$ changing with $T$ (as in the case of local $U$-statistics). It turns out that in many situations we are then able to capture the CLT behaviour in a growing window of parameters $u$ as $T\to \infty$. We will discuss this in more detail when presenting examples in Section \ref{sec:examples}.

\begin{remark} To conclude this section, let us briefly compare the assumptions of our inequalities with those from other articles. Our estimates are written in terms of the quantities $\|\|g_n\|_\mathcal{J}\|_\infty$, which in general may be difficult to estimate, but in certain situations can be handled. The recent papers \cite{schulte2023moderate,bonnet2024concentration} impose assumptions on the growth of integrals of \emph{contracted} products of the type $(|g_{n_1}| \otimes\cdots \otimes |g_{n_m}|)_{\sigma}$, or $(|g^{\otimes m}|)_\sigma$, where $\sigma$ is a partition indicating how certain arguments of the tensor product are glued together (we refer to \cite{schulte2023moderate,bonnet2024concentration} for details). Despite the appearance of partitions in their estimates and in ours, their role is quite different. The truncated products appear in \cite{schulte2023moderate,bonnet2024concentration} as elements of combinatorial expansions of moments or cumulants of $U$-statistics, and assumptions on their growth are reminiscent of the basic condition appearing in Bernstein inequality for sums of independent random variables. While our tail estimates in the end also rely on moment inequalities, they are related to more analytic methods going back to the proof of moment estimates for Gaussian chaoses due to Latała \cite{MR2294983} and subsequent moment inequalities for $U$-statistics from \cite{MR2294982}.

In particular examples of local nonnegative kernels appearing in stochastic geometry, the assumptions from \cite{schulte2023moderate,bonnet2024concentration} can be reduced to certain pointwise estimates on the kernel in the spirit of those introduced in \cite{MR3849811} (where they were used in combination with a Poisson version of convex distance inequality due to Reitzner \cite{MR3151752}). The advantage of this approach is the possibility of getting some estimates on constants in the inequalities in terms of the degree $d$ of the $U$-statistic. In some of these examples (see Section \ref{sec:examples}) we can also estimate the norms in our inequalities and it turns out that the estimates simplify, not all terms in the exponent in the second inequality of Corollary \ref{cor:U-stat-tail} are important. As we will see in applications, our inequalities and those from the other papers are not directly comparable, while in some cases ours perform in a better way, especially when one takes into account the dependence of the estimates on the variance of the $U$-statistic, the advantage of \cite{bonnet2024concentration} is the Poisson correction for large $u$, while the nice feature of \cite{schulte2023moderate} is a good and explicit subgaussian constant, especially for small $d$.

Another approach, used in \cite{MR3485348,MR4372142}, is to exploit certain self-bounding properties of nonnegative $U$-statistics, combined with various functional inequalities. This approach allows to obtain concentration inequalities in a simpler way than via our methods or those described above, but does not allow to easily catch the dependence of the upper tail on the variance, it relies instead on the expectation.

In the general case, when the function $g$ may change signs, as explained above, our inequalities, despite having a complicated form, seem to better correspond to the limiting behaviour of the $U$-statistic as they correspond to two-sided estimates on tails of multiple Gaussian integrals, which are their distributional limits. It also seems that in the general case the assumptions of \cite{schulte2023moderate,bonnet2024concentration}, which rely on passing to absolute values of the function $g$, will not allow to capture the cancellations in the $U$-statistics well enough to recover the tail behaviour of the limiting Gaussian integral.

\end{remark}

\subsection{Laws of the iterated logarithm}\label{sec:LIL}

The classical Strassen's Law of the Iterated Logarithm asserts that if $X_1,X_2,\ldots$ are i.i.d. random variables with $\E X_i = 0$ and $\E X_i^2 = \sigma^2$, then, denoting $S_n = X_1+\ldots+X_n$, with probability one,
\begin{align}\label{eq:Strassen's-LIL}
  \limsup_{n\to \infty} \frac{|S_n|}{\sqrt{2n\loglog n}} = \sigma.
\end{align}
Moreover, the cluster set of the sequence $\frac{|S_n|}{\sqrt{2n\loglog n}}$ is almost surely equal to the interval $[-\sigma,\sigma]$. One can also prove that if the right-hand side of \eqref{eq:Strassen's-LIL} is finite with positive probability, then $\E X_i = 0$, $\E X_i^2 < \infty$.

This result has been generalized in multiple directions, to random variables without second moments (with necessarily a different normalization), to the Banach space setting, to $U$-statistics in independent random variables. In continuous time, Gnedenko proved a version for L\'{e}vy processes $(X_t)_{t\ge 0}$ such that $\E X_1 = 0$, $\E X_1^2 < \infty$. This version allows to obtain in particular the Law of the Iterated Logarithm for $I_t^{(1)}(g)$ with $g \in L_2(\XX)$.

Our main result in this section is the following theorem, providing the Law of the Iterated Logarithm for multiple stochastic integrals of arbitrary order.

\begin{theorem}\label{thm:LIL-integrals} Assume that $g \in L_2^s(\XX^d)$. Then with probability one
\begin{align}\label{eq:bounded-LIL-integrals}
  \limsup_{t\to \infty} \frac{|I^{(d)}_t(g)|}{(2 t\log\log t)^{d/2}} < \infty
\end{align}
and
\begin{align}\label{eq:compact-LIL-integrals}
  \underset{t \to \infty}{\mathbf{lim\; set}} \Big\{ \frac{1}{(2t \log \log t)^{\frac{d}{2}}} I^{(d)}_t(g) \Big\} = \Big\{ \int_{\XX^d} g(x_1,...,x_d)\prod_{k=1}^d \varphi(x_k) & d\lambda^{\otimes d}(x_1,...,x_d)  :\\
   &\varphi \in L_2(\XX), \; \|\varphi\|_2 \le 1\Big\}.\nonumber
\end{align}
\end{theorem}

Recall that the set on the right-hand side of \eqref{eq:compact-LIL-integrals} has been already introduced in \eqref{eq:cluster-set-intro} and denoted by $\mathcal{K}_g$. We will use this notation in the subsequent part of the article.

As a corollary to the above result, Theorem \ref{thm:Reitzner-Schulte} and decomposition \eqref{eq:U-stat-chaos-expansion}, we obtain the following Law of the Iterated Logarithm for Poisson $U$-statistics.

\begin{cor}\label{cor:Ustat-LIL}
  Assume that $g \in L_1^s(\XX^d)$, $g \not\equiv 0$, and $\E |U(g)|^2 < \infty$. Let $g_n$ be the functions given by \eqref{eq:U-stat-chaos-kernels}. Let $m = \min\{1\le n \le d\colon g_n \not\equiv 0\}$. Then, with probability one,
  \begin{align}\label{eq:bounded-LIL-Ustat}
    \limsup_{t\to \infty} \frac{|U_t(g)-\E U_t(g)|}{t^{d-m/2}(2\loglog t)^{m/2}} < \infty
  \end{align}
  and
\begin{multline}\label{eq:compact-LIL-Ustat}
  \underset{t \to \infty}{\mathbf{lim\; set}} \Big\{ \frac{U_t(g) - \E U_t(g)}{t^{d-m/2}(2\loglog t)^{m/2}} \Big\}\\
  = \binom{d}{m} \Big\{ \int_{\XX^m} g_m(x_1,...,x_m)\prod_{k=1}^m \varphi(x_k) d\lambda^{\otimes m}(x_1,...,x_m) : \|\varphi\|_2 \le 1\Big\}.
\end{multline}
\end{cor}

\begin{remark}
  The form of the cluster set in \eqref{eq:compact-LIL-integrals} and \eqref{eq:compact-LIL-Ustat} is analogous as in the compact Law of the Iterated Logarithm for $U$-statistics in independent random variables with square integrable kernels (\cite{MR1348376}, see also the monograph \cite{MR1666908}). It was observed in \cite{MR1385404} that square integrability of the kernel is not necessary for the bounded LIL for $U$-statistics in independent random variables. The necessary and sufficient conditions in this case are more complicated \cite{MR1825163,MR2408582} and the form of the limit set beyond the square integrable case is in general not known (see \cite{MR2033884} for some conjectures and their proofs in the case of certain special classes of kernels).

  The same problems may be stated in the case of multiple Poisson integrals and Poisson $U$-statistics. The square integrable case is arguably the most important one, given that this is the setting in which the Wiener--It\^{o} integral is defined, note however that one may define it beyond $L_2^s(\XX)$, for example on $L_1^s(\XX^d)$. The case of $g \in L_1^s(\XX^d)\setminus L_2^s(\XX^d)$ is not covered by our results.  For Poisson $U$-statistics, one may also ask what happens beyond the $L_1^s(\XX^d)$ case, for instance if it may happen that $U(g)$ is well defined and square integrable or, more fundamentally, under what conditions on $g$ is $U(g)$ well defined. It is well known (see, e.g., \cite{MR3791470}) that for $d=1$, the necessary and sufficient conditions for almost sure absolute summability of the series defining $U(g)$ is $\int |g|\wedge 1 d\lambda < \infty$. To our best knowledge the case of general $U$-statistics is not known. We remark that similar questions were studied in detail for multiple stochastic integrals with respect to certain classes of L\'evy processes in the monograph \cite{MR1167198}.
\end{remark}

  \begin{remark}
  Note that if $g \ge 0$ and $g$ is not identically equal to zero, then $m$ in Corollary \ref{cor:Ustat-LIL} equals to one and as a consequence it is the first term in the chaos expansion, which determines the asymptotic behaviour of the centered Poisson $U$-statistic $U_t(g)- \E U_t(g)$. This is in contrast with the case of classical $U$-statistics, whose behaviour, even in the nonnegative case, may be governed by higher order terms in the corresponding Hoeffding expansion. This observation has consequences for functionals coming from stochastic geometry, which we will consider in the next section.
  \end{remark}

\section{Applications}\label{sec:examples}

In this section we will present specific applications of our abtract results stated above. Depending on the example we will focus more on the Law of the Iterated Logarithm or on exponential inequalities. Our goal is to give an overview of possible applications and improvements/complements of known results, but for simplicity we will not consider any of the examples in the greatest possible generality, even if our results can be applied in such setting. In some cases (in particular for subgraph counts) we will provide a detailed comparison with existing results, in others, especially if they are similar to previously described situations, we will only sketch it briefly.

\subsection{Subgraph counts in random geometric graphs}\label{sec:subgraph-counts}
Assume that $\XX$ is a Polish space equipped with a metric $\rho$, $\mathcal{X}$ is the Borel $\sigma$-field and $\widehat{\eta}$ is simple. The Gilbert graph $G_{t,r_t}$ related to $\eta$ is a random geometric graph in which the vertices are points of the Poisson process $\eta_t$ and two distinct vertices are connected with an edge iff their distance is at most $r_t$.

Let $H = ([d],E)$ be a fixed connected graph on $d\ge 2$ vertices. For $t>0$ by $S_t$ we will denote the number of homeomorphic copies of $H$ in $G_{t,r_t}$. Recalling the definition \eqref{eq:symmetrization-definition} of the symmetrization of a function, we may observe that
\begin{displaymath}
  S_t = \frac{1}{Aut(H)} \sum_{(x_1,\ldots,x_d)\in (\eta_t^d)_{\neq}} g_{H,r_t}(x_1,\ldots,x_d) = \frac{1}{Aut(H)}U_t(g_{H,r_t}^{sym}),
\end{displaymath}
with $g_{H,r_t}\colon \XX^d \to \R$ given by
\begin{align}\label{eq:formula-for-g}
  g_{H,r_t}(x_1,\ldots,x_d) = \prod_{\{k,l\}\in E} \ind{\rho(x_{k},x_{\ell})\le r_t},
\end{align}
where $Aut(H)$ is the number of automorphism of $H$.

The statistic $S_t$ has been extensively studied in the literature with focus on laws of large numbers, central limit theorems and concentration inequalities. Using Corollary \ref{cor:Ustat-LIL} we may complement existing limit theorems with a LIL.

\begin{prop} Assume that $r_t = r$ is independent of $t$ and that $\E S_1^2 < \infty$. Let
\begin{displaymath}
\sigma = \frac{d}{Aut(H)} \sqrt{\int_\XX \Big(\int_{\XX^{d-1}} g_{H,r}^{sym} d\lambda^{\otimes (d-1)}(x_1,\ldots,x_{d-1})\Big)^2d\lambda(x_d)}.
\end{displaymath}
Then with probability one,
\begin{displaymath}
  \liminf_{t \to \infty} \frac{S_t - \E S_t}{t^{d-1/2}\sqrt{2\log \log t}} = - \sigma \textrm{ and } \limsup_{t \to \infty} \frac{S_t - \E S_t}{t^{d-1/2}\sqrt{2\log \log t}} = \sigma.
\end{displaymath}
\end{prop}

We remark that in \cite{MR3849811}  Bachmann and Reitzner proved that for $\XX = \R^N$, if $S_1 < \infty$ a.s., then it has all moments, in particular the assumptions of the above proposition are automatically satisfied.

Bachmann and Reitzner obtained also concentration inequalities for $S_t$, which were later complemented in various specialized settings by Schulte and Th\"ale \cite{schulte2023moderate} and Bonnet and Gusakova \cite{bonnet2024concentration}.

We will now present concentration estimates for $S_t$, which can be obtained from our Corollary \ref{cor:U-stat-tail}. Next we will specialize to the case, when $\lambda$ is the Lebesgue measure restricted to a convex body $W$ and in this special setting we will compare our results with those of \cite{MR3849811,schulte2023moderate,bonnet2024concentration}.

Our goal is to apply Corollary \ref{cor:U-stat-tail} to $U_t(g_{H,r_t}^{sym})$. To this end we need to bound $\|\|(g_{H,r_t}^{sym})_n\|_\mathcal{J}\|_\infty$ (recall \eqref{eq:U-stat-chaos-kernels}). We have the following estimate.
\begin{prop}\label{prop:subgraph-norms} For $x \in \XX$ and $r > 0$ let $B(x,r) = \{y \in \XX\colon \rho(x,y) \le r\}$ be the ball with radius $r$ and center $x$.  Define
\begin{displaymath}
A_{r,k} = \int_\XX \lambda(B(x,r))^kd\lambda(x),
\end{displaymath}
and
\begin{displaymath}
  B_r = \esssup _{x \in \XX}  \lambda(B(x,r)).
\end{displaymath}
Then for $n = 1,2,\ldots,d$,
\begin{align}\label{eq:subgraphs-HS}
  \|(g_{H,r_t}^{sym})_n\|_{\{[n]\}} \le \sqrt{A_{dr_t, 2d - n-1}} \le B_{dr_t}^{d- \frac{n+1}{2}}\sqrt{\lambda(\XX)}
\end{align}
and for $\mathcal{J} \in \mathcal{P}_{[n]}$, with $|\mathcal{J}|\ge 2$,
\begin{align}\label{eq:subgraphs-finer-partitions-k=0}
 \|(g_{H,r_t}^{sym})_n\|_{\mathcal{J}} \le B_{dr_t}^{d-\frac{n}{2}}.
\end{align}

Moreover for $n=1,\ldots,d$, $k = 1,\ldots,n$ and $\mathcal{J} \in \mathcal{P}_{k+1,\ldots,n}$,
\begin{align}\label{eq:subgraphs-finer-partitions}
  \| \|(g_{H,r_t}^{sym})_n\|_{\mathcal{J}}\|_\infty \le B_{dr_t}^{d - \frac{n+k}{2}}.
\end{align}
\end{prop}

The proof of this proposition is presented in Section \ref{sec:proofs-subgraphs}.
Let us now apply it to obtain concentration inequalities. First, note that \eqref{eq:subgraphs-HS} implies that

\begin{multline}\label{eq:subgraph-variance-upper-bound}
  \Var(U_t(g_{H,r_t}^{sym})) = \sum_{n=1}^d n!\binom{d}{n}^2t^{2d-n}\|(g_{H,r_t}^{sym})_n\|_{\{[n]\}}^2 \le C_d \sum_{n=1}^d t^{2d-n} A_{dr_t,2d-n-1}\\
\le C_d \lambda(\XX) \sum_{n=1}^d t^{2d-n}B_{dr_t}^{2d-n-1} \le \widetilde{C}_d \lambda(\XX)\Big(t^{2d-1} B_{dr_t}^{2d-2} + t^d B_{dr_t}^{d-1}\Big),
\end{multline}
where in the last inequality we used the fact that depending on whether $tB_{dr_t}$ exceeds 1 or not, the sequence $t^{2d-n}B_{dr_t}^{2d-n-1}$ is decreasing or increasing in $n$.

We may apply the estimate of the Corollary \ref{cor:U-stat-tail}, gather together all the terms involving the norms $\|(g_{H,r_t}^{sym})_n\|_{\{[n]\}}$, $n=1,\ldots,d$ and use \eqref{eq:subgraphs-finer-partitions-k=0} and \eqref{eq:subgraphs-finer-partitions} to estimate the remaining ones. In the exponent we will get (up to the constant $c_d$) the minimum of $\frac{u^2}{\Var(S_t)}$ and expressions of the form
\begin{align}\label{eq:list-of-terms}
  \Big(\frac{u}{(tB_{dr_t})^{d- (n+k)/2}}\Big)^{\frac{2}{2k+|\mathcal{J}|}},
\end{align}
where $n=1,\ldots,d$, $k = 0,\ldots,n$, $\mathcal{J} \in \mathcal{P}_{\{k+1,\ldots,n\}}$ and $(n,k,\mathcal{J}) \neq (n,0,\{[n]\})$.
Moreover, we can assume that all these expressions are greater than one, otherwise the inequality may be made trivial (the right-hand side greater than one), by an adjustment of the constant $c_d$. For this reason, for fixed $k$ and $n$, the minimal value in \eqref{eq:list-of-terms} will be obtained for the largest possible $|\mathcal{J}|$, i.e., for $|\mathcal{J}| = n - k$. In this case the expression in \eqref{eq:list-of-terms} equals $(tB_r) (\frac{u}{(tB_{dr_t})^d})^{2/(k+n)}$. Depending on whether $u/(tB_{dr_t})^d$ exceeds one or not, the minimal value is achieved for $n=k=d$ or $n=k=1$ (and also $n=2,k=0$). We therefore obtain the following proposition.

\begin{prop}\label{prop:subgraph-tail}
Let $H$ be a connected graph on $d$ vertices. Then for any $u > 0$,
\begin{align}\label{eq:subgraph-counts-bound}
  \p(|S_t - \E S_t| \ge u) &\le 2\exp\Big(-c_d\min\Big(\frac{u^2}{\Var(S_t)},\frac{u}{(tB_{dr_t})^{d-1}},u^{1/d}\Big)\Big).
\end{align}
\end{prop}

In order to compare our results with recent inequalities from \cite{schulte2023moderate,bonnet2024concentration}, we will now specialize  to the case when $\lambda = \Vol$ is the Lebesgue measure on $\R^N$, restricted to a convex body $W$, and $\rho$ is the standard Euclidean distance. Since our results and the results from each of these papers are expressed in different ways and involve different sets of parameters, a general comparison may be difficult and tedious. For this reason we will focus on the case of fixed $W$ and will allow our constants to depend on this set. Both in our approach and in \cite{schulte2023moderate,bonnet2024concentration} the dependence on geometric characteristics of W, such that the volume, inradius, etc., can be made explicit, but it would make the comparison much more delicate. Our goal here is not to pursue the greatest generality but to argue that in some ranges of $u, t, r_t$ our inequalities are better than or comparable to those in the literature, while in some other cases (when $u$ is large) they are weaker and this phenomenon can be illustrated already within the setting of fixed $W$. Let us also remark that the authors of \cite{bonnet2024concentration} in addition to the Euclidean case consider also general constant curvature spaces. We do not pursue this direction.

In our special case for $r < R_W$, where $R_W$ is a constant depending on $W$, we have
\begin{displaymath}
  B_{r} = \omega_N r_t^N,
\end{displaymath}
where $\omega_N$ is the measure of the unit Euclidean ball in $\R^N$. Hence,
an application of Proposition \ref{prop:subgraph-tail} gives
\begin{align}\label{eq:subgraphs-Euclidean}
  \p(|S_t - \E S_t| \ge u) \le 2\exp\Big(-c_{d}\min\Big(\frac{u^2}{\Var(S_t)},\frac{u}{(td^N \omega_N r_t^N)^{d-1}},u^{1/d}\Big)\Big).
\end{align}
Moreover, using the upper bound \eqref{eq:subgraph-variance-upper-bound} together with \eqref{eq:U-stati-variance} and an elementary estimation from below one shows that there exist positive constants, $c_W, C_{d,N}$ such that for $r_t < R_W$
\begin{multline}\label{eq:variance-subgraph}
  c_W\Big(t^{2d-1} (\omega_N r_t^N)^{2d-2} + t^d (\omega_N r_t^N)^{d-1}\Big)  \\
  \le \Var(S_T) \le C_{d,N} \Vol(W)\Big(t^{2d-1} (\omega_N r_t^N)^{2d-2} + t^d (\omega_N r_t^N)^{d-1}\Big).
\end{multline}
We note that the upper estimate is true for any $r_t$, but clearly is suboptimal for the uninteresting regime when $r_t$ is larger than the diameter of $W$. The constants $c_W$ and $R_W$ may be easily estimated from below in terms of the inradius of $W$ (see \cite{bonnet2024concentration} for a similar derivation). For simplicity we will not keep track of this dependence.

We can see that for fixed $W$ if $t\omega_N r_t^N \le 1$ the variance is of order $t^d (\omega_N r_t^N)^{d-1}$, which is the same as the order of $\E S_t$. In the complementary case, the variance is of the order $t^{2d-1} (\omega_N r_t^N)^{2d-2}$

Let us first compare \eqref{eq:subgraphs-Euclidean} with the inequality by Shulte and Th\"ale \cite[Corollary 4.2]{schulte2023moderate}. To do this, following the literature (see \cite{MR1986198}) we will distinguish several regimes: the \emph{sparse regime} when $t\omega_N r_t^N \to 0$, and the \emph{thermodynamic} and \emph{dense} regimes when $t\omega_N r_t^N$ tends to a constant or diverges to infinity, respectively.

In the sparse regime it is easy to see that one can drop the middle term under the minimum in \eqref{eq:subgraphs-Euclidean} and, for $t$ such that $td^N \omega_N r_t^N \le 1$, obtain
\begin{multline}\label{eq:sparse}
  \p(|S_t - \E S_t| \ge u) \le 2\exp\Big(-c_d\min\Big(\frac{u^2}{\Var(S_t)},u^{1/d}\Big)\Big) \\
  \le 2\exp\Big(-c_{W,d}\min\Big(\frac{u^2}{t^d(\omega_n r_t^N)^{d-1}},u^{1/d}\Big)\Big).
\end{multline}

In this case the result by Shulte and Th\"ale is similar to the above estimate. The constant in front of the variance in \cite{schulte2023moderate} is explicit, while in front of $u^{1/d}$ in the first inequality there is a constant depending also on geometric characteristics of $W$. If we additionally assume that $\E S_t \to \infty$ (i.e., $t^d (\omega_N r_t^N)^{d-1}\to \infty$) from both \eqref{eq:sparse} and its counterpart in  \cite{schulte2023moderate} we obtain a subgaussian bound $\p(|S_t - \E S_t| \ge u\sqrt{\Var(S_t)})\le 2e^{-c_W u^2}$ in a window $u \in [0,u_t]$ with $u_t \to \infty$.

In the thermodynamic and dense regimes, by allowing the constants to depend on $W$, one again may drop the middle term under the minimum in \eqref{eq:subgraphs-Euclidean} (using the assumption $r_t < R_W$ in the case $td^N \omega_N r_t^N \ge 1$), and obtain for $t\omega_Nr_t^N \ge \varepsilon > 0$,
\begin{multline}\label{eq:dense}
  \p(|S_t - \E S_t| \ge u) \le 2\exp\Big(-c_{W,d}\min\Big(\frac{u^2}{\Var(S_t)},u^{1/d}\Big)\Big) \\
  \le 2\exp\Big(-c_{W,d,\varepsilon}\min\Big(\frac{u^2}{t^{2d-1}(\omega_N r_t^N)^{2d-2}},u^{1/d}\Big)\Big).
\end{multline}
In this case the inequality from \cite{schulte2023moderate} gives the same subgaussian term (again with a better constant) and instead of $u^{1/d}$ a term of the form $c_{W,d,\varepsilon} (u/(t\omega_N r_t^N)^{d-1})^{1/d}$, which in the dense regime is of worse order than in our estimate. Again, both results give a subgaussian bound on $\p(|S_t - \E S_t| \ge u\sqrt{\Var(S_t)})$ in a window of parameters $u$, which as $t \to \infty$ grows to the whole positive half-line. Our estimate gives a better growth of this window.

A comparison of our estimates with the results by Bonnet and Gusakova is more involved as they use an additional parameter $s \in [0,1]$ over which one can optimize. More precisely, \cite[Corollary 6.6]{bonnet2024concentration} yields that if  $r_t < R_{W,H}$, and $t \ge c_{W,d}(\omega_N r_t^N)^{s/d-1}$, then
\begin{align}\label{eq:BG}
  \p(|S_t - \E S_t| \ge u) \le \exp\Big(- \alpha(t)\min(A(t,u),B(t,u))\Big),
\end{align}
where
\begin{align*}
  \alpha(t) & = c_{d,W}\min(1, (\omega_N r_t^N))^{s-s/d},\\
  A(t,u) &= \frac{u^2}{t^{2d-1}(\omega_N r_t^N)^{2d - 2 - s + s/d}},\\
  B(t,u) &= \frac{u^{1/d}}{(\omega_N r_t^N)^{(s-1)/(2d)}}\Big(1+ \log_+\Big(\frac{c_{W,d}u}{t^d (\omega_N r_t^N)^{d- (s+1)/2}}\Big)\Big).
\end{align*}

If $r_t = r < R_{W,H}$ is independent of $t$, then basically any choice of $s$ leads to a bound of the form
\begin{displaymath}
  \p(|S_t - \E S_t| \ge u) \le 2\exp\Big(- c_{d,W,r} \min\Big(\frac{u^2}{\Var(S_t)},u^{1/d}\Big(1 + \log_+(c_{r}ut^{-d})\Big)\Big)\Big),
\end{displaymath}
for $t > C_{W,d,r}$, which for large $u$ improves on our estimates as it implies a bound on the tail which agrees with the behaviour of a product of independent Poisson variables.
If $r_t \to 0$, which may happen in all three regimes, the situation is more complicated. The bound one obtains from \eqref{eq:BG} is
\begin{align}\label{eq:dense-BG}
  \p(|S_t - \E S_t| \ge u)
  &\le 2\exp\Big(- c_{d,W}\min(\widetilde{A}(t,u),\widetilde{B}(t,u))\Big),
\end{align}
where
\begin{align*}
  \widetilde{A}(t,u) & = \frac{u^2}{t^{2d-1}(\omega_N r_t^N)^{2d-2 + 2s(1/d-1)}},\nonumber\\
  \widetilde{B}(t,u) & = (\omega_N r_t^N)^{s + \frac{1- 3s}{2d}}u^{1/d}\Big(1 + \log_+\Big(\frac{c_{W,d} u}{t^d (\omega_N r_t^N)^{d- (s+1)/2}}\Big)\Big).\nonumber
\end{align*}
Recall that the above bound holds for $t \ge c_{W,d}(\omega_N r_t^N)^{s/d-1}$.

Since $\omega_N r_t^N$ is eventually smaller than 1, and $d \ge 2$, for large $t$ the estimate becomes best for $s=0$.

Recalling \eqref{eq:variance-subgraph} one can see that in the dense regime we recover the variance as the subgaussian coefficient, similarly as in our inequality \eqref{eq:dense}, however the Poissonian term is multiplied by $(\omega_N r_t^N)^{1/2d} \to 0$. For $u$ which is large with respect to $t$, \eqref{eq:dense-BG} will still perform better than \eqref{eq:dense}, but for \emph{moderate} $u$ \eqref{eq:dense} may be better. Moreover, the subgaussian estimate on $\p(|S_t - \E S_t| \ge u\sqrt{\Var(S_t)})$ provided by \eqref{eq:dense} will hold in a larger window.

In the sparse regime one is not allowed to take $s = 0$ due to the restriction $t \ge c_{W,d}(\omega_N r_t^N)^{s/d-1}$. The variance of $S_t$ is of the order $t^d(\omega_Nr_t^N)^{d-1}$, while for any admissible $s$ the subgaussian coefficient in \eqref{eq:dense-BG}
equals
\begin{multline*}
t^{2d-1}(\omega_N r_t^N)^{2d-2 + 2s(1/d-1)} = t^d(\omega_Nr_t^N)^{d-1} (t (\omega_N r_t^N)^{1 - 2s/d})^{d-1} \\
\ge t^d(\omega_Nr_t^N)^{d-1} c_{W,d}^{d-1} (\omega_N r_t^N)^{-s(d-1)/d}.
\end{multline*}
The last factor on the right-hand side above diverges to infinity, so the estimate \eqref{eq:dense-BG}, contrary to \eqref{eq:sparse} does not recover the subgaussian estimate on $\p(|S_t - \E S_t| \ge u\sqrt{\Var(S_t)})$. On the other hand, one can always choose $s$ sufficiently close to 1, for which $\eqref{eq:dense-BG}$ will be applicable and then for $u$ sufficiently large with respect to $t$ one obtains a better estimate than given by \eqref{eq:sparse}.

Let us finally comment on the earlier estimates by Bachmann and Reitzner \cite{MR3849811}. They are formulated for general  measures on $\R^N$. The estimates on the upper tail $\p(S_t \ge \E S_t + u)$ instead of the variance use the median or mean of $S_t$. They also provide a subgaussian bound for small values of $u$ and a bound of the form $e^{-cu^{1/d}}$ for larger $u$. In the example discussed above the subgaussian coefficient does not match the variance of $S_t$, in general it may depend on the intensity measure $\lambda$. On the other hand, Bachmann and Reitzner provide a good estimate on the deviation of $S_t$ below the mean, of the form
\begin{displaymath}
  \p(S_t \le \E S_t - u) \le \exp\Big(-\frac{u^2}{2d\Var(S_t)}\Big).
\end{displaymath}
Such an estimate could  not be obtained by our methods, since we rely on decomposition of centered $U$-statistics into mean-zero random variables, coming from the chaos expansion and estimate deviations of absolute values of the summands. The above subgaussian bound on the lower tail explores the fact that $S_t$ is a sum of positive random variables. The phenomenon, when the lower tail of a positive random variable is lighter than the upper one is quite common in the theory of concentration of measure, starting with the classical example of sums of independent positive random variables.

Summarizing, our estimates on subgraph counts complement those existing in the literature. The problem of finding optimal inequalities in various ranges of parameters is still open. We remark that the corresponding problem in the case of Erd\H{o}s-R\'enyi random graphs attracted a lot of attention in the last two decades, leading ultimately to the formulation of nonlinear large deviation principles by Chatterjee and Dembo \cite{MR3519474}.

To finish let us point out that all the upper estimates above hold also in the case of counting induced subgraphs, since the functions in the chaos decomposition are bounded from
above by their counterpart in the included case.

\subsection{Length power functionals}\label{sec:length-power}
The subgraph count statistics considered in the previous section include as a special case the number of edges in a geometric random graph. We will now consider a family of closely related random variables known as length power functionals, which have been investigated, e.g., in \cite{MR3641808,MR3485348,bonnet2024concentration}. Following \cite{MR3485348}, we will consider a generalization of the random geometric graph, in which the range of edges may vary with the spatial location.
We will assume that $\XX = \R^N$ is equipped with Euclidean norm $| \cdot |$ and a locally finite diffuse intensity measure $\lambda$ (in particular $\widehat{\eta}$ is simple). Let $r:\R^N \to (0,\infty)$ be a bounded function. For $t > 0$ we will consider a random graph $G_t$ whose vertices are points of a Poisson process $\eta_t$. Two points  $x,y$ of $\eta_t$ are joint by an edge in $G_t$ if and only if the closed Euclidean balls  $B(x,r(x))$ and $B(y,r(y))$ have nonempty intersection. If $r(x) = r/2$ for a number $r > 0$, we obtain the Euclidean case of the Gilbert graph discussed in the previous section, but one can consider other functions $r\colon \R^N \to \infty$, e.g., $r(x) = \big( |x| + 1\big)^{-\gamma}; \gamma > 0$ -- the example investigated in \cite{MR3485348}. We will discuss functions $r$ which do not change too rapidly, i.e., satisfying the condition
\begin{align}\label{eq:A-beta}
\sup\{ r(x+v) : |v| \le 2\|r\|_\infty\} \le (\beta-1) r(x)
\end{align}
for some constant $\beta \ge 2$ and all $x \in \R^N$.
We remark that this assumption is satisfied in particular in the examples mentioned above.

For $\alpha \ge 0$ define now the function $g = g_{r,\alpha} \colon \R^N \times \R^N \to \mathbb R$ by
\begin{equation} g(x,y) = |x-y|^\alpha \ind{|x-y| \le r(x) + r(y)}
\end{equation}
and consider the $U$-statistic
\begin{displaymath}
  U_t(g) = \frac{1}{2}\sum_{(x,y) \in (\eta_t^2)_{\neq}} g(x,y).
\end{displaymath}

Note that the case $\alpha = 0$ corresponds to the edge count, while the case $\alpha = 1$ describes the total length of edges.

Our main results imply the following Law of the Iterated Logarithm and an exponential inequality, the proofs of which are presented in Section \ref{sec:length-power-proofs}.

\begin{prop}\label{prop:power-length-LIL} Assume that $r\colon \R^N\to (0,\infty)$  is bounded, satisfies the condition \eqref{eq:A-beta} and that the function $x\mapsto r^{\alpha}(x)\lambda(B(x,\beta r(x)))$ belongs to $L_1(\lambda)\cap L_2(\lambda)$. Then, with probability one,
\begin{displaymath}
\liminf_{t \to \infty}\  \frac{U_t(g) - \E U_t(g)}{t^{\frac{3}{2}} \sqrt{2 \log \log t}} = -\sqrt{ \int_{\R^N} \Big( \int_{\R^N}g(x,y)\lambda(dx) \Big)^2 \lambda(dy)}
\end{displaymath}
and
\begin{displaymath}
\limsup_{t \to \infty}\  \frac{U_t(g) - \E U_t(g)}{t^{\frac{3}{2}} \sqrt{2 \log \log t}} = \sqrt{ \int_{\R^N} \Big( \int_{\R^N}g(x,y)\lambda(dx) \Big)^2 \lambda(dy)}.
\end{displaymath}
\end{prop}

\begin{prop}\label{prop:PowerLengthIneq} For $p,\gamma \ge 0$ define
\begin{displaymath}
 A_{\gamma,p} = \int_{\R^N} r^\gamma(x) \lambda^p(B(x,\beta r(x))) \lambda(dx),\; B_{\gamma,p} = \esssup_{x \in \R^N} \Big(r^\gamma(x) \lambda^p(B(x,\beta r(x)))\Big).
\end{displaymath}
Under the assumptions of Proposition \ref{prop:power-length-LIL}, for any $t>0$ and $u>0$,

\begin{multline*} \p(|U_t(g)-\E U_t(g)| > u)\\
 \le 2\exp\Big( - c \min \Big ( \frac{u^2}{\Var (U_t(g))}, \frac{u}{t \beta^\alpha B_{\alpha,1}}, \frac{u^{\frac{2}{3}}}{t^{\frac{1}{3}} \beta^{\frac{2}{3}\alpha} B_{\frac{2}{3}\alpha,\frac{1}{3}} },\frac{u^{\frac{1}{2}}}{\beta^{\frac{\alpha}{2}} B_{\frac{\alpha}{2},0}}   \Big ) \Big).
\end{multline*}
Moreover,
\begin{displaymath}
  \Var(U_t(g)) \le C \beta^{2\alpha}(t^3 A_{2\alpha,2} + t^2 A_{2\alpha,1}).
\end{displaymath}
\end{prop}

In order to compare this result with inequalities from \cite{schulte2023moderate,bonnet2024concentration} we may again specialize to the case when $\lambda = \Vol$ is the Lebesgue measure restricted to a convex body $W \subset \R^N$ and  $r = r_t$ does not depend on $x$ but may vary with $t$. In this special case, $B_{\gamma,p} \le \omega_N^p r_t^{\gamma + Np}$ and $A_{\gamma,p} \le \Vol(W)\omega_N^pr_t^{\gamma + Np}$, where $\omega_N$ is the volume of the unit ball in $\R^N$. Using similar arguments as in the case of subgraph counts, we may eliminate from the exponential inequality the part corresponding to $u^{2/3}$ and obtain for $t,u>0$,
\begin{displaymath} \p(|U_t(g)-\E U_t(g)| > u)
 \le 2\exp \Big(-c\Big( \frac{u^2}{\Var(U_t(g))}, \frac{u}{t \omega_N r_t^{\alpha+N}}, \frac{u^{1/2}}{r_t^{\alpha/2}}\Big)\Big)
 \end{displaymath}
and $\Var(U_t(g)) \le C\Vol(W)(t^3 \omega_N^2 r_t^{2(\alpha + N)} + t^2 \omega_N r_t^{2\alpha + N})$.

Similarly as with inequalities for subgraph counts,depending on the regime of $r_t,t,u$, the above inequality may give better or worse estimates than those from \cite{bonnet2024concentration}. On the one hand, as with all our results, we do not capture the Poisson behaviour for large $u$. On the other hand, we obtain estimates with better subgaussian behaviour, especially in the sparse regime. Since the precise comparison of the estimates is quite tedious and parallels the one for subgraph counts we do not present it here and leave it for the interested reader.

\subsection{Quadratic functional of the Ornstein--Uhlenbeck L\'evy process}\label{sec:OUL}
We will now consider an example discussed in \cite{MR2537812,MR2642882,schulte2023moderate}.
An Ornstein--Uhlenbeck L\'evy process $(U_s)_{s\ge 0}$ is a process of the form
\begin{displaymath}
U_s = \sqrt{2\rho}\int_{(-\infty,s]\times \R} ue^{-\rho(s-x)}d(\eta - \mu)(x,u),
\end{displaymath}
where $\eta$ is a Poisson process on $\R^2$ with intensity $\mu = \mathcal{L} \otimes \nu$, with $\mathcal{L}$ being the Lebesgue measure on $\R$ and $\nu$ -- a measure on $\R$, satisfying the normalizing condition $\int_\R u^2 d\nu(u) = 1$. Such processes have applications in survival analysis, network modelling and finance, see the references in \cite{MR2537812}.

We will be interested in the functional
\begin{displaymath}
  Q_T = \int_0^T U_s^2 ds.
\end{displaymath}

It was shown in \cite{MR2537812} that as $T \to \infty$, $\frac{Q_T - T}{\sqrt{T}}$ converges weakly to a centered Gaussian measure. The authors of \cite{MR2642882} provided bounds on the speed of convergence in the Wasserstein distance, while in \cite{schulte2023moderate} these results were complemented with moderate deviations, a Bernstein type concentration inequality and Cram\'er type quantitative bounds on the speed of convergence to the limiting measure.

We will additionally assume that $\nu$ has compact support and set $A = \sup\{|u|\colon u \in {\rm supp \;} \nu\}$. Let also $c_\nu^2 = \int_\R u^4d\nu(u)$.

By using Theorem \ref{thm:tails-and-moments-main} one may obtain a version of the concentration inequality from \cite{schulte2023moderate} with an increased range of applicability. More precisely, we have the following proposition, the proof of which is deferred to Section \ref{sec:OUL-proofs}.

\begin{prop}\label{prop:OUL}
For any $T > 0$ and $u > 0$,
\begin{align}\label{eq:prop-3.6}
  \p(|Q_T - E Q_T|\ge u) \le 2\exp\Big(-c\min\Big(\frac{u^2}{\Var(Q_T)},\frac{u^{2/3}}{A^{2/3}\rho^{-1/3}},\frac{u^{1/2}}{A}\Big)\Big).
\end{align}
Moreover, $\Var(Q_T) \le (c_\nu^2 + 2\rho^{-1})T$.
\end{prop}
Let us briefly comment on the differences between our estimate and \cite[Corollary 4.5]{schulte2023moderate}. The estimate therein, instead of our parameter $A$, uses a number $M\ge 1$, which is assumed to satisfy a bound $\int_\R |u|^m d\nu(m) \le M^m$ for all positive integers $m$. It is easy to see that such an $M$ exists if and only if $\nu$ is compactly supported and in this case $M \ge A$. The estimate from \cite{schulte2023moderate} works under an additional assumption that $T > 1/\rho$ and contains the subgaussian part and a part of the form $(u b_n)^{1/2}$, where $b_n$ depends on $T, c_\nu, M$ and $\rho$ through a quite involved formula. One can show that $b_n$ is bounded from above by $\frac{c}{M^8 (\max(1,2/\rho))^3}$, for some explicit absolute constant $c$, which demonstrates that the dependence on the parameters $\rho$ and $A$ (being a lower bound on $M$) in our estimate is better. On the other hand, the multiplicative constants in \cite{schulte2023moderate} are explicit. To conclude the discussion, let us point out that for large $T$, both estimates give subgaussian bounds in a growing window of parameters $u$, agreeing (up to the value of constant) with the central limit theorem for $\frac{Q_T - T}{\sqrt{T}}$.

\subsection{Poisson flat process} Another process, which has been thoroughly studied in stochastic geometry, is the Poisson flat process, see, e.g., the monograph \cite{MR2455326} or articles \cite{MR2244437,MR3161465,MR3215537,bonnet2024concentration}. For positive integers $k< N$ let $G(N,k)$ be the Grassmann manifold of all $k$-dimensional linear subspaces of $\R^N$ and let $A(N,k)$ be the space of all $k$-dimensional affine subspaces of $\R^N$ (which we will call $k$-flats or simply flats). For a probability measure $Q$ on $G(N,k)$ let $\lambda$ be a measure on $A(N,k)$ given by
\begin{displaymath}
  \lambda(\cdot) = \int_{G(N,k)} \int_{E^\perp}\ind{x+E \in \cdot}d H_{N-k}(x)dQ(E),
\end{displaymath}
where $H_{N-k}$ is the $(N-k)$-dimensional Hausdorff measure. This defines an intensity of a Poisson process $\eta$ on $A(N,k)$, which is invariant under translations of flats. For $Q$ being the uniform measure on $G(N,k)$, $\eta$ is invariant also under rotations. We refer to \cite[Chapter 4.4]{MR2455326} for more details on Poisson flat processes.

Among the functionals of the Poisson flat process of interest in stochastic geometry are $U$-statistics of the form $U_t(g)$, where
\begin{displaymath}
  g(H_1,\ldots,H_d) = \psi(W\cap H_1\cap\cdots\cap H_d),
\end{displaymath}
for a convex body $W$ and a function $\psi$ on the space of all convex compact subsets of $\R^N$ (we note that one can show that the intersection $H_1\cap \cdots \cap H_d$ of distinct flats from the process $\eta_t$ is almost surely of dimension $N - d(N-k)$ or empty, see \cite[p. 130]{MR2455326} and \cite[Example 3]{MR3215537}). A typical example is $\psi(L) = V_i(L)^\beta$, where $V_i$ is the $i$-th intrinsic volume, $i=0,\ldots,N$ (see, e.g., \cite[Chapter 14.2]{MR2455326}) and $\beta > 0$. In particular, for $i=0$, $V_0$ is the Euler characteristic, equal to one for convex bodies (so $U_t(g)$ counts the number of intersections of $d$ distinct flats of $\eta_t$ in $W$), for $i= N-d(N-k)$ and $L$ of dimension $i$ we obtain the $N-d(N-k)$ dimensional Haussdorff measure of $L$ and for $i = N - d(N-k)-1$ one half of the $(i-1)$-dimensional Haussdorff measure of the boundary of $L$ in the $i$-dimensional affine space spanned by $L$.

Taking into account that $\lambda$ restricted to the set of $k$-flats intersecting $W$ is finite (so square integrability implies integrability), one can use Corollary \ref{cor:Ustat-LIL} to get the following result.

\begin{prop}\label{prop:k-flats-LIL}
In the above setting, assume that $\psi$ is nonnegative and
\begin{displaymath}
  \int_{A(N,k)^d} \psi(W\cap H_1\cap\cdots\cap H_d)^2 d\lambda^{\otimes d}(H_1,\ldots,H_d) < \infty.
\end{displaymath}
Set
\begin{displaymath}
\sigma^2 = d\int_{A(N,k)} \Big(\int_{A(N,k)^{d-1}} \psi(W\cap H_1\cap\cdots\cap H_d)d\lambda^{\otimes(d-1)}(H_1,\ldots,H_{d-1})\Big)^2 d\lambda(H_d).
\end{displaymath}
Then, with probability one,
\begin{displaymath}
  \limsup_{t\to \infty} \frac{U_t(g) - \E U_t(g)}{t^{d-1/2}\sqrt{2\log \log t}} = \sigma \textrm{ and } \liminf_{t\to \infty} \frac{U_t(g) - \E U_t(g)}{t^{d-1/2}\sqrt{2\log \log t}} = -\sigma.
\end{displaymath}
\end{prop}

This result complements central limit theorems obtained for $U_t(g)$ in \cite{MR2244437,MR3215537}. If $\psi$ corresponds to intrinsic volumes then the assumptions of the above proposition are satisfied (in fact $g$ is bounded, see \cite{MR3215537}).

We remark that in the rotationally invariant case with $\psi = V_i$ one can simplify the expression for $\sigma^2$, using Crofton's formula \cite[Theorem 5.1.1]{MR2455326}. It turns out that (up to a prefactor related to $i,k,N,d$), $\sigma^2$ is equal  to $\int_{A(N,k)} V_{i+(d-1)(N-k)}(W\cap H)^2 d \lambda(H)$ (see \cite{bonnet2024concentration} for details). Similarly, $\E U_t(g)$ is proportional to $t^d V_{i + d(N-k)}(W)$ .

As for the exponential inequality, even though one can in principle apply our results, getting estimates on the norms $\|\|g\|_\mathcal{J}\|_\infty$ seems rather involved. In the rotationally invariant case, for $\psi = V_i$ and $k = N-1$ an exponential inequality has been recently derived in \cite{bonnet2024concentration}. As with other results from this paper, described in more detail above, it provides a subgaussian estimate for small values of the parameters and an exponent of the form $-u^{1/d}\log(u)$ for large values of $u$. Since in this case, for fixed kernel $g$, the best we can hope to obtain from Corollary \ref{cor:U-stat-tail} is a subgaussian estimate for small values and an exponent $-u^{1/d}$ for large values of $u$, we do not try to obtain explicit estimates.

Let us also mention that in \cite{MR4635692,bonnet2024concentration} the authors consider $k$-flats processes in constant curvature spaces. A counterpart of Proposition \ref{prop:k-flats-LIL} is clearly available also in this setting.

\subsection{Local $U$-statistics of marked processes} In the examples described in previous sections, the functions $g$ generating the $U$-statistic and as a consequence integrands in stochastic integrals involved were nonnegative. In this case the norms $\|\cdot\|_\mathcal{J}$ decrease if one decreases the values of the function. In general this may not be the case. Moreover, asymptotically for nonnegative kernels as $t\to \infty$ the dominating part of the $U$-statistics comes from the first chaos and as a consequence the normalization in the LIL is  $t^{d-1/2}\sqrt{2t\log\log t}$. In this section we will present a few examples of $U$-statistics based on kernels of varying sign (focusing for simplicity on the case of $U$-statistics of order $2$) which may be of interest from the point of view of statistics or modelling of physical or social phenomena. We will not describe in detail any particular model, the goal is to just sketch potential applications of our results. We would also like to argue that they may be applied in the context of marked processes, which are of interest in applications. For this reason we will now consider the space $\XX \times \mathbb{M}$, where $\mathbb{M}$ is the space of \emph{marks}, which is endowed with a $\sigma$-field $\mathcal{M}$.

We will assume that the Poisson point process $\widehat{\eta}$ on $\XX\times [0,\infty)$ is proper, of the form $\widehat{\eta} = \sum_{n=1}^\kappa \delta_{(X_n,T_n)}$ , and that to each of its points $(X_i,T_i)$, $i \le \kappa$, one attaches a random mark $Y_i$ in such a way that conditionally on $\widehat{\eta}$, $Y_i$'s are independent and $Y_i$ is distributed according to the measure $K(X_i,\cdot)$, where $K\colon \mathcal{X} \times \mathcal{M} \to [0,1]$ is a kernel, i.e., for fixed $A \in \mathcal{M}$, $K(\cdot,A)$ is measurable and for fixed $x \in \XX$, $K(x,\cdot)$ is a probability measure on $\mathcal{M}$. Note that the distribution of the mark $Y_n$ does not depend on the time $T_n$, but only on the \emph{spatial} location $X_n$. The Marking Theorem  \cite[Theorem 5.6]{MR3791470} asserts that $\vartheta = \sum_{n=1}^\kappa \delta_{(X_n,Y_n,T_n)}$ is then a Poisson point process on $\XX \times \mathbb{M}\times [0,\infty)$ with intensity given by $\mu\otimes \mathcal{L}$, where $\mu(B) = \lambda \otimes K (B) := \int_\XX K(x,\{y\colon (x,y) \in B\})d\lambda(x)$. We can therefore apply our results to the new space $\XX\times \mathbb{M}$ and obtain the LIL or concentration inequalities. In many situations the kernel $K$ is independent of $X$, which simplifies the setting and corresponds to marks assigned independently also of the spatial location of the points.
Since for $g \colon (\XX \times\mathbb{M})^2 \to \R$ we have
\begin{displaymath}
\E U(g^{sym},\vartheta_t) = t^2 \int_{(\XX\times  \mathbb{M})^2} g(x_1,y_1,x_2,y_2)d\mu(x_1,y_1)d\mu(x_2,y_2),
\end{displaymath}
the normalized $U$-statistic $t^{-2}U_t(g)$ may be considered as an unbiased estimator of the quantity $\int_{(\XX\times  \mathbb{M})^2} g(x_1,y_1,x_2,y_2)d\mu(x_1,y_1)d\mu(x_2,y_2)$ in a model in which the data points are acquired according to the homogeneous in time, but not necessarily in space, Poisson point processes. Of course, this encompasses as the special case estimators of quantities of the type $\int_{\XX^2} f(x_1,x_2)\lambda(dx_1)\lambda(dx_2)$ for $f\colon \XX^2 \to \R$.

The LIL describes then the almost sure limiting behaviour of the estimator as $t\to \infty$ (which admittedly is of rather theoretical than practical interest) while the exponential inequalities give confidence intervals for finite $t$.

To give some simple examples, consider, e.g., the situation where $\XX$ is endowed with a metric $\rho$, and $\mathcal{M} = \{+1,-1\}$ and let $K(x,\cdot) = p_x\delta_1 + q_x\delta_{-1}$, where $q_x = 1 - p_x$. In this case the mark of a point may be considered a \emph{spin} or a \emph{belief/preference}.

Setting
\begin{displaymath}
  g(x_1,y_1,x_2,y_2) = \ind{\rho(x_1,x_2)\le r}y_1y_2,
\end{displaymath}
we see that $U_t(g)$ may be interpreted as a measure of agreement of beliefs between the vertices of the random geometric graph. If say $\XX$ is the sphere and $\lambda$ is the uniform measure, then the corresponding $U$-statistic after an appropriate normalization may be used to estimate the probability that two randomly chosen individuals at distance at most $r$, choosing their preferences independently and according to the measures assigned to their locations, agree.

The examples presented above are very simple, and they serve only illustrative purposes, we would however like to remark that there are many more involved statistical applications of Poisson $U$-statistics with kernels of varying sign, related for instance to cosmology. The corresponding kernels are obtained from needlet or wavelet expansions of various functionals of interest. One may hope that concentration inequalities will lead to non-asymptotic guarantees in corresponding statistical problems. The setting for such applications is too involved to describe it here, we refer the reader to \cite{MR3447000,MR3770838,MR3585404}. Other statistical applications of Poisson $U$-statistics, in particular to testing the intensity of Poisson processes on the line have been presented in \cite{MR2779402}.

Let us finally note that if $g$ is not of fixed sign, it may happen that some elements of its chaos expansion vanish and the normalization in the LIL for $U(g,\vartheta_t)$ may be $2t\log \log t$ and not $t^{3/2}\sqrt{2\loglog t}$ as in the examples from previous sections. This happens for instance if $p_x = 1/2$ independently of $x$ or in an even simpler, unmarked example, if $\XX = [-1,1]$, $\lambda$ is the Lebesgue measure and $g(x_1,x_2) = \mathrm{sgn}(x_1x_2)$ (the latter example being discussed in \cite{MR3161465}). A similar phenomenon at the level of weak limit theorems has been thoroughly discussed  in \cite{MR3096352}, where it was proven that Poisson $U$-statistics converge in distribution to corresponding Gaussian multiple stochastic integrals, as is the case for $U$-statistics in independent random variables.

More generally, in the context of marked processes, Corollary \ref{cor:Ustat-LIL} gives, e.g., the following proposition.
\begin{prop}
In the setting described above, assume that for all $x \in \XX$, $K(x,\cdot) = \nu(\cdot)$ for some probability measure $\nu$ on $\mathbb{M}$. Let $h\colon \mathbb{M}^2\to \R$ be a symmetric function, such that $\int_{\mathbb{M}^2} h(y_1,y_2)^2 d\nu(y_1)d\nu(y_2) < \infty$ and for all $y_1$ in $\mathbb{M}$, $\int_\mathbb{M} h(y_1,y_2)d\nu(y_2) = 0$. Consider any $r \in (0,\infty]$ and set $g(x_1,y_1,x_2,y_2) = h(y_1,y_2)\ind{\rho(x_1,x_2) \le r}$.
If
\begin{displaymath}
  \int_{\XX} \lambda(B(x,r))d \lambda(x) < \infty,
\end{displaymath}
then with probability one
\begin{align*}
\limsup_{t \to \infty}\frac{|U(g,\vartheta_t)|}{2t\loglog t} = \|h\|_{\{1\}\{2\}}\|\ind{\rho(\cdot,\cdot) \le r}\|_{\{1\}\{2\}} < \infty
\end{align*}
and the limit set of $\frac{U(g,\vartheta_t)}{2t\loglog t}$ is almost surely equal to
\begin{displaymath}
{\rm Conv}\Big(\{u_1 u_2\colon u_1 \in \sigma(h), u_2 \in \sigma(\ind{\rho(\cdot,\cdot) \le r})\cup\{0\}\Big),
\end{displaymath}
where $\sigma(h)$, $\sigma(\ind{\rho(\cdot,\cdot) \le r})$ are the spectra of integral Hilbert--Schmidt operators defined by the corresponding square integrable functions and ${\rm Conv}$ denotes the convex hull.
\end{prop}

The above proposition follows directly from the form of the limiting set in Corollary \ref{cor:Ustat-LIL} together with the relation between the sets $\mathcal{K}_g$ and spectra for symmetric Hilbert-Schmidt operators (using the spectral theorem it is not difficult to show that $\mathcal{K}_g$ is the convex hull of the union of the spectrum and $\{0\}$) and the form of the spectrum of the tensor product of linear operators.

\section{Proofs of main results}\label{sec:proofs}

\subsection{Preliminaries}

We will carry out the proofs under an additional assumption that the measure $\lambda$ is non-atomic, i.e., for all $A \in \mathcal{X}$ with $\lambda(A) > 0$ there exists a measurable $B \subset A$, such that $0 < \lambda(B) < \lambda(A)$. Although the reduction to this case seems quite standard, it is somewhat technical, moreover we have not been able to find in the literature a comprehensive reference corresponding to our setting. We will therefore provide a full argument justifying it in Appendix \ref{sec:reductions}, which also contains a demonstration of existence of a c\`adl\`ag version of the process $(I_t(g))_{t\ge 0}$ for which our results are stated.

The advantage of the non-atomic case stems from the fact that one can approximate general functions $g \in L_2^s(\XX^d)$ by a special class of step functions as described below.

Fix a positive integer $d$ and let $\mathcal{E}_d$ be the family of all functions of the form
\begin{align}\label{eq:step-function}
h = \sum_{i_1,...,i_d=1}^R a_{i_1,...,i_d} \cdot \Ind{ \prod_{j=1}^d A_{i_j}},
\end{align}
where
\begin{itemize}
  \item[(i)] $R \in \N$,
  \item[(ii)] $A_i$, $i \le R$, are pairwise disjoint and $\lambda(A_i) < \infty$,
  \item [(iii)]  $a_{i_1,...,i_d}=0$ whenever $i_k = i_j$ for some $k \neq j$.
\end{itemize}
By $\mathcal{E}_d^s$ we will denote the subspace of $\mathcal{E}_d$ consisting of symmetric functions, i.e., functions of the form \eqref{eq:step-function} such that
\begin{itemize}
\item[(iv)] the coefficients $a_{i_1,...,i_d}$ are symmetric under permutation of indices.
\end{itemize}

Using formula \eqref{eq:L_1-integral-non-symmetric} for $\eeta_t$, together with the observation that on a product of pairwise disjoint sets $\eeta^{(d)}$ coincides with $\eeta^{\otimes d}$, we obtain that for $C_1,\ldots,C_d \subseteq \XX$, pairwise disjoint and of finite measure $\lambda$,
\begin{align}\label{eq:integral-disjoint-product}
  I^{(d)}_t(\Ind{C_1\times \cdots \times C_d}) = \prod_{i=1}^d(\eeta(C_i\times[0,t])-t\lambda(C_i)).
\end{align}

Thus, for a function $h \in \mathcal{E}_d$ of the form \eqref{eq:step-function} we have
\begin{align}\label{eq:step-function-integral}
  I^{(d)}_t(h) = \sum_{i_1,...,i_d=1}^{R} a_{i_1,...,i_d}\prod_{j=1}^d \Big(\eeta\Big(A_{i_j}\times [0,t]\Big) - t\lambda(A_{i_j})\Big).
\end{align}

Our proofs will be based on the above explicit formula, together with the following approximation result, the proof of which is deferred to Appendix \ref{sec:approximation}.

\begin{lemma}\label{le:approximation}
  Assume that $\lambda$ is non-atomic. Then for any $g \in L_2(\XX^d)$ and any $\varepsilon > 0$ there exists a function $h \in \mathcal{E}_d$, such that $\|g-h\|_2 < \varepsilon$. If $g \in L_2^s(\XX^d)$, then one can find $h\in \mathcal{E}^s_d$ with this property.

Moreover, if $g \in L_2^s(\XX^d)$ and for all $I \subseteq [d]$ and $\mathcal{J} \in \PP_{I}$, the quantity $\|\|g\|_{\mathcal{J}}\|_\infty$ is finite, then there exists $h \in \mathcal{E}^s_d$ such that $\|g-h\|_2 < \varepsilon$ and for all $I \subseteq [d]$ and $\mathcal{J} \in \PP_{I}$, $\|\|h\|_\mathcal{J}\|_\infty \le 3 \|\|g\|_{\mathcal{J}}\|_\infty$.
\end{lemma}

\subsection{Maximal inequality}

The goal of this section is to establish the following lemma.

\begin{lemma}\label{le:maximal-inequality}
For any positive integer $d$ there exists a constant $C_d$ such that for all $g \in L_2^s(\XX^d)$, $T \ge 0$ and $u > 0$,
\begin{displaymath}
  \p\Big(\sup_{t \le T} |I^{(d)}_t(g)| \ge u\Big) \le C_d \p(|I^{(d)}_T(g)| \ge u/C_d).
\end{displaymath}
\end{lemma}

\begin{proof}
By the right-continuity of trajectories of the process $I^{(d)}_t(g)$ it is enough to show that for every $u > 0$,
\begin{align}\label{eq:maximal-inequality-proof}
\liminf_{M\to \infty} \p\Big(\max_{\ell=1,\ldots,M} |I^{(d)}_{\ell T/M}(g)| > u\Big) \le C_d \p(|I^{(d)}_T(g)| \ge u/C_d).
\end{align}

 Let $g_n \in \mathcal{E}^s_d$ be a sequence of step functions converging to $g$ in $L_2^s(\XX^d)$, given by Lemma \ref{le:approximation}. Then for any $t$, $I^{(d)}_t(g_n)$ converges in $L_2$ to $I^{(d)}_t(g)$.

If
\begin{displaymath}
  g_n = \sum_{i_1,\ldots,i_d = 1}^{R} a_{i_1,\ldots,i_d} \Ind{\prod_{j=1}^d A_{i_j}},
\end{displaymath}
where $R$, $a_{i_1,\ldots,i_d}$, $A_i$ satisfy the condition (i)--(iv) listed in the definition of $\mathcal{E}_d^s$ (below eq. \eqref{eq:step-function}), then
\begin{displaymath}
  I^{(d)}_t(g_n) = \sum_{i_1,\ldots,i_d = 1}^{R} a_{i_1,\ldots,i_d} \prod_{j=1}^d \Big(\eeta(A_{i_j}\times [0,t]) - t\lambda(A_{i_j})\Big).
\end{displaymath}
Note that the $R, a_{i_1,\ldots,i_d}$ and $A_i$ depend on $n$, but since we will temporarily work with fixed $n$, we suppress this in the notation.

Our goal is to represent each $I^{(d)}_{\ell T/M}(g_n)$ as a limit of $U$-statistics in independent random variables. To this end, for fixed $M$ let us introduce further discretization and for a positive integer $N$, write
\begin{displaymath}
\eeta(A_i \times [0,\ell T/M]) - \frac{\ell T}{M}\lambda(A_i) = \sum_{k=1}^{\ell N} X_{N,k,i},
\end{displaymath}
where for $i \le R$ and $k \le MN$, $X_{N,k,i} = \eeta(A_i \times (\frac{(k-1)T}{M N},\frac{kT}{M N}]) - \frac{T}{M N}\lambda(A_i)$. (We can ignore contribution from the sets $A_i\times\{0\}$, since almost surely, $\widehat{\eta}(A_i \times \{0\}) = 0$). Since $M$ is fixed, it is not reflected in the notation.

Note that for fixed $N$, the variables $X_{N,k,i}$, $i \le R, k \le M N$, are independent, centered, and $\Var(X_{N,k,i}) = \lambda(A_i)\frac{T}{MN}$.

We have for every $N$,
\begin{displaymath}
I^{(d)}_{\ell T/M}(g_n) = \sum_{k_1,\ldots,k_d = 1}^{\ell N} \sum_{i_1,\ldots,i_d = 1}^R a_{i_1,\ldots,i_d} \prod_{j=1}^d X_{N,k_j,i_j}.
\end{displaymath}
For $\ell =1,\ldots,M$ let us define $\mathcal{K}_{\ell,N} = \{(k_1,\ldots,k_d) \in [\ell N]^d\colon \exists_{1 \le j_1 < j_2 \le d} \; k_{j_1} = k_{j_2}\}$. We have
\begin{multline*}
  \E \Big(\sum_{(k_1,\ldots,k_d)\in \mathcal{K}_{\ell,N}} \sum_{i_1,\ldots,i_d = 1}^R a_{i_1,\ldots,i_d} \prod_{j=1}^d X_{N,k_j,i_j}\Big)^2 \\
= \sum_{(k_1,\ldots,k_d)\in \mathcal{K}_{\ell,N}} \sum_{(k_1',\ldots,k_d')\in \mathcal{K}_{\ell,N}}  \sum_{i_1,\ldots,i_d = 1}^R \sum_{i_1',\ldots,i_d' = 1}^R
a_{i_1,\ldots,i_d}a_{i_1',\ldots,i_d'} \E \prod_{j=1}^d X_{N,k_j,i_j}\prod_{j=1}^d X_{N,k_j',i_j'}.
\end{multline*}

Since $a_{i_1,\ldots,i_d} = 0$ whenever some of the indices $i_1,\ldots,i_d$ coincide, we can restrict the summation above to $\ii = (i_1,\ldots,i_d)$ and $\ii' = (i_1',\ldots,i_d')$ with pairwise distinct coordinates. If the second of these multi-indices is not a permutation of the first one, then certain variable $X_{N,k,i}$ appears in the product under expectation above only once, and by the independence and mean zero properties, the expectation of the product vanishes. Thus, nonzero contribution to the left-hand side comes only from pairs $\ii$, $\ii'$ such that for some permutation $\sigma$ of $[d]$ and all $j\in[d]$, $i'_{\sigma(j)} = i_j$. In this case, again by independence, the expectation of the product equals $\prod_{j=1}^d \E X_{N,k_j,i_j}X_{N,k_{\sigma(j)},i_j}$ and is nonzero only if $k_{\sigma(j)} = k_j$ for all $j$.
We can thus write
\begin{multline*}
  \E \Big(\sum_{(k_1,\ldots,k_d)\in \mathcal{K}_{\ell,N}} \sum_{i_1,\ldots,i_d = 1}^R a_{i_1,\ldots,i_d} \prod_{j=1}^d X_{N,k_j,i_j}\Big)^2 \\
  \le \max_{i_1,\ldots,i_d \le R} |a_{i_1,\ldots,i_d}|^2 |\mathcal{K}_{\ell,N}| R^d d! \Big(\max_{1\le k \le N, 1\le i \le R}\Var(X_{N,k,i})\Big)^d \\
  \le C(d,R) (MN)^{d-1} \Big(\max_{i_1,\ldots,i_d \le R} |a_{i_1,\ldots,i_d}|^2 \max_{1\le i \le R} \lambda(A_i)^d \Big) \frac{T^d}{M^d N^d} \to 0
\end{multline*}
as $N \to \infty$ (recall that $M$, $n$, $R$ are temporarily fixed).

Therefore, for any $\ell \le M$ we have the convergence in probability
\begin{align}\label{eq:U-stat-representation}
  I^{(d)}_{\ell T/M}(g_n)  = \lim_{N \to \infty} \sum_{1\le k_1 \neq \ldots\neq k_d \le \ell N} F(X_{N,k_1},\ldots,X_{N,k_d}),
\end{align}
where $X_{N,k} = (X_{N,k,1},\ldots,X_{N,k,R})$ and
$F \colon (\R^R)^d \to \R$ is given by
\begin{displaymath}
  F(x_1,\dots,x_d) = \sum_{i_1,\ldots,i_d = 1}^R a_{i_1,\ldots,i_d} \prod_{j=1}^d x_{j,i_j}.
\end{displaymath}

Thus,
\begin{multline}\label{eq:another-discretization}
  \p(\max_{\ell=1,\ldots,M} |I^{(d)}_{\ell T/M}(g_n)| > u) \\
  \le \liminf_{N\to \infty}\p\Big(\max_{\ell =1,\ldots,M} \Big|\sum_{1\le k_1 \neq \ldots\neq k_d \le \ell N} F(X_{N,k_1},\ldots,X_{N,k_d})\Big| > u\Big).
\end{multline}

The quantity $\sum_{1\le k_1 \neq \ldots\neq k_d \le \ell N} F(X_{N,k_1},\ldots,X_{N,k_d})$ is a (classical) $U$-statistic in i.i.d. random variables $X_{N,1},\ldots,X_{N,MN}$.
Let us now temporarily fix also $N$ (in addition to $M$). We will use the tools presented in Appendix \ref{sec:decoupling} to get rid of the $\max$ on the right-hand side of the above inequality. We will first pass to a decoupled
version of the $U$-statistic, then apply a maximal inequality for $U$-statistics to estimate the maxima of partial sums by the sum over all indices, finally we will go back to the original $U$-statistic, using again decoupling inequalities.

Let $(X^{(j)}_{N,k})_{k \in [M N]}$, $j = 1,\ldots,d$, be independent copies of the sequence $(X_{N,k})_{k \in [MN]}$.

Consider the space $E = (\R^M,\|\cdot\|_\infty)$, where for $y = (y_1,\ldots,y_M)$, $\|y\|_\infty = \max_{\ell \le M} |y_\ell|$. For $k_1,\ldots,k_d \in [MN]$ let $G_{k_1,\ldots,k_d} \colon (\R^R)^d \to E$ be given by
\begin{displaymath}
  G_{k_1,\ldots,k_d}(x_1,\ldots,x_d) = \Big(F(x_1,\ldots,x_d)\ind{k_1,\ldots,k_d \le  \ell N}\Big)_{\ell \le M}.
\end{displaymath}
Then
\begin{multline*}
\max_{\ell =1,\ldots,M} \Big|\sum_{1\le k_1 \neq \ldots\neq k_d \le \ell N} F(X_{N,k_1},\ldots,X_{N,k_d})\Big| \\= \Big\|\sum_{1\le k_1 \neq \ldots\neq k_d \le MN} G_{k_1,\ldots,k_d}(X_{N,k_1},\ldots,X_{N,k_d})\Big\|_\infty
\end{multline*}
and similarly
\begin{multline*}
\max_{\ell=1,\ldots,M} \Big|\sum_{1\le k_1 \neq \ldots\neq k_d \le \ell N} F(X^{(1)}_{N,k_1},\ldots,X^{(d)}_{N,k_d})\Big| \\ = \Big\|\sum_{1\le k_1 \neq \ldots\neq k_d \le MN} G_{k_1,\ldots,k_d}(X^{(1)}_{N,k_1},\ldots,X^{(d)}_{N,k_d})\Big\|_\infty.
\end{multline*}
Thus, by the first decoupling inequality of Theorem \ref{thm:decoupling},
\begin{multline}\label{eq:decoupling}
\p\Big(\max_{\ell =1,\ldots,M } \Big|\sum_{1\le k_1 \neq \ldots\neq k_d \le \ell N} F(X_{N,k_1},\ldots,X_{N,k_d})\Big| > u\Big) \\
\le
C_d \p\Big(\max_{\ell =1,\ldots,M } \Big|\sum_{1\le k_1 \neq \ldots\neq k_d \le \ell N} F(X^{(1)}_{N,k_1},\ldots,X^{(d)}_{N,k_d})\Big| > u/C_d\Big).
\end{multline}
A similar second moment estimation as the one above (note that now we have even more independence) shows that for any $\ell \le M$,
\begin{align}\label{eq:off-diagonal-term-vanishes}
  \sum_{(k_1,\ldots,k_d) \in \mathcal{K}_{\ell,N}} F(X^{N,(1)}_{k_1},\ldots,X^{(d)}_{N,k_d})\to 0
\end{align}
in probability. Therefore, for any $u > 0$ we can write
\begin{multline*}
  \p(\max_{\ell=1,\ldots,M} |I^{(d)}_{\ell T/M}(g_n)| > u) \\
  \le C_d\liminf_{N\to \infty}\p\Big(\max_{\ell =1,\ldots,M} \Big|\sum_{k_1, \ldots, k_d = 1}^{\ell N} F(X^{(1)}_{N,k_1},\ldots,X^{(d)}_{N,k_d})\Big| \ge u/(2C_d)\Big)\\
  \le C_d^2 \liminf_{N\to \infty}\p\Big(\Big|\sum_{k_1, \ldots, k_d = 1}^{MN} F(X^{(1)}_{N,k_1},\ldots,X^{(d)}_{N,k_d})\Big| \ge u/(2C_d^2)\Big) \\
  \le C_d^2 \liminf_{N\to \infty}\p\Big(\Big|\sum_{1\le k_1 \neq \ldots\neq k_d \le  MN} F(X^{(1)}_{N,k_1},\ldots,X^{(d)}_{N,k_d})\Big| > u/(4C_d^2)\Big)\\
  \le C_d^3 \liminf_{N\to \infty}\p\Big(\Big|\sum_{1\le k_1 \neq \ldots\neq k_d \le MN} F(X_{N,k_1},\ldots,X_{N,k_d})\Big| > u/(4C_d^3)\Big)\\
  \le C_d^3\p\Big(I^{(d)}_T(g_n) \ge u/(4C_d^3)\Big),
\end{multline*}
where the first inequality follows by \eqref{eq:another-discretization}, \eqref{eq:decoupling} and \eqref{eq:off-diagonal-term-vanishes}; the second one by Lemma \ref{le:MS-maximal-inequality}; the third one again by \eqref{eq:off-diagonal-term-vanishes}; the fourth one by another application of decoupling (third inequality of Theorem \ref{thm:decoupling}, note that thanks to the symmetry of $a_{i_1,\ldots,i_d}$'s, the function $F$ is symmetric), and finally the last one by \eqref{eq:U-stat-representation} (with $\ell=M$).

Now, (still for fixed $M$),
\begin{multline*}
  \p(\max_{\ell=1,\ldots,M} |I^{(d)}_{\ell T/M}(g)| > u) \le \liminf_{n\to \infty} \p(\max_{\ell =1,\ldots,M} |I^{(d)}_{\ell T/M}(g_n)| > u) \\
  \le C_d^3 \liminf_{n\to \infty} \p\Big(I^{(d)}_T(g_n) \ge u/(4C_d^3)\Big) \le C_d^3 \p\Big(I^{(d)}_T(g) \ge u/(4C_d^3)\Big),
\end{multline*}
which after adjustment of constants implies \eqref{eq:maximal-inequality-proof} and ends the proof of the lemma.
\end{proof}

\subsection{Proof of the exponential inequalities}\label{sec:inequalities-proofs}

The proofs of exponential inequalities will rely on their counterparts for generalized $U$-statistics in independent random variables, obtained in \cite{MR1857312,MR2073426} for $U$-statistics of order 2, and in \cite{MR2294982} for general $d$. In order to formulate them we need some additional notation, in particular we need to define the norms $\|\cdot\|_\mathcal{J}$ for arrays of functions.

Consider $d$ independent sequences, $(X^{(j)}_i)_{i\in [n]}$ of independent real random variables with values in $\R$ and an array of functions $(h_{i_1,\ldots,i_d})_{i_1,\ldots,i_d \in [n]}$, where $h_{i_1,\ldots,i_d} \colon \R^d \to \R$ for each $i_1,\ldots,i_d$. A generalized decoupled $U$-statistic is a random variable of the form
\begin{displaymath}
  Z = \sum_{i_1,\ldots,i_d = 1}^n h_{i_1,\ldots,i_d}(X^{(1)}_{i_1},\ldots,X^{(d)}_{i_d}).
\end{displaymath}

To shorten the notation and to make it more transparent, we will apply to indices a similar convention as for arguments $x$ of functions, namely for $I\subseteq [d]$ and $\ii = (i_1,\ldots,i_d) \in [n]^d$ we will write $\ii_I = (i_j)_{j\in I}$. We will also use $\ii_I$ as a stand-alone index, writing for instance $\sum_{\ii_I \in [n]^I}$. Thus, e.g., $\sum_{\ii_{\{2,3\}} \in [n]^2} = \sum_{i_2,i_3=1}^n$. We interpret $\ii_\emptyset$ as the empty index being the unique element of $[n]^\emptyset$. We will write $\E_I$ to denote integration with respect just to variables $(X^{(j)}_i)_{j\in I, i\in [n]}$ (i.e., the conditional expectation with respect to $(X^{(j)}_i)_{j\in I^c, i\in [n]}$).

We will now provide the definition of the norms $\|\cdot\|_\mathcal{J}$ for $U$-statistics. Due to involved notation, the readers may consider examining first the examples below the general definition.
For $I \subseteq [d]$, a partition $\mathcal{J} = \{J_1,\ldots,J_k\} \in \mathcal{P}_{I}$ and a fixed value of $\ii_{I^c}$ we define
\begin{align*}
  &\|(h_\ii)_{\ii_I \in [n]^I}\|_{\mathcal{J}} \\
  &= \sup\Big\{  \sum_{\ii_I \in [n]^I} \E_I h_\ii(X^{(1)}_{i_1},\ldots,X^{(d)}_{i_d})\prod_{j=1}^k \varphi^{(j)}_{\ii_{J_j}}((X^{(\ell)}_{i_\ell})_{\ell\in J_j})\colon\\
  &\phantom{aaaaaaaaaaaaaaaaaaaaaaaaaaaaaaaaa} \forall_{j\in [k]}\;\E \sum_{i_{J_j}\in [n]^{J_j}} (\varphi^{(j)}_{\ii_{J_j}}((X^{(\ell)}_{i_\ell})_{\ell\in J_j}))^2 \le 1\Big\},
\end{align*}
with the supremum taken over all Borel measurable functions $\varphi^{(j)}_{\ii_{J_j}}\colon \R^{J_j} \to \R$,  $j\in [k]$, $\ii_{J_j} \in [n]^{J_j}$, satisfying the integrability condition stated above.

In what follows the meaning of $n$ will be usually clear from the context, so we will write simply $\|(h_\ii)_{\ii_I}\|_{\mathcal{J}}$ rather than $\|(h_\ii)_{\ii_I \in [n]^I}\|_{\mathcal{J}}$. We remark that philosophically the norms $\|(h_\ii)_{\ii_I}\|_{\mathcal{J}}$ are of the same nature as the norms $\|g\|_\mathcal{J}$ of functions on $\XX^d$, i.e., they are injective tensor product norms of $(h_{\ii})_{\ii_I}$ considered as element of the tensor product of Hilbert spaces
$\bigotimes_{j=1}^k (\bigoplus_{\ii_{J_j} \in [n]^{J_j}}L_2((X^{(\ell)}_{i_\ell})_{\ell \in J_j}))\simeq \bigotimes_{j=1}^k (\bigotimes_{\ell \in {J_j}} (\bigoplus_{i=1}^n L_2(X^{(\ell)}_{i})))$, i.e., norms of the array $(h_{\ii})_{\ii_I}$ of functions seen as a $k$-linear functional on $\prod_{j=1}^k (\bigoplus_{\ii_{J_j} \in [n]^{J_j}}L_2((X^{(\ell)}_{i_\ell})_{\ell \in J_j}))$.

Note that for $I = [d]$, $\|(h_\ii)_{\ii_I}\|_{\mathcal{J}}$ is a number, while for $I \subsetneq [d]$ it is a random variable, measurable with respect to $(X^{(\ell)}_{i_\ell})_{\ell \in I^c}$. For instance (using some obvious notational shortcuts, similar as in the case of $\|g\|_{\mathcal{J}}$ in Section \ref{sec:inequalities}),
\begin{align*}
\|(h_{i_1,i_2,i_3})_{i_2,i_3}\|_{\{2,3\}}  = &\sup\Big\{\sum_{i_2,i_3=1}^n \E_{2,3} h_{i_1,i_2,i_3}(X^{(1)}_{i_1},X^{(2)}_{i_2},X^{(3)}_{i_3}) \varphi_{i_2,i_3} (X^{(2)}_{i_2},X^{(3)}_{i_3})\colon\\
  & \phantom{aaaaaa}\sum_{i_2,i_3 =1}^n \E \varphi_{i_2,i_3}(X^{(2)}_{i_2},X^{(3)}_{i_3})^2 \le 1\Big\}\\
 = & \sqrt{\sum_{i_2,i_3=1}^n \E_{2,3} h_{i_1,i_2,i_3}(X^{(1)}_{i_1},X^{(2)}_{i_2},X^{(3)}_{i_3})^2}.
\end{align*}
(here $i_1$ is fixed and we obtain a function of $X^{(1)}_{i_1}$)
and

\begin{align*}
&\|(h_{i_1,i_2,i_3})_{i_1,i_2,i_3}\|_{\{2\}\{1,3\}}\\  & =\sup\Big\{\sum_{i_1,i_2,i_3=1}^n \E h_{i_1,i_2,i_3}(X^{(1)}_{i_1},X^{(2)}_{i_2},X^{(3)}_{i_3}) \varphi^{(1)}_{i_2}(X^{(2)}_{i_2}) \varphi^{(2)}_{i_1,i_3}(X^{(1)}_{i_1},X^{(3)}_{i_3}) \colon \\
 & \phantom{aaaaaaaaaaaaaaa} \sum_{i_2=1}^n \E \varphi^{(1)}_{i_2}(X^{(2)}_{i_2})^2 \le 1, \sum_{i_1,i_3=1}^n \E \varphi^{(2)}_{i_1,i_3}(X^{(1)}_{i_1},X^{(3)}_{i_3})^2 \le 1\Big\}
\end{align*}
(here all the random variables are integrated out and we obtain a number).

We are going to use the following result, which is \cite[Theorem 6]{MR2294982} (note a typo in the original formulation, namely an exponent $p$ is missing, on the right-hand side of the inequality there should be $K_d^p$ instead of $K_d$, and this is how the theorem is applied in the subsequent part of \cite{MR2294982} and also in \cite{MR2408582}; in our notation $K_d$ corresponds to the constant $C_d$ below).

\begin{theorem}\label{thm:U-stat-decoupled} In the setting described above, for all $p \ge 2$,
\begin{displaymath}
  \E \Big|\sum_{\ii \in [n]^d} h_{i_1,\ldots,i_d}(X^{(1)}_{i_1},\ldots,X^{(d)}_{i_d})\Big|^p \le C_d^p \sum_{I \subseteq [d]} \sum_{\mathcal{J} \in \mathcal{P}_I} p^{p|I^c| + p|\mathcal{J}|/2}\E_{I^c}\max_{\ii_{I^c} \in [n]^{I^c}}   \|(h_\ii)_{\ii_{I}}\|_\mathcal{J}^p.
\end{displaymath}

\end{theorem}

To prove Theorem \ref{thm:tails-and-moments-main} we will need the following lemma.

\begin{lemma}\label{le:special functions}
Let $\eeta^{(j)}$, $j=1,\ldots,d$, be i.i.d. copies of the Poisson process $\eeta$. Consider a function $g\colon \XX^d\to \R$ of the form
\begin{displaymath}
  g= \sum_{i_1,\ldots,i_d=1}^{n} b_{i_1,\ldots,i_d} \Ind{\prod_{j=1}^d B_{i_j}},
\end{displaymath}
where $B_1,\ldots,B_n\in \mathcal{X}$ are pairwise disjoint, $0< \lambda(B_i) < \infty$ and $b_{i_1,\ldots,i_d} \in \R$.
For $j \in [d], i \in [n]$ define $X^{(j)}_i = \eeta^{(j)}(B_i\times [0,T]) - T\lambda(B_i)$, and let
\begin{displaymath}
Z = \sum_{i_1,\ldots,i_d = 1}^{n}b_{i_1,\ldots,i_d}\prod_{j=1}^dX^{(j)}_{i_j}.
\end{displaymath}

Then, for any $p\ge 2$,
\begin{align}\label{eq:Z-estimate}
  \E|Z|^p \le C_d^p \sum_{I \subseteq [d]} T^{p|I|/2}\Big(\sum_{\mathcal{J} \in \mathcal{P}_I} p^{p|I^c| + p|\mathcal{J}|/2}\|\|g\|_\mathcal{J}\|_{\infty}^p\Big) \E_{I^c}\max_{\ii_{I^c}\in [n]^{I^c}}\prod_{j\in I^c} |X^{(j)}_{i_j}|^p.
\end{align}
\end{lemma}

\begin{proof}
We may rewrite $Z$ as
\begin{displaymath}
Z = \sum_{i_1,\ldots,i_d=1}^{n} h_{i_1,\ldots,i_d}(X^{(1)}_{i_1},\ldots,X^{(d)}_{i_d}),
\end{displaymath}
with $h_{i_1,\ldots,i_d}(x_1,\ldots,x_d) = b_{i_1,\ldots,i_d}x_{1}\cdots x_d$.

Applying Theorem \ref{thm:U-stat-decoupled} we thus obtain
\begin{align}\label{eq:application-of-Ustat-bound}
  \E|Z|^p \le C_d^p \sum_{I \subseteq [d]} \sum_{\mathcal{J} \in \mathcal{P}_I} p^{p|I^c| + p|\mathcal{J}|/2}\E_{I^c}\max_{\ii_{I^c} \in [n]^{I^c}}   \|(h_\ii)_{\ii_{I}}\|_\mathcal{J}^p.
\end{align}

Let us fix $I$, $\mathcal{J}=\{J_1,\ldots,J_k\} \in \mathcal{P}_I$ and $\ii_{I^c}$ and examine the random variable $\|(h_\ii)_{\ii_{I}}\|_\mathcal{J}$.
We have
\begin{displaymath}
\|(h_\ii)_{\ii_{I}}\|_\mathcal{J} = \sup\Big\{\sum_{\ii_I \in [n]^I}\E_{I}b_{i_1,\ldots,i_d} X^{(1)}_{i_1}\cdots X^{(d)}_{i_d}\prod_{j=1}^k  \varphi^{(j)}_{\ii_{J_j}}((X^{(\ell)}_{i_\ell})_{\ell\in J_j})\Big\},
\end{displaymath}
where the supremum is taken over Borel functions   $\varphi^{(j)}_{\ii_{J_j}}\colon \R^{J_j} \to \R$,  $j=1,\ldots,k$, $\ii_{J_j} \in [n]^{J_j}$, such that
  $\E \sum_{i_{J_j}\in [n]^{J_j}} (\varphi^{(j)}_{\ii_{J_j}}((X^{(\ell)}_{i_\ell})_{\ell\in J_j}))^2 \le 1$. The expression under the supremum can be written as
\begin{multline*}
 T^{|I|/2} \Big(\prod_{j\in I^c} X^{(j)}_{i_j}\Big) \int_{\XX^{I}} g(x_{I^c},x_I) \prod_{j=1}^k f_j(x_{{J}_j})d\lambda^{\otimes |I|}(x_I),
\end{multline*}
where $x_{I^c}$ is any element of the set $\prod_{j\in I^c} B_{i_j}$ and $f_j\colon \XX^{J_j}\to \R$, $j \in [k]$, are given by
\begin{displaymath}
  f_j(x_{J_j}) = \sum_{\ii_{J_j} \in [n]^{J_j}}\frac{1}{T^{|J_j|/2}\prod_{r \in J_j}\lambda(B_{i_r})} \Big(\E_I \varphi^{(j)}_{\ii_{J_j}}((X^{(\ell)}_{i_\ell})_{\ell\in J_j})\prod_{r\in J_j}X^{(r)}_{i_r}\Big)\Ind{\prod_{r\in J_j}B_{i_r}}(x_{J_j}).
\end{displaymath}
Observe that by the Cauchy--Schwarz inequality, independence of $X^{(r)}_{i_r}$ across  distinct $r$, and the equality $\E (X^{(r)}_{i_r})^2 = T\lambda(B_{i_r})$,
\begin{multline*}
  \int_{\XX^{J_j}} f_j(x_{J_j})^2 d\lambda^{\otimes |J_j|}(x_{J_j}) \\=
  \sum_{\ii_{J_j} \in [n]^{J_j}}\frac{1}{T^{|J_j|}\prod_{r \in J_j}\lambda(B_{i_r})^2} \Big(\E_I \varphi^{(j)}_{\ii_{J_j}}((X^{(\ell)}_{i_\ell})_{\ell\in J_j})\prod_{r\in J_j}X^{(r)}_{i_r}\Big)^2\prod_{r\in J_j}\lambda(B_{i_r})\\
  \le \sum_{\ii_{J_j} \in [n]^{J_j}}\E_I (\varphi^{(j)}_{\ii_{J_j}}((X^{(\ell)}_{i_\ell})_{\ell\in J_j}))^2  \le 1.
\end{multline*}
Thus (using that in the case $I \neq [d]$, $\prod_{j\in I^c} B_{i_J}$ is of positive measure), we obtain
\begin{displaymath}
  \|(h_\ii)_{\ii_{I}}\|_\mathcal{J} \le  T^{|I|/2} \Big|\prod_{j\in I^c} X^{(j)}_{i_j}\Big| \|g(x_{I^c},\cdot)\|_\mathcal{J} \le T^{|I|/2} \Big|\prod_{j\in I^c} X^{(j)}_{i_j}\Big| \|\|g\|_\mathcal{J}\|_{\infty}.
\end{displaymath}
Going back to \eqref{eq:application-of-Ustat-bound}, we obtain \eqref{eq:Z-estimate}
\end{proof}

\begin{proof}[Proof of Theorem \ref{thm:tails-and-moments-main}]

Recall that we assume (without loss of generality) that $\lambda$ is non-atomic.

We may assume that for all $\mathcal{J}$, the norms $\|\|g\|_\mathcal{J}\|_\infty$ are finite. By the second part of Lemma \ref{le:approximation}, there exists a sequence of functions $g_m \in \mathcal{E}^s_d$, converging to $g$ in $L_2(\XX^d)$ and such that for all $I \subseteq [d]$ and $\mathcal{J} \in \mathcal{P}_I$, $\|\|g_m\|_\mathcal{J}\|_{\infty} \le 3\|\|g\|_\mathcal{J}\|_{\infty}$. Thus, for each $m$,
\begin{align*}
 \sum_{i_1,...,i_d=1}^{R_m} a^{(m)}_{i_1,...,i_d} \cdot \Ind{ \prod_{j=1}^d A^{(m)}_{i_j}},
\end{align*}
where
\begin{itemize}
  \item[(i)] $R_m \in \N$,
  \item[(ii)] $A^{(m)}_i$, $i \le R_m$, are pairwise disjoint and $0< \lambda(A^{(m)}_i) < \infty$,
  \item [(iii)]  $a^{(m)}_{i_1,...,i_d}=0$ whenever $i_k = i_j$ for some $k \neq j$,
  \item[(iv)] the coefficients $a^{(m)}_{i_1,...,i_d}$ are symmetric under permutation of indices.
\end{itemize}

By \eqref{eq:step-function-integral} we have
\begin{align}\label{eq:U-stat-representation-1}
  I^{(d)}_T(g_m) = \sum_{i_1,\ldots,i_d = 1}^{R_m} a^{(m)}_{i_1,\ldots,i_d} \prod_{j=1}^d\Big(\eeta(A^{(m)}_{i_j}\times[0,T]) - T\lambda(A^{(m)}_{i_j})\Big).
\end{align}

The $L_2(\XX^d)$-convergence of $g_m$ to $g$ implies that $I^{(d)}_T(g_m)$ converges in $L_2$ to $I^{(d)}_T(g)$. By passing to a subsequence we may assume that we also have the almost sure convergence.

Note that \eqref{eq:U-stat-representation-1} can be written as
\begin{displaymath}
  I^{(d)}_T(g_m) = \sum_{i_1,\ldots,i_d = 1}^{R_m}h^{(m)}_{i_1,\ldots,i_d}(X_{i_1},\ldots,X_{i_d}),
\end{displaymath}
where $X_i = \eeta(A_i\times[0,T]) - T\lambda(A_i)$, $i = 1,\ldots,R_m$ is a sequence of independent random variables and
\begin{displaymath}
  h_{i_1,\ldots,i_d}(x_1,\ldots,x_d) = a^{(m)}_{i_1,\ldots,i_d}x_{i_1}\cdots x_{i_d}.
\end{displaymath}
Thus, by Theorem \ref{thm:decoupling}, if $\eeta^{(j)}$, $j=1,\ldots,d$ are i.i.d. copies of the Poisson process $\eeta$, then for any $p \ge 2$,
\begin{align}\label{eq:U-stat-estimate}
  \E|I^{(d)}_T(g_m)|^p \le C_d^p \E|Z_m|^p,
\end{align}
where
\begin{displaymath}
  Z_m = \sum_{i_1,\ldots,i_d = 1}^{R_m}a^{(m)}_{i_1,\ldots,i_d}\prod_{j=1}^d\Big(\eeta^{(j)} (A^{(m)}_{i_j}\times[0,T]) - T\lambda(A^{(m)}_{i_j})\Big)
\end{displaymath}
(note that we use the fact that the coefficients $a^{(m)}_{i_1,\ldots,i_d}$ vanish outside the generalized diagonals).

We will now proceed similarly as in the proof of Lemma \ref{le:maximal-inequality}, but in the `spatial domain' and not in the `time domain', i.e., we will take advantage of the non-atomicity of the measure $\lambda$ and split the sets $A^{(m)}_i$ into disjoint subsets of small measure. For now we fix $m$, and so we will sometimes suppress the dependence on this parameter in the notation and write, e.g., $A_i$ instead of $A^{(m)}_i$ and $a_{i_1,\ldots,i_d}$ instead of $a^{(m)}_{i_1,\ldots,i_d}$.

Set $A = \bigcup_{i=1}^{R_m} A_i$, $\Lambda = \lambda(A) < \infty$. By the well known property of non-atomic measures, for every $x \in (0,\lambda(A_i))$ there exists measurable $B \subseteq A_i$ such that $\lambda(B) = x$. Thus, for each positive integer $N$, we can split
each of the sets $A_i$ into $N^2$ pairwise disjoint sets of measure at most $\lambda(A_i)/N^2 \le \Lambda/N^2$. Together with additivity of the measures $\eeta$ and $\lambda$, this allows us to represent the function $g_m$ and the generalized $U$-statistic $Z_m$ on the right-hand side of \eqref{eq:U-stat-estimate} as
\begin{displaymath}
  g_m = \sum_{i_1,\ldots,i_d=1}^{N^2R_m} b^{(N)}_{i_1,\ldots,i_d} \Ind{\prod_{j=1}^d B^{(N)}_{i_j}}
\end{displaymath}
and
\begin{displaymath}
Z_m = \sum_{i_1,\ldots,i_d = 1}^{N^2 R_m}b^{(N)}_{i_1,\ldots,i_d}\prod_{j=1}^d\Big(\eeta^{(j)} (B^{(N)}_{i_j}\times[0,T]) - T\lambda(B^{(N)}_{i_j})\Big),
\end{displaymath}
where $B^{(N)}_{i}$, $i\in[N^2R_m]$ are pairwise disjoint and $0< \lambda(B^{(N)}_i) \le \Lambda/N^2$. More precisely, $\{B^{(N)}_i\colon i \in [N^2R_m]\}$ is a refinement of the partition $\{A_i\colon i \in [R_m]\}$ and $b^{(N)}_{i_1,\ldots,i_d} = a_{\ell_1,\ldots,\ell_d}$, where $\ell_j$ is uniquely determined by the condition $B^{(N)}_{i_j} \subseteq A_{\ell_j}$.

Lemma \ref{le:special functions} with $n = N^2R^m$ gives
\begin{align}\label{eq:application-of-Ustat-bound-1}
  \E|Z_m|^p \le C_d^p \sum_{I \subseteq [d]} T^{p|I|/2}\Big(\sum_{\mathcal{J} \in \mathcal{P}_I} p^{p|I^c| + p|\mathcal{J}|/2}\|\|g_m\|_\mathcal{J}\|_{\infty}^p\Big) \E_{I^c}\max_{\ii_{I^c} \in [N^2R_m]^{I^c}}\prod_{j\in I^c} |X^{(N,j)}_{i_j}|^p,
\end{align}
where for $j\in [d]$, $i\in [N^2R_m]$,
\begin{displaymath}
X^{(N,j)}_i = \eeta^{(j)}(B^{(N)}_i\times [0,T]) - T\lambda(B^{(N)}_i).
\end{displaymath}

Since $\lambda(B^{(N)}_i) \le \Lambda/N^2$, for large $N$ we have
\begin{displaymath}
\p(|X^{(N,j)}_i| \ge 2) \le\p(\eeta^{(j)}(B^{(N)}_i\times [0,T]) \ge 2) \le (\lambda(B^{(N)}_i)T)^2 \le \Lambda^2T^2/N^4.
\end{displaymath}
Thus, by the union bound,
\begin{displaymath}
  \p(\exists_{j \in [d], i \in [N^2R_m]} |X^{(N,j)}_i| \ge 2) \le d N^2R_m \Lambda^2T^2/N^4 = dR_m \Lambda^2T^2/N^2.
\end{displaymath}

By the Borel-Cantelli Lemma, $\limsup_{N\to \infty} \max_{\ii_{I^c} \in [N^2R_m]^{I^c}}\prod_{j\in I^c} |X^{(N,j)}_{i_j}|^p \le 2^{pd}$ for all subsets $I \subsetneq [d]$.
Moreover, since
\begin{displaymath}
  \max_{\ii_{I^c} \in [N^2R_m]^{I^c}}\prod_{j\in I^c} |X^{(N,j)}_{i_j}| \le \prod_{j\in I^c} (\eeta^{(j)}(A)+ T\Lambda)
\end{displaymath}
and the right-hand side has all moments finite, by the Fatou lemma we have
\begin{displaymath}
  \limsup_{N\to \infty} \E \max_{\ii_{I^c} \in [N^2R_m]^{I^c}}\prod_{j\in I^c} |X^{(N,j)}_{i_j}|^p \le \E \limsup_{N\to \infty}  \max_{\ii_{I^c} \in [N^2R_m]^{I^c}}\prod_{j\in I^c} |X^{(N,j)}_{i_j}|^p \le 2^{pd}.
\end{displaymath}

Thus, passing with $N$ to infinity in \eqref{eq:application-of-Ustat-bound-1} and adjusting the constants, we obtain
\begin{displaymath}
  \E|Z_m|^p \le C_d^p \sum_{I \subseteq [d]} T^{p|I|/2}\sum_{\mathcal{J} \in \mathcal{P}_I} p^{p|I^c| + p|\mathcal{J}|/2}\|\|g_m\|_\mathcal{J}\|_{\infty}^p
\end{displaymath}
and so, by \eqref{eq:U-stat-estimate} and by the assumption $\|\|g_m\|_\mathcal{J}\|_{\infty} \le 3\|\|g\|_\mathcal{J}\|_{\infty}$,
\begin{displaymath}
  \E|I^{(d)}_T(g_m)|^p  \le C_d^p \sum_{I \subseteq [d]} T^{p|I|/2}\sum_{\mathcal{J} \in \mathcal{P}_I} p^{p|I^c| + p|\mathcal{J}|/2}\|\|g\|_\mathcal{J}\|_{\infty}^p.
\end{displaymath}

Using now the almost sure convergence of $g_m$ to $g$, and again the Fatou lemma, we obtain
\begin{multline*}
  \E|I^{(d)}_T(g)|^p  \le C_d^p \sum_{I \subseteq [d]} T^{p|I|/2}\sum_{\mathcal{J} \in \mathcal{P}_I} p^{p|I^c| + p|\mathcal{J}|/2}\|\|g\|_\mathcal{J}\|_{\infty}^p\\
  \le \widetilde{C}_d^p\sum_{k=0}^d \sum_{\mathcal{J} \in \PP_{\{k+1,\ldots,d\}}} p^{pk + p|\mathcal{J}|/2} T^{p(d-k)/2} \Big\| \|g\|_{\mathcal{J}}\Big\|_\infty^p,
\end{multline*}
where in the last inequality we used the symmetry of $g$ to collect together equal terms corresponding to partitions with the same `shape'.
The rest of the proof is standard and we skip the computations, just indicating the main points of the argument. Our goal is to prove \eqref{eq:moment-main} and \eqref{eq:tail-main}.

First, using Chebyshev's inequality in $L_p$ and optimizing in $p$, one obtains
\begin{displaymath}
  \p\Big(|I_T^{(d)}(g)| \ge u\Big) \le 2\exp\Big(-c_d\min_{0\le k \le d} \min_{\mathcal{J} \in \PP_{\{k+1,\ldots,d\}}} \Big(\frac{u}{T^{(d-k)/2}\|\|g\|_{\mathcal{J}}\|_\infty}\Big)^{\frac{2}{2k+|\mathcal{J}|}}\Big).
\end{displaymath}
Next, using Lemma \ref{le:maximal-inequality} one replaces the left-hand side of the above inequality by the left-hand side of \eqref{eq:tail-main}. This results in replacing the 2 on the right-hand side by $2C_d$, but adjustment of the constant $c_d$ leads to \eqref{eq:tail-main}.

Finally, the inequality \eqref{eq:moment-main} follows from \eqref{eq:tail-main} via integration by parts.
\end{proof}

\begin{proof}[Proof of Corollary \ref{cor:simplified-estimate}]
One uses Theorem \ref{thm:tails-and-moments-main} and observes that
\begin{displaymath}
B_k = \|\|g\|_{\{k+1,\ldots,d\}}\|_\infty \ge \|\|g\|_\mathcal{J}\|_\infty
\end{displaymath}
for all $\mathcal{J} \in \mathcal{P}_{\{k+1,\ldots,d\}}$.
The proof is then concluded by adjusting the constants.
\end{proof}

\begin{proof}[Proof of Corollary \ref{cor:U-stat-tail}]
The corollary follows from the chaos expansion \eqref{eq:U-stat-chaos-expansion} combined with  the triangle inequality, the union bound, Theorem \ref{thm:tails-and-moments-main} applied to each term in the chaos expansion, and adjustment of constants.
\end{proof}
\subsection{Auxiliary polynomial tail estimates}

In this section we will prove some additional tail estimates for stochastic integrals $I^{(d)}_T$ to be applied in situations when the quantities appearing in Theorem \ref{thm:tails-and-moments-main} are infinite. They will be used in the proof of Theorem \ref{thm:LIL-integrals} to perform initial truncations, enabling us to use the exponential inequalities.

For a measurable function $g \colon \XX^d \to \R$ and $I \subseteq [d]$, we denote
\begin{displaymath}
\|g\|_{L_1(\XX^{I^c},L_2(\XX^I))} = \int_{\XX^{I^c}} \sqrt{\int_{\XX^I} g^2(x_1,...,x_d) d\lambda^{\otimes |I|}(x_I)}\ d\lambda^{\otimes |I^c|}(x_{I^c}).
\end{displaymath}

\begin{prop}\label{prop:Czebyszew} Let $g \in L_2(\XX^d)$ and $I \subseteq [d]$. For any $T > 0$ and $u > 0$,
\begin{align}\label{eq:polynomial-estimate}
 \p(\sup_{t \le T}|I_t^{(d)}(g)| > u) \le \frac{C_d}{u} T^{d-|I|/2} \|g\|_{L_1(\XX^{I^c},L_2(\XX^I))}.
 \end{align}
\end{prop}

\begin{remark} Note that in the above theorem we do not assume symmetry of the function $g$.
\end{remark}
\begin{proof}
Thanks to \eqref{eq:integral-non-symmetric-symmetric} and Lemma \ref{le:maximal-inequality} it is enough to prove \eqref{eq:polynomial-estimate} with the left-hand side replaced by $\p(|I^{(d)}_T| > u)$.

Consider an ascending sequence of sets $A_m$ such that $\lambda(A_m) < \infty$ and $\bigcup_m A_m = \XX$. Define $g_m = g\Ind{A_m^d}$. Then
\begin{displaymath}
  \|g_m\|_{L_1(\XX^{I^c},L_2(\XX^I))} \le \|g\|_{L_1(\XX^{I^c},L_2(\XX^I))}
\end{displaymath}
and $g_m \to g$ in $L_2^s(\XX^d)$. In particular, $I^{(d)}_T(g_m) \to I^{(d)}_T(g)$ in $L_2$. Thus, by Fatou's lemma we may assume that $g$ is supported on $A^d$ for some set $A$ such that $\lambda(A) < \infty$.

Note that for $g,h$ of this form, by the Cauchy--Schwarz inequality
\begin{displaymath}
  \|g-h\|_{L_1(\XX^{I^c},L_2(\XX^I))} \le \lambda(A)^{(d-|I|)/2} \|g-h\|_2,
\end{displaymath}
so we may use Lemma \ref{le:approximation} to further reduce the proposition to the case of $g \in \mathcal{E}_d$, of the form \eqref{eq:step-function}, satisfying the conditions (i)--(iii) stated after this equation.

For such functions we can use decoupling, as in the proof of Theorem \ref{thm:tails-and-moments-main}, to obtain
\begin{displaymath}
  \p(|I^{(d)}_T(g)| > u) \le C_d \p(|Z| > u/C_d),
\end{displaymath}
where
\begin{displaymath}
  Z = \sum_{i_1,...,i_d=1}^R a_{i_1,...,i_d} \prod_{j=1}^d \Big(\eeta^{(j)}(A_{i_j}\times[0,T]) - T\lambda(A_{i_j})\Big)
\end{displaymath}
and $\eeta^{(j)}$, $j=1,\ldots,d$ are i.i.d. copies of $\eeta$. (Note that we use here the part of Theorem \ref{thm:decoupling} which does not require symmetry).

Now we can simply use the Chebyshev inequality to obtain (denoting integration with respect to $\{\eta^{(j)}\}_{j\in I}$ by $\E_I$)

\begin{align*}
  C_d &\p(|Z| > u/C_d) \le \frac{C_d^2}{u} \E|Z| = \frac{C_d^2}{u} \E\Big|\sum_{i_1,...,i_d=1}^R a_{i_1,...,i_d} \prod_{j=1}^d \Big(\eeta^{(j)}(A_{i_j}\times[0,T]) - T\lambda(A_{i_j})\Big)\Big|\\
  &= \frac{C_d^2}{u} \E_{I^c}\E_{I}\Big|\sum_{i_1,...,i_d=1}^R a_{i_1,...,i_d} \prod_{j=1}^d \Big(\eeta^{(j)}(A_{i_j}\times[0,T]) - T\lambda(A_{i_j})\Big)\Big|\\
  &\le \frac{C_d^2}{u} \E_{I^c}\sum_{\ii_{I^c} \in [R]^{I^c}}\E_{I}\Big|\sum_{\ii_I \in [R]^I} a_{i_1,...,i_d} \prod_{j=1}^d \Big(\eeta^{(j)}(A_{i_j}\times[0,T]) - T\lambda(A_{i_j})\Big)\Big|\\
  &\le \frac{C_d^2}{u}\E_{I^c}\sum_{\ii_{I^c} \in [R]^{I^c}}  \Big(\E_{I}\Big|\sum_{\ii_I \in [R]^I} a_{i_1,...,i_d} \prod_{j=1}^d \Big(\eeta^{(j)}(A_{i_j}\times[0,T]) - T\lambda(A_{i_j})\Big)\Big|^2\Big)^{1/2}\\
  &= \frac{C_d^2}{u}\sum_{\ii_{I^c} \in [R]^{I^c}}  \Big(\sum_{\ii_I \in [R]^I} a^2_{i_1,...,i_d} T^{|I|} \prod_{j\in I} \lambda(A_{i_j})\Big)^{1/2}\prod_{\ell\in I^c} \E|\eeta(A_{i_\ell}\times[0,T]) - T\lambda(A_{i_\ell})|\\
  &\le \frac{2^{d-|I|}C_d^2T^{d-|I|/2}}{u}\sum_{\ii_{I^c} \in [R]^{I^c}}  \Big(\sum_{\ii_I \in [R]^I} a^2_{i_1,...,i_d} \prod_{j\in I} \lambda(A_{i_j})\Big)^{1/2}\prod_{\ell \in I^c} \lambda(A_{i_\ell}) \\
  & = \frac{2^{d-|I|}C_d^2T^{d-|I|/2}}{u} \|g\|_{L_1(\XX^{I^c},L_2(\XX^I))},
\end{align*}
where the third equality follows from the fact that the coefficients $a_{i_1,\ldots,i_d}$ vanish on generalized diagonals, independence and the identities $\E \eeta(A_i\times[0,T]) = \Var(\eeta(A_i\times[0,T])) = T\lambda(A_i)$, while the last inequality from the estimate
\begin{displaymath}
\E|\eeta(A_i\times[0,T]) - T\lambda(A_i)| = \E|\eeta(A_i\times[0,T]) - \E \eeta(A_i\times[0,T])| \le 2\E \eeta(A_i\times[0,T]) = 2T\lambda(A_i).
\end{displaymath}
The proposition follows now by adjusting the constants.
\end{proof}

\subsection{Proof of the laws of the iterated logarithm}\label{sec:lil-proofs}

In this section we will prove Theorem \ref{thm:LIL-integrals}. As explained in Section \ref{sec:LIL}, Corollary \ref{cor:Ustat-LIL} is its easy consequence. We will need the following lemma.

\begin{lemma}\label{le:series}

For any $\alpha > 0$ there exists a constant $C_\alpha$ such that for every $x > 0$,
\[ \sum_{n=3}^\infty \Big( \frac{2^n}{\log n}\Big)^\alpha  \ind{ 2^n \le x^{\frac{1}{\alpha}}\log n } \le C_\alpha x.\]

\end{lemma}

\begin{proof}
Denote $a_n = \big(\frac{2^n}{\log n }\big)^\alpha$, $n \ge 3$. If $x < \min_n a_n$ there is nothing to prove. Otherwise, consider maximal $m$ such that $a_m \le x$. Since $\frac{a_n}{a_{n+1}} = 2^{-\alpha} \frac{\log(n+1)}{\log n} \to 2^{-\alpha}$ as $n \to \infty$, we can find $N_\alpha$ such that $\frac{a_n}{a_{n+1}} < 2^{-\frac{\alpha}{2}} < 1$ for $n > N_\alpha$. Hence,
\begin{multline*}
\sum_{n=3}^\infty a_n \ind{a_n \le x} \le \sum_{n=3}^{N_\alpha} a_n \ind{a_n \le x} + \sum_{n=N_\alpha+1}^m a_n \\
\le N_\alpha \cdot x + \sum_{n=N_\alpha+1}^m (2^{-\frac{\alpha}{2}})^{m-n} a_m \le N_\alpha \cdot x + \frac{a_m}{1-2^{-\frac{\alpha}{2}}} \le C_\alpha\cdot x.
\end{multline*}
\end{proof}

\begin{proof}[Proof of Theorem \ref{thm:LIL-integrals}]
We recall that by $C_d$ we denote constants which depend only on $d$, which may vary between occurrences. We will first prove the bounded LIL, i.e., that for any $g \in L_2^s(\XX^d)$, with probability one,
\begin{align}\label{eq:bounded-LIL}
    \limsup_{t\to \infty} \frac{|I^{(d)}_t(g)|}{(t\log\log t)^{d/2}} \le C_d \|g\|_2.
\end{align}

By homogeneity we may assume that $\|g\|_{L_2(\XX^d)} = 1$. For any subset $I \subset \{1,...,d\}$ consider the function $h_I \colon \XX^{I^c} \to [0,\infty]$, given by
\begin{displaymath}
h_I^2(x_{I^c}) := \int_{\XX^{I}} g^2(x_1,...,x_d)d\lambda^{\otimes |I|}(x_I).
\end{displaymath}
For $I \subsetneq [d]$,  and $n \ge 2$, consider the sets
\begin{displaymath}
E_{I,n} = \Big\{x\in \XX^d \colon h_I^2(x_{I^c}) \le \Big( \frac{2^n}{\log n}\Big)^{d-|I|}\Big\}
\end{displaymath}
and let
\begin{displaymath}
A_n = \bigcap_{I \subsetneq [d]} E_{I,n}.
\end{displaymath}

Define $\widetilde{g}_n = g\Ind{A_n}$, $\widehat{g}_n = g\Ind{A_n^c}$. Note that by symmetry of $g$, the function $\Ind{A_n}$ is symmetric as well and hence so is $\widetilde{g}_n$.

By the definition of the set $A_n$, for all $k = 1,\ldots,d$,

\begin{multline*}
B_{n,k}^2 := \esssup_{x_1,\ldots,x_k} \int_{\XX^{d-k}} \widetilde{g}_n(x_1,\ldots,x_d)^2d\lambda^{\otimes(d-k)}(x_{k+1},\ldots, x_d)\\
\le \sup_{x_1,\ldots,x_d} \Ind{E_{\{k+1,\ldots,d\},n}}(x_1,\ldots,x_d) h^2_{\{k+1,\ldots,d\}}(x_{1},\ldots,x_k) \le \Big(\frac{2^n}{\log n}\Big)^{k}
\end{multline*}
and that this inequality remains true for $k = 0$, by the assumption $\|g\|_2 = 1$.

Thus, by Corollary \ref{cor:simplified-estimate}, applied with $T= 2^n$, for sufficiently large $C_d$,

\begin{align}\label{eq:LIL-bounded-part}
& \sum_{n=3}^\infty \p\Big( \sup_{t \in (2^{n-1},2^n]} |I^{(d)}_t(\widetilde{g}_n)| > C_d(2^n \log n )^{\frac{d}{2}}\Big) \\ &\le \sum_{n=3}^\infty 2\sum_{k=0}^d \exp\Big( - 2 \Big( \frac{(2^n \log n)^{\frac{d}{2}}}{(2^n)^{\frac{d-k}{2}}(2^n/\log n)^{k/2}} \Big)^{\frac{2}{d+k}} \Big)
 \le 2(d+1)\sum_{n=3}^\infty \frac{1}{n^2} < \infty. \nonumber
\end{align}
\smallskip

Let us now pass to $I^{(d)}_t(\widehat{g}_n)$. Observe that there exist disjoint sets $F_{I,n}$, $I \subsetneq [d]$, such that $F_{I,n} \subseteq E_{I,n}^c$ and $\bigcup_{I \subsetneq [d]} F_{I,n} = A_n^c$.

Since $\widehat{g}_n = \sum_{I \subsetneq [d]} g\Ind{F_{I,n}}$, by the triangle inequality and Proposition \ref{prop:Czebyszew}, applied with $T = 2^n$, we have
\begin{align*}
  &\sum_{n=3}^\infty \p\Big( \sup_{t \in (2^{n-1},2^n]} |I^{(d)}_t(\widehat{g}_n))| > C_d(2^n \log n)^{\frac{d}{2}}\Big)\\
  &\le \sum_{I \subsetneq [d]} \sum_{n=3}^\infty \p\Big( \sup_{t \in (2^{n-1},2^n]} |I^{(d)}_t(g\Ind{F_{I,n}}))| > 2^{-d}C_d(2^n \log n)^{\frac{d}{2}}\Big)\\
  &\le \sum_{I \subsetneq [d]} \sum_{n=3}^\infty \widetilde{C}_d \frac{2^{n(d-|I|/2)}}{(2^n \log n)^{d/2}}\int_{\XX^{I^c}} \sqrt{\int_{\XX^I} g^2(x)\Ind{F_{I,n}}(x)d\lambda^{\otimes|I|}(x_I)}d\lambda^{\otimes(d-|I|)}(x_{I^c})\\
  & = \widetilde{C}_d \sum_{I \subsetneq [d]} \sum_{n=3}^\infty \int_{\XX^{I^c}}\frac{2^{n(d-|I|)/2}}{(\log n)^{d/2}}\ind{h_{I}(x_{I^c}) > (2^n/\log n)^{(d-|I|)/2}}h_I(x_{I^c})d\lambda^{\otimes(d-|I|)}(x_{I^c})\\
  &\le \widetilde{C}_d \sum_{I \subsetneq [d]} \int_{\XX^{I^c}} h_I(x_{I^c}) \Big(\sum_{n=3}^\infty \Big(\frac{2^n}{\log n}\Big)^{(d-|I|)/2}\ind{h_{I}(x_{I^c}) > (2^n/\log n)^{(d-|I|)/2}}\Big)d\lambda^{\otimes(d-|I|)}(x_{I^c}).
\end{align*}

By Lemma \ref{le:series}, for any $x_{I^c}$, the internal series on right-hand side above does not exceed $C_d h_I(x_{I^c})$, therefore, using the fact that for every $I$,
\begin{displaymath}
  \int_{\XX^{I^c}} h_I(x_{I^c})^2 d\lambda^{\otimes(d-|I|)}(x_{I^c}) = \int_{\XX^d} g(x_1,\ldots,x_d)^2 d\lambda^{\otimes d}(x_1,\ldots,x_d) = 1,
\end{displaymath}
we obtain
\begin{align*}
 \sum_{n=3}^\infty \p\Big( \sup_{t \in (2^{n-1},2^n]} |I^{(d)}_t(\widehat{g}_n))| > C_d(2^n \log n)^{\frac{d}{2}}\Big)
& \le C_d \sum_{I \subsetneq [d]} \int_{\XX^{I^c}} h_I(x_{I^c})^2 d\lambda^{\otimes(d-|I|)}(x_{I^c}) \\
& = C_d(2^d-1) < \infty.
\end{align*}

Combining the above estimate with \eqref{eq:LIL-bounded-part} and using the triangle inequality, we obtain that for sufficiently large constant $C_d$,
\begin{multline*}
\sum_{n=3}^\infty \p\Big( \sup_{t \in (2^{n-1},2^n]} |I^{(d)}_t(g)| > 2C_d(2^n \log n)^{\frac{d}{2}}\Big) \\
\le \sum_{n=3}^\infty\p\Big( \sup_{t \in (2^{n-1},2^n]} |I^{(d)}(\widetilde{g}_n)| > C_d(2^n \log n)^{\frac{d}{2}}\Big) \\+ \sum_{n=3}^\infty \p\Big( \sup_{t \in (2^{n-1},2^n]} |I^{(d)}_t(\widehat{g}_n))| > C_d(2^n \log n)^{\frac{d}{2}}\Big)
 < \infty.
\end{multline*}

Thus, by the Borel-Cantelli Lemma, with probability one, for all $n$ sufficiently large
\begin{displaymath}
  \sup_{t \in (2^{n-1},2^n]} \frac{|I^{(d)}_t(g)|}{(2^n\log n)^{d/2}} \le 2C_d.
\end{displaymath}

Since for some $c>0$ and all $n \ge 3$ and $t \in (2^{n-1},2^n]$, $t \log \log t \ge 2^{n-1}\log\log 2^{n-1} \ge c 2^n \log n$, taking into account that $\|g\|_2 = 1$ and adjusting the constants we obtain that \eqref{eq:bounded-LIL} indeed holds with probability one.

Having established the bounded LIL, we may approximate $g \in L_2^s(\XX^d)$ by functions from $\mathcal{E}_d^s$ for which the limiting set can be obtained as a corollary to the Law of the Iterated Logarithm for the Poisson process.

Let us fix $h \in \mathcal{E}_d^s$ of the form \eqref{eq:step-function}. By \eqref{eq:step-function-integral},
\begin{displaymath}
I_t^{(d)}(h) = \sum_{i_1,...,i_d=1}^R a_{i_1,...,i_d} \prod_{j=1}^d X_{i_j}(t),
\end{displaymath}
where $(X_i(t))_{t\ge 1}$ $j=1,\ldots,R$ are independent one-dimensional compensated Poisson processes with parameters $\lambda(A_i)$, $i=1,\ldots,R$, respectively.
By Strassen's Law of the Iterated Logarithm, with probability one, for each $i \in [d]$, $\limsup_{t\to \infty} \frac{|X_i(t)|}{\sqrt{2t\loglog t}} < \infty$ and, moreover, the set of limit points as $t \to \infty$ of the function $\frac{1}{\sqrt{2t\log \log t}}(X_1(t),\ldots,X_d(t))$ is the ellipsoid
\begin{displaymath}
\mathcal{E} := \Big\{ (y_1,....,y_R) \in \R^R : \sum_{i=1}^R \frac{y_i^2}{\lambda(A_i)} \le 1  \Big\}.
\end{displaymath}
By continuity of the function $\R^R \ni (x_1,\ldots,x_R) \mapsto \sum_{i_1,\ldots,i_d=1}^R a_{i_1,\ldots,i_d} \prod_{j=1}^d x_{i_j} \in \R$ we obtain that the set of limit points of $(2t\log \log t)^{-\frac{d}{2}} \cdot I^{(d)}_t(h)$ is the set
\begin{displaymath}
\Big \{ \sum_{i_1,...,i_d=1}^R a_{i_1,...,i_d} \prod_{k=1}^d y_{i_k} :  \sum_{i=1}^R \frac{y_i^2}{\lambda(A_i)} \le 1 \Big\} = \Big\{ \int_{\XX^d} h \varphi^{\otimes d} d\lambda^{\otimes d} : \int_\XX \varphi^2 d\lambda \le 1\Big\} = \mathcal{K}_h
\end{displaymath}
(recall the notation introduced in \eqref{eq:cluster-set-intro}).
Let us now fix $g \in L_2^s(\XX^d)$ and consider a sequence of step functions $g_m \in \mathcal{E}^s_d$, such that $g_m \to g$ in $L_2(\XX^d)$ as $m\to \infty$. By the bounded LIL \eqref{eq:bounded-LIL} and the above discussion of the limit set for functions in $\mathcal{E}_d^s$, it follows that with probability one for all $m$,
\begin{align}\label{eq:bounded-LIL-difference}
  \limsup_{t\to \infty} \frac{|I^{(d)}_t(g - g_m)|}{(t\loglog t)^{d/2}} \le C_d \|g-g_m\|_2
\end{align}
and for each $g_m$, the limit set of $\frac{I^{(d)}_t(g_m)}{(2t\loglog t)^{d/2}}$ at infinity equals to $\mathcal{K}_{g_m}$.

It is not difficult to show that on this event \eqref{eq:compact-LIL-integrals} holds. Indeed, consider any $z \in \mathcal{K}_g$ of the form
\begin{displaymath}
z = \int_{\XX^d} g(x_1,\ldots,x_d) \varphi(x_1)\cdots\varphi(x_d) d\lambda^{\otimes d}(x_1,\ldots,x_d)
\end{displaymath}
for some $\varphi\colon \XX \to \R$ with $\|\varphi\|_2 \le 1$ and let $z_m\in \mathcal{K}_{g_m}$ be defined by the same formula as $z$, but with $g$ replaced by $g_m$. By the Cauchy--Schwarz inequality we have $|z - z_m| \le \|g-g_m\|_2$. Moreover, there exists a sequence $t_m \to \infty$, such that
\begin{displaymath}
\Big|\frac{I^{(d)}_{t_m}(g_m)}{(2t_m\loglog t_m)^{d/2}} - z_m\Big| \le \frac{1}{m}
\end{displaymath}
and
\begin{displaymath}
  \Big|\frac{I^{(d)}_{t_m}(g-g_m)}{(2t_m\loglog t_m)^{d/2}}\Big| \le C_d \|g-g_m\|_2 + \frac{1}{m}.
\end{displaymath}
We thus obtain
\begin{multline*}
  \Big|\frac{I^{(d)}_{t_m}(g)}{(2t_m\loglog t_m)^{d/2}} - z\Big|
  \le \Big|\frac{I^{(d)}_{t_m}(g-g_m)}{(2t_m\loglog t_m)^{d/2}}\Big| +
  \Big|\frac{I^{(d)}_{t_m}(g_m)}{(2t_m\loglog t_m)^{d/2}} - z_m\Big| + |z_m - z|\\
  \le C_d \|g - g_m\|_2 + \frac{2}{m} + \|g - g_m\|_2 \to 0,
\end{multline*}
which shows that $z$ is in the cluster set of $\frac{I^{(d)}_{t}(g)}{(2t\loglog t)^{d/2}}$.

It remains to show that all limiting points are elements of $\mathcal{K}_g$. Consider thus a point $z$ such that \begin{displaymath}
  z = \lim_{n\to \infty} \frac{I^{(d)}_{t_n}(g)}{(2t_n\loglog t_n)^{d/2}}
\end{displaymath}
for some sequence $t_n\to \infty$.
Selecting convergent subsequences of $\frac{I^{(d)}_{t_n}(g_m)}{(2t_n\loglog t_n)^{d/2}}$ for each $m$, and using \eqref{eq:bounded-LIL-difference} together with the structure of the limit set for $g_m$, we obtain that there exists a sequence of functions $\varphi_m \in L_2(\XX)$ such that $\|\varphi_m\|_2 \le 1$ and
\begin{displaymath}
  \int_{\XX^d} g_m\varphi_m^{\otimes d} d\lambda^{\otimes d} \to z.
\end{displaymath}
By the weak compactness of the unit ball of $L_2(\XX)$ we may assume that $\varphi_m$ converge weakly to some $\varphi\in L_2(\XX)$ (to get sequential weak compactness from the Banach-Alaoglu theorem, in the non-separable case one may restrict to the closed span of the set $\{g,g_1,g_2,\ldots\}$). Then $\|\varphi\|_2 \le 1$ and $\varphi_m^{\otimes d}$ converges weakly in $L_2(\XX^d)$ to $\varphi^{\otimes d}$. We thus obtain
\begin{multline*}
  \Big|z - \int_{\XX^d} g\varphi^{\otimes d} d\lambda^{\otimes d}\Big| \le \Big|z - \int_{\XX^d} g_m\varphi_m^{\otimes d} d\lambda^{\otimes d}\Big|
  + \Big|\int_{\XX^d} (g_m - g)\varphi_m^{\otimes d} d\lambda^{\otimes d}\Big|\\ +\Big|\int_{\XX^d} g\varphi_m^{\otimes d} d\lambda^{\otimes d} - \int_{\XX^d} g\varphi^{\otimes d} d\lambda^{\otimes d}\Big|.
\end{multline*}
The first summand on the right-hand side converges to zero by the definition of the functions $\varphi_m$; the second one is, by the Cauchy--Schwarz inequality, bounded by $\|g_m-g\|_2 \|\varphi_m\|_2^d$ which converges to $0$ due to $\|\varphi_m\|_2 \le 1$; finally the last one converges to zero by the weak convergence of $\varphi_m^{\otimes d}$ to $\varphi^{\otimes d}$. Thus, the left-hand side equals to zero, i.e., $z$ indeed belongs to $\mathcal{K}_g$. This ends the proof of the theorem.
\end{proof}

\section{Proofs of results from Section \ref{sec:examples} (Applications)} \label{sec:proofs-applications}

We will now prove those results concerning applications, which are not simple corollaries to our main theorems, but require additional technical arguments, i.e., propositions on subgraph counts, power length functionals and the Ornstein--Uhlenbeck L\'evy process. The remaining applications, i.e., results on Poisson flat processes and marked processes are either direct consequences of results from Section \ref{sec:main-results} or can be obtained quite easily, sketches of arguments were provided in Section \ref{sec:examples} together with their formulations.

\subsection{Proofs of results from Section \ref{sec:subgraph-counts}}\label{sec:proofs-subgraphs}

\begin{proof}[Proof of Proposition \ref{prop:subgraph-norms}]

It is easy to see that for a connected graph $H$ on $d$ vertices we have
\begin{align}\label{eq:clique-comparison}
  g_{K_d,r} \le g_{H,r}^{sym} \le g_{K_d,d r},
\end{align}
where $K_d$ is the clique on $d$ vertices (with the second estimate following from the triangle inequality).
The functions involved are nonnegative, so pointwise estimates imply inequalities between the norms $\|\cdot\|_\mathcal{J}$. From now on, to simplify the notation we will write $g = g_{K_d,r}$.
Note that by integrating out $x_{n+1},\ldots,x_d$ from the formula \eqref{eq:formula-for-g} for $g$, we obtain
\begin{align}\label{eq:from-g-to-g_n}
  g_n(x_1,\ldots,x_n) \le \lambda(B(x_1,r))^{d-n}g_{K_n,r}(x_1,\ldots,x_n)\le B_{r}^{d-n}g_{K_n,r}(x_1,\ldots,x_n).
\end{align}
Again, by integrating out consecutive variables one obtains
\begin{multline*}
  \|g_n\|_{\{[n]\}}^2 =\int_{\XX^n} g_n(x_1,\ldots,x_n)^2 d\lambda^{\otimes n}(x_1,\ldots,x_n) \\
\le \int_{\XX} \lambda(B(x_1,r))^{2d-2n} \lambda(B(x_1,r))^{n-1}d\lambda(x_1) = A_{r,2d-n-1},
\end{multline*}
which when combined with \eqref{eq:clique-comparison} proves the first inequality of \eqref{eq:subgraphs-HS}. The second one is trivial.

Consider now a partition $\mathcal{J} = \{J_1,\ldots,J_m\}$ of $[d]$ with $m \ge 2$. Due to symmetry of $g$, without loss of generality we can assume that $J_1 = \{1,\ldots,\ell\} \neq \emptyset,[d]$.
We have (with convention $x=(x_1,...,x_d)$)
\begin{multline*}
  \|g\|_{\mathcal{J}} \le \|g\|_{\{J_1,J_1^c\}} \\
  =
  \sup\Big\{ \int_{\XX^d}  \varphi(x_1,\ldots,x_\ell)\psi(x_{\ell+1},\ldots,x_d) \prod_{1\le i,j\le d} \ind{\rho(x_i, x_j) \le r}d\lambda^{\otimes d}(x): \|\varphi\|_2,\|\psi\|_2\le 1\Big\}.
\end{multline*}
Using the Cauchy--Schwarz inequality, we obtain
\begin{align*}
&\Big(\int_{\XX^d} \varphi(x_1,\ldots,x_\ell)\psi(x_{\ell+1},\ldots,x_d)\prod_{1\le i,j\le d} \ind{\rho(x_i,x_j) \le r}  d\lambda^{\otimes d}(x)\Big)^2\\
  \le& \int_{\XX^d}\varphi(x_1,\ldots,x_\ell)^2 \prod_{1\le i\le d} \ind{\rho(x_i,x_1) \le r} d\lambda^{\otimes d}(x)  \\
  &\times \int_{\XX^d}\psi(x_{\ell+1},\ldots,x_d)^2 \prod_{1\le i\le d} \ind{\rho(x_i,x_d) \le r}  d\lambda^{\otimes d}(x)   \\
  \le& B_{r}^{d-\ell}\int_{\XX^\ell}\varphi(x_1,\ldots,x_\ell)^2d\lambda^{\otimes \ell}(x_1,\ldots,x_\ell) \\& \times B_r^\ell  \int_{\XX^{d-\ell}}\psi(x_{\ell+1},\ldots,x_d)^2 d\lambda^{\otimes(d-\ell)}(x_{\ell+1},\ldots,x_d) \le B_r^{d}.
\end{align*}
This proves that $\|g\|_\mathcal{J} \le B_r^{d/2}$. Applying this inequality to $K_n$ instead of $K_d$ and using \eqref{eq:from-g-to-g_n} we obtain $\|g_n\|_\mathcal{J} \le B_r^{d-n/2}$, which by \eqref{eq:clique-comparison} implies \eqref{eq:subgraphs-finer-partitions-k=0}.

Fix now $1\le k\le d$. For any $x_{[k]} \in \XX^{[k]}$ the function $g(x_{[k]},\cdot)\colon \XX^{[k]^c}\to \R$ can be bounded from above by
\begin{displaymath}
  h(x_{k+1},\ldots,x_d) = \prod_{j=k+1}^d \ind{\rho(x_i,x_1)\le r_t},
\end{displaymath}
and so for any $\mathcal{J} \in \mathcal{P}_{[k]^c}$,
\begin{displaymath}
  \|g(x_{[k]},\cdot)\|_\mathcal{J} \le \|g(x_{[k]},\cdot)\|_{\{[k]^c\}} \le B_r^{\frac{d-k}{2}},
\end{displaymath}
where we again integrated out consecutive variables over the balls of radius $r$. This shows that $\|\|g\|_\mathcal{J}\|_\infty \le B_r^{\frac{d-k}{2}}$.

When combined with \eqref{eq:clique-comparison}, this proves \eqref{eq:subgraphs-finer-partitions} for $n = d$, the general case again follows from the special one and the estimate \eqref{eq:from-g-to-g_n}.
\end{proof}

\subsection{Proofs of results from Section \ref{sec:length-power}}\label{sec:length-power-proofs}
Propositions \ref{prop:power-length-LIL} and \ref{prop:PowerLengthIneq} will follow from the variance formula \eqref{eq:U-stati-variance} and Corollaries \ref{cor:Ustat-LIL}, \ref{cor:U-stat-tail} respectively, once we estimate $\|g\|_1$ and $\|\|g_n\|_\mathcal{J}\|_\infty$ for $n = 1,2$.

Due to the boundedness of $r$, if $|x-y| \le r(x) + r(y)$ then $|x-y| \le 2\|r\|_\infty$. Thus, by \eqref{eq:A-beta}, $r(y) \le \sup\{ r(x+v) : |v| \le 2\|r\|_\infty\} \le (\beta-1)r(x)$ for every such pair $(x,y)$. Hence,

\begin{multline*}
  \|g\|_1 = \int_{\R^N}\int_{\R^N}|x-y|^{\alpha}\ind{|x-y| \le r(x) + r(y)} \lambda(dx)\lambda(dy)\\
  \le \beta^{\alpha} \int_{\R^N}r^{\alpha}(x) \int_{\R^N}\ind{|x-y| \le \beta r(x)}\lambda(dy)\lambda(dx) = \beta^{\alpha} \int_{\mathbb R^N} r^{\alpha}(x) \lambda(B(x,\beta r(x))) \lambda(dx).
\end{multline*}

The $L_1$ part of the integrability assumption asserts that the integral on the right-hand side is finite.
Similarly, taking into account that $g_2 = g$, we obtain

\begin{multline*}
\|g\|_2^2 = \|g_2\|_2^2 = \|g_2\|_{\{1,2\}}^2 = \int_{\R^N}\int_{\R^N}|x-y|^{2\alpha}\ind{|x-y| \le r(x) + r(y)} \lambda(dx)\lambda(dy) \\
\le \beta^{2\alpha} \int_{\R^N}r^{2\alpha}(x) \int_{\R^N}\ind{|x-y| \le \beta r(x)}\lambda(dy)\lambda(dx) = \beta^{2\alpha} \int_{\mathbb R^N} r^{2\alpha}(x) \lambda(B(x,\beta r(x))) \lambda(dx) \\
= \beta^{2\alpha} A_{2\alpha,1}< \infty,
\end{multline*}
where in the last inequality we used again the assumption that $x \mapsto r(x)^\alpha\lambda(B(x,\beta r(x)))$ is integrable, together with the assumption that the function $r$ is bounded.

Passing to $\|g_2\|_{\{1\}\{2\}}$, for every $\phi,\psi:\R^N \to \R$ with $\|\phi\|_2,\|\psi\|_2 \le 1$, by the Cauchy--Schwarz inequality we have
\begin{multline*} \int_{\R^N} \int_{\R^N} |x-y|^\alpha \ind{\|x-y\| \le r(x) + r(y)} \phi(x)\psi(y) \lambda(dx)\lambda(dy) \\
\le \beta^{\alpha}\int_{\R^N} \int_{\R^N} \phi(x) r^{\frac{\alpha}{2}}(x) \psi(y) r^{\frac{\alpha}{2}}(y) \ind{ |x-y| \le \beta r(x)} \ind{ |x-y| \le \beta r(y)} \lambda(dy) \lambda(dx) \\
\le \beta^{\alpha} \sqrt{\int_{\R^N}  \phi^2(x) r^\alpha(x) \lambda(B(x,\beta r(x)))\lambda(dx)} \sqrt{\int_{\R^N} \psi^2(y) r^{\alpha}(y)\lambda(B(y,\beta r(y))) \lambda(dy)} \\
\le \beta^\alpha \esssup_{x \in \R^N} r^\alpha(x)\lambda(B(x,\beta r(x)))
\end{multline*}
and hence
\begin{displaymath}
\|g_2\|_{\{1\}\{2\}} \le \beta^\alpha \esssup_{x \in \R^N} r^\alpha(x) \lambda(B(x,\beta r(x))) = \beta^\alpha B_{\alpha,1}.
\end{displaymath}

Now,

\begin{displaymath}
\| \|g_2\|_{\{2\}} \|_\infty = \esssup_{x \in \R^N} \| g(x,\cdot)\|_2 \le \esssup_{x \in \R^N} \beta^{\alpha} r^{\alpha}(x) \sqrt{\lambda( B(x,\beta r(x)))}
= \beta^\alpha B_{\alpha,1/2}
\end{displaymath}
and
\begin{displaymath}
\| \|g_2\|_\emptyset \|_\infty = \|g\|_\infty \le \beta^\alpha \esssup_{x \in \R^N} r^\alpha (x) = \beta^{\alpha}B_{\alpha,0}.
\end{displaymath}

It remains to estimate $\|g_1\|_{\{1\}} = \|g_1\|_2$ and $\|\|g_1\|_\emptyset\|_\infty = \|g_1\|_\infty$. Since for $x \in \R^N$,
\begin{displaymath}
g_1(x) = \int_{\R^N} g(x,y)\lambda(dy) \le \beta^\alpha r^{\alpha}(x) \lambda (B(x,\beta r(x))),
\end{displaymath}
we have
\begin{displaymath}
\| \|g_1\|_{\emptyset} \|_\infty \le \beta^\alpha \esssup_{x \in \R^N} r^\alpha(x) \lambda( B(x,\beta r(x))) = \beta^\alpha B_{\alpha,1}
\end{displaymath}
and
\begin{displaymath}
\|g_1\|_{\{1\}}^2 \le \beta^{2\alpha} \int_{R^N} r^{2\alpha}(x) ( \lambda( B(x,\beta r(x))))^2 \lambda(dx) = \beta^{2\alpha} A_{2\alpha,2} < \infty
\end{displaymath}
by the $L_2$ part of the integrability assumption of Proposition \ref{prop:power-length-LIL}. As announced, the proofs of both propositions follow now from the above estimates, \eqref{eq:U-stati-variance} and Corollaries \ref{cor:Ustat-LIL}, \ref{cor:U-stat-tail} (in the case of the tail inequality we also use the trivial observation that $B_{\gamma,p}^q = B_{\gamma q,pq}$).

\subsection{Proof of results from Section \ref{sec:OUL}}\label{sec:OUL-proofs}

To prepare the setting for the application of Theorem \ref{thm:tails-and-moments-main} we will use some calculations from \cite{MR2642882,schulte2023moderate} concerning the chaos decomposition of $Q_T$.
It turns out that
\begin{displaymath}
  Q_T = I^{(1)}(g_1) + I^{(2)}(g_2),
\end{displaymath}
where $g_1(x,u) = u^2f_1(x)$
with
\begin{displaymath}
f_1(x) = \Ind{(-\infty,T]}(x)e^{2\rho x}\Big((1 - e^{-2\rho T})\ind{x \le 0} + (e^{-2\rho x} -e^{-2\rho T})\ind{x > 0}\Big),
\end{displaymath}
and $g_2((x_1,u_1),(x_2,u_2)) = u_1u_2f_2(x_1,x_2)$ with
\begin{align*}
  f_2(x_1,x_2)  = & \Ind{(-\infty,T]^2}(x_1,x_2)e^{\rho(x_1+x_2)}\\
  & \times \Big((1-e^{-2\rho T})\ind{\max(x_1,x_2)\le 0}
  +(e^{-2\rho\max(x_1,x_2)}-e^{-2\rho T})\ind{\max(x_1,x_2) >0}\Big).
\end{align*}
Since $0 \le f_1,f_2 \le 1$, we obtain $\|g_1\|_\infty, \|g_2\|_\infty \le A^2$. Explicit integration shows that
\begin{align}\label{eq:Ornstein-L2-norms}
\|g_1\|_{\{1\}}^2 & = \|g_1\|_2^2 = c_\nu^2 \Big(T +\frac{e^{-2\rho T} - 1}{2\rho}\Big),\nonumber \\
 \|g_2\|_{\{1,2\}}^2 & = \|g_2\|_2^2 = \frac{T}{\rho} + \frac{e^{-2\rho T} - 1}{2\rho^2}.
\end{align}
Thus,
\begin{displaymath}
\Var(Q_T) = \|g_1\|_2^2 + 2\|g\|_2^2 \le (c_\nu^2 +2\rho^{-1})T.
\end{displaymath}

To apply Theorem \ref{thm:tails-and-moments-main} it remains to bound $\|g_2\|_{\{1\}\{2\}}$ and $\| \|g_2\|_{\{1\}}\|_\infty$
Let us start with $\|g_2\|_{\{1\}\{2\}}$, which is just the operator norm of the kernel operator defined by $g_2$. Due to the tensor product structure of $g_2$ and the assumption $\int_\R u^2 d\nu(u) = 1$, it is easy to see that $\|g_2\|_{\{1\}\{2\}} = \|f_2\|_{\{1\}\{2\}}$. Moreover, one can verify (it is in fact a part of the calculations leading to \eqref{eq:Ornstein-L2-norms}) that
\begin{align}\label{eq:Ornstein-integral-on-3-quadrants}
\int_{\R^2\setminus [0,T]^2} f_2(x_1,x_2)^2 dx_1dx_2 \le \frac{1}{\rho^2},
\end{align}
so it remains to bound $\|f_2\Ind{[0,T]^2}\|_{\{1\}\{2\}}$, which by symmetry and the triangle inequality can be bounded from above by 2$\|f_2\Ind{\Delta}\|_{\{1\}\{2\}}$, where $\Delta = \{(x_1,x_2)\in [0,T]^2\colon x_1 \le x_2\}$.

We have $0 \le f_2\Ind{\Delta}(x_1,x_2) \le e^{\rho(x_1-x_2)}$, and so for any $\psi_1,\psi_2\colon \R\to \R$ with $\|\psi_i\|_2 \le 1$,
\begin{multline*}
  \int_{\R}\int_\R f_2(x_1,x_2)\Ind{\Delta}(x_1,x_2)\psi_1(x_1)\psi_2(x_2)dx_1dx_2\\
  \le \Big(\int_0^\infty e^{2 \rho x_1} \Big(\int_{x_1}^\infty  |\psi_2(x_2)| e^{-\rho x_2}dx_2\Big)^2 dx_1\Big)^{1/2}\\
  \le \Big(\int_0^\infty e^{2 \rho x_1} \Big(\int_{x_1}^\infty  |\psi_2(x_2)|^2 e^{-\rho x_2}dx_2\Big) \Big(\int_{x_1}^\infty e^{-\rho x_2}dx_2\Big) dx_1\Big)^{1/2}\\
  = \frac{1}{\sqrt{\rho}} \Big(\int_0^\infty e^{\rho x_1} \int_{x_1}^\infty  |\psi_2(x_2)|^2 e^{-\rho x_2}dx_2 dx_1\Big)^{1/2}\\
  = \frac{1}{\sqrt{\rho}}  \Big(\int_0^\infty |\psi_2(x_2)|^2 e^{-\rho x_2}\int_{0}^{x_2} e^{\rho x_1}  dx_1 dx_2 \Big)^{1/2}
  \le \frac{1}{\rho}\Big(\int_0^\infty |\psi_2(x_2)|^2 dx_2\Big)^{1/2} \le \frac{1}{\rho},
\end{multline*}
where we used twice the Cauchy-Schwarz inequality. Combining this with \eqref{eq:Ornstein-integral-on-3-quadrants}, we thus obtain $\|g_2\|_{\{1\}\{2\}} \le \frac{3}{\rho}$.

As for $\| \|g_2\|_{\{1\}}\|_\infty$, by the Fubini theorem and the assumption $\int_\R u^2d\nu(u) = 1$, we have for $x_1 \le 0$,
\begin{multline*}
\int_\R \int_\R g_2((x_1,u_1),(x_2,u_2))^2 d\lambda(x_2,u_2) \\
= u_1^2\Big(\int_{-\infty}^{0} e^{2\rho (x_1+x_2)} (1-e^{-2\rho T})^2 dx_2 + \int_0^T e^{2\rho(x_1+x_2)}(e^{-2\rho x_2} - e^{-2\rho T})^2dx_2\Big)\\
\le A^2 \frac{1}{\rho}e^{2\rho x_1} \le \frac{A^2}{\rho},
\end{multline*}
while for $0<x_1 < T$, we have
\begin{multline*}
  \int_\R \int_\R g_2((x_1,u_1),(x_2,u_2))^2 d\lambda(x_2,u_2) \\
  = u_1^2\Big(\int_{-\infty}^{x_1}e^{2\rho(x_1+x_2)}(e^{-2\rho x_1} - e^{-2\rho T})^2dx_2 + \int_{x_1}^T e^{2\rho(x_1+x_2)}(e^{-2\rho x_2} - e^{-2\rho T})^2dx_2\Big)\\
   \le A^2\Big(e^{-2\rho x_1}\frac{1}{2\rho}e^{2\rho x_1} + e^{2\rho x_1}\frac{1}{2\rho}e^{-2\rho x_1}\Big) \le \frac{A^2}{\rho}.
\end{multline*}

Thus, $\|\|g_2\|_{\{2\}}\|_\infty \le \frac{A}{\sqrt{\rho}}$ and taking into account previous estimates we may use Theorem \ref{thm:tails-and-moments-main} for $I^{(i)}(g_i)$, combined with the triangle inequality, to obtain
\begin{multline*}
  \p(|Q_T - \E Q_T|\ge u) \le 2\exp\Big(-c\min\Big(\frac{u^2}{\Var(Q_T)}, \frac{u}{\max(\rho^{-1},A^2)},\frac{u^{2/3}}{A^{2/3}\rho^{-1/3}},\frac{u^{1/2}}{A}\Big)\Big),
\end{multline*}
where (we recall) $c_\nu = \int_\R u^4 d\nu(u)$, $A = \sup\{|x|\colon x \in {\rm supp\;}\nu\}$.
Note that for $\frac{u^{1/2}}{A} \ge 1$, we have $\frac{u^{1/2}}{A} \le \frac{u}{\max(\rho^{-1},A^2)}$, while in the complementary case, $2\exp(-\log(2) u^{1/2}/A) \ge 1$. Therefore, by adjusting the constant $c$, one can skip the second argument of the $\min$ in the above inequality.
Ultimately, we thus obtain \eqref{eq:prop-3.6}, which ends the proof of Proposition \ref{prop:OUL}.

\appendix

\section{$U$-statistics in independent random variables}\label{sec:decoupling}

In this section we present basic inequalities for $U$-statistics in independent random variables, used in our proofs.

Let us start with a decoupling inequality due to de la Pe\~{n}a and Montgromery-Smith {\cite[Theorem 1]{MR1261237}}. We refer to the monograph \cite{MR1666908} for a detailed discussion of decoupling.

\begin{theorem}\label{thm:decoupling} Let $d$ be a positive integer and for $n \ge d$ let $(X_i)_{i=1}^n$ be a sequence of independent random variables with values in a measurable space $(S,\mathcal{S})$ and let $(X^{(j)}_i)_{i=1}^n$, $j= 1,\ldots,d$ be $d$ independent copies of this sequence. Let $E$ be a separable Banach space and for each $(i_1,\ldots,i_d) \in [n]^d$ with pairwise distinct coordinates let $h_{i_1,\ldots,i_d} \colon S^d \to E$ be a measurable function. There exists a numerical constant $C_d$, depending only on $d$, such that for all $u > 0$,
\begin{multline*}
\p\Big(\Big\|\sum_{1\le i_1\neq \ldots\neq i_d\le n} h_{i_1,\ldots,i_d} (X_{i_1},\ldots,X_{i_d})\Big\|> u\Big)\\
\le C_d \p\Big(\Big\|\sum_{1\le i_1\neq \ldots\neq i_d\le n} h_{i_1,\ldots,i_d}(X^{(1)}_{i_1},\ldots,X^{(d)}_{i_d})\Big\|> u/C_d\Big).
\end{multline*}
As a consequence, for all $p \ge 1$,
\begin{displaymath}
\Big\|\sum_{1\le i_1\neq \ldots\neq i_d\le n} h_{i_1,\ldots,i_d}  (X_{i_1},\ldots,X_{i_d})\Big\|_p \le C_d' \Big\|\sum_{1\le i_1\neq \ldots\neq i_d\le n} h_{i_1,\ldots,i_d} (X^{(1)}_{i_1},\ldots,X^{(d)}_{i_d})\Big\|_p,
\end{displaymath}
where $C_d'$ is another numerical constant depending only on $d$.

If moreover the functions $h_{i_1,\ldots,i_d}$ are symmetric in the sense that, for all $x_1,\ldots,x_d \in S$ and all permutations $\pi\colon [d]\to [d]$,
$h_{i_1,\ldots,i_d}(x_1,\ldots,x_d) = h_{i_{\pi_1},\ldots,i_{\pi_d}}(x_{\pi_1},\ldots,x_{\pi_d})$,
then for all $u > 0$,
\begin{multline*}
\p\Big(\Big\|\sum_{1\le i_1\neq \ldots\neq i_d\le n} h_{i_1,\ldots,i_d} (X^{(1)}_{i_1},\ldots,X^{(d)}_{i_d})\Big\|> u\Big)\\
\le \widetilde{C}_d \p\Big(\Big\|\sum_{1\le i_1\neq \ldots\neq i_d\le n} h_{i_1,\ldots,i_d} (X_{i_1},\ldots,X_{i_d})\Big\|>u/\widetilde{C}_d\Big),
\end{multline*}
where $\widetilde{C}_d$ is a constant depending only on $d$.
As a consequence, for some numerical constant $\widetilde{C}_d'$, depending only on $d$, and all $p \ge 1$,
\begin{displaymath}
\Big\|\sum_{1\le i_1\neq \ldots\neq i_d\le n} h_{i_1,\ldots,i_d} (X^{(1)}_{i_1},\ldots,X^{(d)}_{i_d})\Big\|_p
\le \widetilde{C}_d'\Big\|\sum_{1\le i_1\neq \ldots\neq i_d\le n} h_{i_1,\ldots,i_d} (X_{i_1},\ldots,X_{i_d})\Big\|_p.
\end{displaymath}

\end{theorem}

The next result is a maximal inequality for decoupled $U$-statistics in i.i.d. random variables. It is a corollary to \cite[Lemma 9]{MR2408582}, which in turn is based on a result due to Montgomery-Smith \cite{MR1321767} (which corresponds to the case $d=1$). Note that the assumption of identical distribution is crucial for this type of result.

\begin{lemma}\label{le:MS-maximal-inequality}
Let $d$ be a positive integer and let $X^{(j)}_i$, $i\in [n], j\in [d]$ be i.i.d. random variables with values in a measurable space $(S,\mathcal{S})$. Consider a measurable function $h\colon S^d \to E$, where $E$ is a separable Banach space. Then for any $u > 0$,
\begin{multline*}
  \p\Big(\max_{\ell=1,\ldots,n} \Big\| \sum_{i_1,\ldots,i_d = 1}^{\ell} h(X^{(1)}_{i_1},\ldots,X^{(d)}_{i_d})\Big\|
  \ge u\Big) \\ \le C^d\p\Big(\Big\| \sum_{i_1,\ldots,i_d = 1}^{\ell} h(X^{(1)}_{i_1},\ldots,X^{(d)}_{i_d})\Big\| \ge u/C^d\Big),
\end{multline*}
where $C$ is a universal constant.
\end{lemma}

\section{Proof of approximation results (Lemma \ref{le:approximation})}\label{sec:approximation}

The first part of Lemma \ref{le:approximation} is standard, the statement concerning symmetric functions can be found in exactly this formulation in \cite[Lemma C.1]{MR4193885}, while the statement concerning $g \in L_2(\XX^d)$ can be inferred from the proof therein; it also appears as a part of a proof, e.g., in \cite[pages 8--9]{MR2200233}. A related result is \cite[Proposition E.16]{MR1474726}, which interprets $L_2^s(\XX^d)$ as the symmetric tensor power of $L_2(\XX)$. We will therefore focus on the second part. Its version for not necessarily symmetric functions on a product of intervals appeared in \cite{MR2294982}. We will need the following special case of Lemma 11 from this article. We note that to adapt it to our setting, in the formulation we exchanged the roles of the sets $I$ and $I^c$ with respect to \cite{MR2294982}.

\begin{lemma}\label{le:subset-extraction} Let $(S,\mathcal{F},\mu)$ be a probability space and let $d$ be a positive integer. There exists a constant $C_d$, such that for every $\varepsilon > 0$ and every measurable $A \in \mathcal{F}^{\otimes d}$ with $\mu^{\otimes d}(A)> 1 -\varepsilon$ there exists a set $B \in \mathcal{F}^{\otimes d}$ such that $B \subseteq A$, $\mu^{\otimes d}(B) \ge 1 - C_d\varepsilon^{1/2^{d-1}}$, and for all $\emptyset \neq I \subseteq [d]$ and $x_{I^c} \in S^{I^c}$, setting $B_{x_{I^c}} = \{x_I \in S^{I}\colon (x_I,x_{I^c}) \in B\}$, we have
\begin{displaymath}
  B_{x_{I^c}} = \emptyset \textrm{ or } \mu^{\otimes |I|} (B_{x_{I^c}}) \ge 1 - C_d \varepsilon^{1/2^{d-1}}.
\end{displaymath}
\end{lemma}

\begin{proof}[Proof of the 2nd part of Lemma \ref{le:approximation}]

Let us assume without loss of generality that $g$ is not the zero function. In particular, all the norms $\|\|g\|_\mathcal{J}\|_\infty$ are nonzero.

By the standing assumption that $\lambda$ is $\sigma$-finite, there exists a set $K \subseteq \XX$, with $\lambda(K) < \infty$, such that $\|g\Ind{\XX\setminus K^d}\|_2 < \varepsilon$. Set $\widetilde{g} = g\Ind{K^d}$. Note that for any $I$, $x_{I^c} \in \XX^{I^c}$ and $\mathcal{J} \in \PP_I$, $\|\widetilde{g}(x_{I^c},\cdot)\|_{\mathcal{J}} \le \|g(x_{I^c},\dot)\|_{\mathcal{J}}$. Thus, without loss of generality we may and will assume that $\lambda(\XX) < \infty$.

By the first part of the lemma there exists a sequence of functions $h_n \in \mathcal{E}_d$, such that $h_n \to g$ $\lambda^{\otimes d}$-a.e. Note that the assumption on boundedness of $\|\|g\|_\mathcal{J}\|_\infty$, specialized to $I = \emptyset$, $\mathcal{J} = \emptyset$, translates into $\|g\|_\infty <\infty$. We may therefore assume that $\|h_n\|_\infty \le \|g\|_\infty$. Thus, by the Fubini and Lebesgue dominated convergence theorems, for any $I \subseteq [d]$, $\mathcal{J} \in \PP_I$ and $x_{I^c}\in \XX^{I^c}$, we have
\begin{multline*}
  \Big| \|h_n(x_{I^c},\cdot)\|_{\mathcal{J}} - \|g(x_{I^c},\cdot)\|_\mathcal{J} \Big| \le \|h_n(x_{I^c},\cdot) - g(x_{I^c},\cdot)\|_\mathcal{J} \\
  \le \Big(\int_{\XX^I} |h(x) - g(x)|^2 d\lambda^{|I|}(x_I)\Big)^{1/2} \to 0, \textrm{ $\lambda^{|I^c|}$-a.e.}
\end{multline*}
In particular, the sets
\begin{displaymath}
  C_n := \Big\{x\in \XX^d\colon \forall_{ I\subseteq [d],\mathcal{J} \in \PP_I} \,\|h_n(x_{I^c},\cdot)\|_\mathcal{J} \le  2 \|\|g\|_\mathcal{J} \|_\infty\Big\}
\end{displaymath}
satisfy $\lambda^{\otimes d}(C_n^c) \to 0$ as $n\to \infty$.

For all $n$, we have
\begin{displaymath}
h_n =  \sum_{i_1,...,i_d=1}^{R_n} a_{i_1,...,i_d}^{(n)} \cdot \Ind{ \prod_{j=1}^d A^{(n)}_{i_j}}
\end{displaymath}
for some $R_n \in \N$, a collection $\mathcal{A}_n = \{A^{(n)}_1,\ldots,A^{(n)}_{R_n}\}$ of pairwise disjoint measurable sets and $a^{(n)}_{i_1,\ldots,i_d}$ such that $a^{(n)}_{i_1,\ldots,i_d} = 0$ whenever some of the indices $i_1,\ldots,i_d$ coincide. We may assume without loss of generality that $\bigcup_{i\le R_n} A^{(n)}_i = \XX$.

Denote by $\mathcal{G}_n$ the $\sigma$-field generated by $\mathcal{A}_n$ and note that the functions $x_{I^c} \mapsto \|h_n(x_{I^c},\cdot)\|_\mathcal{J}$ are $\mathcal{G}_n^{\otimes |I^c|}$-measurable. Thus, $C_n \in \mathcal{G}_n^{\otimes d}$.

By Lemma \ref{le:subset-extraction} applied to the measure $\lambda$ restricted to $\mathcal{G}_n$ and normalized, we can find $B_n \subseteq C_n$ such that
$B_n \in \mathcal{G}_n^{\otimes d}$, $\lambda^{\otimes d}(\XX^d \setminus B_n) \to 0$, and for all $I\subseteq [d]$,
\begin{align}\label{eq:bad-set}
  \sup\Big\{\lambda^{|I|}(\XX^{I} \setminus (B_n)_{x_{I^c}})\colon x_{I^c} \in \XX^{I^c}, (B_n)_{x_{I^c}} \neq \emptyset\Big\} \to 0
\end{align}
as $n\to \infty$.

Setting $\widetilde{h}_n = h_n \Ind{B_n}$ we have
\begin{displaymath}
  \|\widetilde{h}_n - g\|_2 \le \|h_n - g\|_2 + \|h_n \Ind{\XX^d\setminus B_n}\|_2 \le \|h_n - g\|_2 + \|g\|_\infty \Big(\lambda^{\otimes d}(\XX^d\setminus B_n)\Big)^{1/2} \to 0,
\end{displaymath}
i.e., $\widetilde{h}_n \to g$ in $L_2$. Moreover, for any $I \subseteq [d]$, $\mathcal{J} \in \PP_I$ and $x_{I^c} \in \XX^{I^c}$ we have either $\widetilde{h}_n(x_{I^c},\cdot) \equiv 0$ or
\begin{multline*}
  \|\widetilde{h}_n(x_{I^c},\cdot)\|_\mathcal{J} \le \|h_n(x_{I^c},\cdot)\|_\mathcal{J} + \|h_n(x_{I^c},\cdot) \Ind{\XX^{I} \setminus (B_n)_{x_{I^c}}}(\cdot)\|_\mathcal{J}\\
  \le 2\|\|g\|_\mathcal{J}\|_\infty + \|h_n(x_{I^c},\cdot) \Ind{\XX^{I} \setminus (B_n)_{x_{I^c}}}(\cdot)\|_\mathcal{J}\\
  \le 2\|\|g\|_\mathcal{J}\|_\infty + \|g\|_\infty \Big(\lambda^{|I|}(\XX^{I} \setminus (B_n)_{x_{I^c}})\Big)^{1/2} \le 3\|\|g\|_\mathcal{J}\|_\infty
\end{multline*}
for $n$ large enough. In the second inequality we use the observation  that if $(B_n)_{x_{I^c}} \neq \emptyset$ then  by $B_n \subseteq C_n$ we have $\|h_n(x_{I^c},\cdot)\|_\mathcal{J} \le 2\|\|g\|_\mathcal{J}\|_\infty$. The third inequality follows from the estimate $\|\cdot\|_\mathcal{J} \le \|\cdot\|_{L_2(\XX^{I})}$, $\|h_n\|_\infty \le \|g\|_\infty$ and a direct estimate. Finally, the fourth one follows from \eqref{eq:bad-set}.

Thus, for any $\varepsilon > 0$ there exists $n$ such that $\|\widetilde{h}_n - g\|_2 \le \varepsilon$ and $\|\|\widetilde{h}_n\|_\mathcal{J}\|_\infty \le 3\|\|g\|_\mathcal{J}\|_\infty$ for all $I \subseteq [d]$ and $\mathcal{J} \in \PP_I$. For fixed $n$ with this property we will now modify $\widetilde{h}_n$, to a function $h$, satisfying the requirements of the second part of Lemma \ref{le:approximation}.

Observe that since $B_n \in \mathcal{G}_n^{\otimes d}$, $\widetilde{h}_n$ is of the form  \eqref{eq:step-function}, with $R \in \N$ and the sets $A_i$ and coefficients $a_{i_1,\ldots,i_d}$ satisfying (ii) and (iii) except perhaps for the symmetry condition, which is not guaranteed by Lemma \ref{le:subset-extraction}.
Recall the formula \eqref{eq:symmetrization-definition}, defining the symmetrization of the function, and define
\begin{displaymath}
  h = \widetilde{h}_n^{sym} = \frac{1}{d!}\sum_{\sigma \in S_d} \widetilde{h}_n^\sigma = \sum_{i_1,\ldots,i_d=1}^{R} \Big(\frac{1}{d!}\sum_{\sigma \in S_d} a_{i_{\sigma(1)},\ldots,i_{\sigma(d)}}\Big) \Ind{\prod_{j=1}^d A_{i_j}},
\end{displaymath}
where for a permutation $\sigma$ of $[d]$, we define the function $\widetilde{h}_n^\sigma$ by
\begin{displaymath}
\widetilde{h}_n^\sigma(x_1,\ldots,x_d) = \widetilde{h}_n(x_{\sigma(1)},\ldots,x_{\sigma(d)}).
\end{displaymath}
Clearly, $h \in \mathcal{E}_d^s$.

Using the fact that $g$ is symmetric under permutations of arguments, we have
\begin{displaymath}
  \|h - g\|_2 \le \frac{1}{d!}\sum_{\sigma \in S_d} \|\widetilde{h}_n^\sigma - g\|_2 = \frac{1}{d!}\sum_{\sigma \in S_d} \|\widetilde{h}_n - g\|_2 \le \varepsilon.
\end{displaymath}

For any $I \subseteq [d]$, $\mathcal{J} = \{J_1,\ldots,J_k\} \in \PP_I$, $x\in \XX^d$ and $\sigma \in S_d$ define $x^{\sigma} = (x_{\sigma(1)},\ldots,x_{\sigma(d)})$ and $\sigma(\mathcal{J}) = \{\sigma(J_1),\ldots,\sigma(J_k)\} \in \PP_{\sigma(I)}$.
Then,
\begin{displaymath}
\| \widetilde{h}_n^{\sigma}(x_{I^c}, \cdot) \|_{\mathcal{J}} = \| \widetilde{h}_n((x^{\sigma})_{\sigma^{-1}(I^c)},\cdot)\|_{\sigma^{-1}(\mathcal{J})} \le 3\|\|g\|_{\sigma^{-1}(\mathcal{J})}\|_\infty = 3\|\|g\|_\mathcal{J}\|_{\infty},
\end{displaymath}
where in the last equality we used the symmetry of $g$. Thus, another application of the triangle inequality implies that $\|\|h\|_\mathcal{J}\|_\infty \le 3\|\|g\|_\mathcal{J}\|_\infty$, which ends the proof.
\end{proof}

\section{Regularity and reductions to the non-atomic case}\label{sec:reductions}

In this section we will provide some technical details concerning regularity of the process $I_t^{(d)}(g)$ as well as reduction to the case of non-atomic measure $\lambda$.

Let $\eeta$ be a Poisson point process on $\YY = \XX\times [0,\infty)$ with intensity $\lambda\otimes \mathcal{L}$ as described in Section \ref{sec:notation} (recall that $\mathcal{L}$ is the Lebesgue measure on $[0,\infty)$). Assume for now that $\eeta$ is proper, i.e.,
\begin{displaymath}
  \eeta = \sum_{n=1}^\kappa \delta_{X_n}
\end{displaymath}
for some $\N\cup \{\infty\}$-valued random variable $\kappa$ and a sequence of $\XX\times[0,\infty)$ valued random variables $X_n$.

Let $\eeta\,'$ be a marking of $\eeta$ by i.i.d random variables uniform on $[0,1]$, i.e., $\eeta\,'=\sum_{n=1}^{\kappa} \delta_{(X_n,Y_n)}$, where $Y_n$ are i.i.d., uniform on $[0,1]$. By the Marking Theorem (see, e.g., \cite[Theorem 5.6]{MR3791470}), $\eeta\,'$ is a Poisson process on $\YY\times [0,1] \simeq (\XX \times [0,1])\times [0,\infty)$, with intensity measure $\lambda' \otimes \mathcal{L}$, where $\lambda' = \lambda \otimes \mathcal L|_{[0,1]}$.

\smallskip

For a positive integer $k$ consider the canonical projection $P \colon (\YY\times [0,1])^k \to \YY^k$ (to simplify the notation we suppress the dependence on $k$).

For any $A \in \mathcal{Y}^{\otimes k}$, the equation \eqref{eq:factorial-proper} gives
\begin{displaymath}
  \eeta^{(k)}(A) = (\eeta\,')^{(k)}(P^{-1}(A))
\end{displaymath}
and as a consequence, for any $g \in L_1^s(\YY^k)$, we have the pointwise equality
\begin{displaymath}
  \int_{\YY^k} g d \eeta^{(k)} = \int_{(\YY\times[0,1])^k} g\circ P d(\eeta\,')^{(k)}.
\end{displaymath}

Clearly, an analogous equality holds for integration with respect to deterministic measures $(\lambda \otimes \mathcal{L})^{\otimes k}$ and $(\lambda' \otimes \mathcal{L})^{\otimes k}$. Thus, by  \eqref{eq:L_1-integral}, for any $f \in L_1^s(\YY^d)$ we have $I^{(d)}(f,\eeta) = I^{(d)}(f\circ P,\eeta\,')$ and by density this equality extends to all $f \in L_2^s(\YY^d)$. (We remark in passing that an analogous argument may be used to justify the second equality of \eqref{eq:integral-equalities}).
Let now $\eeta\,'_t$ be the restriction of $\eeta\,'$ to $(\XX\times [0,1])\times [0,t]$.
Since for any $t \ge 0$, $I^{(d)}(f,\eeta_t) = I^{(d)}(f\Ind{(\XX\times [0,t])^d},\eeta)$ and an analogous equality holds for $\eeta\,'$ we obtain in particular
that for any $g \in L_2^s(\XX^d)$, we have the almost sure equalities
\begin{align}\label{eq:atomic-non-atomic}
I^{(d)}_t(g) = I^{(d)}(g\circ Q,\eeta\,'_t) \textrm{ and } U_t(g) = U(g\circ Q, \eeta\,'_t),
\end{align}
where $Q$ is the canonical projection from $(\XX \times [0,1])^d$ onto $\XX^d$.

Our goals now are as follows. First we will show that in the non-atomic case and for proper $\eeta$, the process $(I^{(d)}_t(g))_{t\ge 0}$ admits a c\`adl\`ag modification.
Note that the existence of such a modification depends only on finite-dimensional distributions of a process. Moreover, by the construction of the stochastic integral (and the fact that the operation assigning factorial measures to a measure from $\NN$ is measurable, see \cite[Proposition 4.3]{MR3791470}), the finite dimensional distributions of the process $(I_t^{(d)}(g))_{t\ge 0}$ depend only on $\lambda$ and $g$.  Having the existence of a c\`adl\`ag version for the proper process with non-atomic intensity, one can use \eqref{eq:atomic-non-atomic} to conclude that such a modification exists in full generality (i.e., also for processes which are not proper and for measures with atoms).

To describe our next goal, let us recall the definition \eqref{eq:U-stat-chaos-kernels} of the functions $g_n$. We will show that the sets $\mathcal{K}_{g_n}$ given by \eqref{eq:cluster-set-intro} agree with $\mathcal{K}_{\widetilde{g}_n}$ where $\widetilde{g} = g\circ Q$, and that for all $k\le n \le d$ and
 $\mathcal{J} \in P_{\{k+1,\ldots,n\}}$,  $\|\|g_n\|_\mathcal{J}\|_\infty = \|\|\widetilde{g}_n\|_\mathcal{J}\|_\infty$. Since again, validity of both the Law of the Iterated Logarithm and tail inequalities for suprema of stochastic integrals and $U$-statistics for the c\`adl\`ag version of these processes, depend only on their finite dimensional distributions, this will complete the argument reducing the general case of our results to the non-atomic one.  Note that in the case of the LIL we use here the compactness of sets $\mathcal{K}_g$, which for square integrable $g$ is an easy consequence of the weak compactness of the unit ball in $L_2(\XX)$.

Let us thus pass to the question of regularity. To simplify the notation, we will not work with the process $\eeta\,'$ but simply assume that $\lambda$ has no atoms. Thus, we can apply Lemma \ref{le:approximation} and consider a sequence $h_m \in \mathcal{E}_d^s$, of the form
\begin{displaymath}
  h_m(x_1,...,x_d) = \sum_{i_1,...,i_d=1}^{R_m} a_{i_1,...,i_d}^{(m)} \cdot  \Ind{\prod_{j=1}^d A_{i_j}^{(m)}}
\end{displaymath}
converging to $g$ in $L_2^s(\XX^d)$. Here, as in Lemma \ref{le:approximation}, $A_1^{(m)}$ are pairwise disjoint, $\lambda(A_i^{(m)}) < \infty$ and $a_{i_1,\ldots,i_d}^{(m)} = 0$ whenever the indices $i_1,\ldots,i_d$ are not pairwise distinct. Note that $h_m \in L_1^s(\XX^d)$.

Recall from \eqref{eq:integral-disjoint-product} that for pairwise disjoint $C_1,\ldots,C_d \subseteq \XX$, with $\lambda(C_i) < \infty$,
\begin{displaymath}
  I^{(d)}_t(\Ind{C_1\times \cdots \times C_d}) = \prod_{i=1}^d(\eeta(C_i\times[0,t])-t\lambda(C_i)).
\end{displaymath}

The factors in the above product are independent, compensated Poisson processes on $[0,\infty)$ with intensities $\lambda(C_1),\ldots,\lambda(C_d)$. The c\`adl\`ag property follows from the assumption that $\eeta$ is proper, since with probability one for every $t \ge 0$, in $C_i\times [0,t]$ there are only finitely many points of $\eeta$. Using the independence properties of $\eeta$ it is now easy to see that $(I^{(d)}_t(\Ind{C_1\times \cdots \times C_d}))_{t\ge 0}$ is a c\`adl\`ag martingale with respect to the filtration $\mathcal{F}_t = \sigma(\eta_t)$. Thus, the same is true for
\begin{displaymath}
  I^{(d)}_t(h_m) = \sum_{i_1,...,i_d=1}^{R_m} a_{i_1,...,i_d}^{(m)} \cdot \prod_{j=1}^d \Big(\eeta\Big(A_{i_j}^{(m)}\times [0,t]\Big) - t\lambda(A_{i_j}^{(m)})\Big).
\end{displaymath}

For any fixed $T > 0$, by Doob's inequality we have
\begin{displaymath}
  \E \sup_{t\le T} |I^{(d)}_t(h_m) - I^{(d)}_t(h_n)|^2 \le 4 \E |I^{(d)}_T(h_m) - I^{(d)}_T(h_n)|^2 = 4 d!T^{d/2}\|h_n - h_m\|_2.
\end{displaymath}

Thus, using the $L_2$ convergence of $h_m$ and the Borel--Cantelli lemma, one can select a subsequence $m_n$ such that with probability one $I^{(d)}_t(h_{m_n})$ converge uniformly on any compact interval of $[0,\infty)$. As the c\`adl\`ag property is preserved under uniform limit, and since for any fixed $t$, $I^{(d)}_t(h_m)$ converges to $I^{(d)}_t(g)$ in $L_2$, the limiting process is a c\`adl\`ag modification of $(I^{(d)}_t(g))_{t \ge 0}$. We remark that to obtain the existence of such a modification we could also rely on general results concerning modifications of martingales, but since we would still need to show that the process in question is a martingale and, e.g., verify its stochastic right-continuity, the above approach seems more straightforward.

Let us now pass to the question of equivalence of $\mathcal{K}_{g_n}$ and $\mathcal{K}_{\widetilde{g}_n}$ for $\widetilde{g} = g\circ Q$, as well as $\|\|g_n\|_{\mathcal{J}}\|_\infty$ and $\|\|\widetilde{g}_n\|_\mathcal{J}\|_\infty$ (now we go back to the setting introduced at the beginning of this section and do not assume that $\lambda$ is non-atomic). Since $\widetilde{g}_n = g_n\circ Q$, it is enough to consider the case of $g_d = g$. To obtain  the equality of norms it is enough to show that for any partition $\mathcal{J} = \{J_1,\ldots,J_k\} \in \mathcal{P}_{[d]}$,
\begin{multline}\label{eq:equality-of-sets}
\Big\{\int_{\XX^d} g(x_1,\ldots,x_d)\prod_{i=1}^k \varphi_i(x_{J_i})d\lambda^{\otimes d}(x_1,\ldots,x_d)\colon\  \varphi_{i}\colon \XX^{J_i}\to \R, \|\varphi_i\|_2\le 1\Big\}\\
= \Big\{\int_{(\XX\times[0,1])^d} g(x_1,\ldots,x_d)\prod_{i=1}^k \psi_i(x_{J_i},t_{J_i})d(\lambda')^{\otimes d}(x_1,t_1,\ldots,x_d,t_d)\colon \\
\psi_{i}\colon (\XX\times[0,1])^{J_i}\to \R,
\|\psi_i\|_2\le 1\Big\}
\end{multline}
(for $\mathcal{J} \in \mathcal{P}_I$ with $I \subsetneq [d]$ we apply this to $g(x_{I^c},\cdot)$ and $|I|$ instead of $d$).

An analogous statement, for $\mathcal{J}$ being the partition into singletons and an additional restriction that all $\varphi_i$'s are equal and so are all $\psi_i$'s, amounts to the equality of $\mathcal{K}_g$ and $\mathcal{K}_{\widetilde{g}}$. Since the arguments are the same, we will restrict to proving \eqref{eq:equality-of-sets}.

The inclusion $\subseteq$ is clear since one can always take $\psi_i = \varphi_i\circ Q$, which does not depend on $t_i$'s and take into account the fact that $\lambda' = \lambda \otimes \mathcal{L}\vert_{[0,1]}$, which (since $\mathcal{L}([0,1])=1$) gives $\|\psi_i\|_2 = \|\varphi_i\|_2$. In the other direction one simply takes
\begin{displaymath}
\varphi_i(x_{J_i}) = \int_{[0,1]^{J_i}} \psi_i(x_{J_i},t_{J_i})dt_{J_i}
\end{displaymath}
and observes that with such a choice the integrals on the left- and right-hand side of \eqref{eq:equality-of-sets} coincide, while by Jensen's inequality (we again use $\mathcal{L}([0,1]) = 1$), $\|\varphi_i\|_2 \le \|\psi_i\|_2 \le 1$.

Thus, \eqref{eq:equality-of-sets} indeed holds, which ends the reduction step.

\bibliographystyle{amsplain}
\bibliography{Poisson-U-stat}
\end{document}